\documentclass{amsart}
\usepackage{amsmath, amssymb, amsthm, graphicx}

\usepackage{pst-all}
\psset{unit=1mm}
\psset{fillcolor=white}
\psset{dotsep=1.5pt}
\psset{dash=1.5pt 1.5pt}
\psset{linewidth=0.8pt}
\psset{arrowsize=3.5pt}
\psset{doublesep=1pt}
\psset{coilwidth=1.2mm}
\psset{coilheight=2}
\psset{coilaspect=20}
\psset{coilarm=2}
\newpsstyle{eps}{linewidth=0.5pt, doubleline=true}
\newpsstyle{etc}{linestyle=dotted}
\newpsstyle{exist}{linestyle=dashed}
\newpsstyle{thick}{linewidth=1.5pt}
\newpsstyle{thickexist}{linewidth=1.5pt, linestyle=dashed}
\newpsstyle{thin}{linewidth=0.5pt}
\newpsstyle{thinexist}{linewidth=0.5pt, linestyle=dashed}

\theoremstyle{plain}
\newtheorem{proposition}{Proposition}[section]
\newtheorem{corollary}[proposition]{Corollary}
\newtheorem{conjecture}[proposition]{Conjecture}
\newtheorem{lemma}[proposition]{Lemma}
\newtheorem{question}[proposition]{Question}
\newtheorem{theorem}[proposition]{Theorem}

\theoremstyle{definition}
\newtheorem{definition}{Definition}
\newtheorem{example}[proposition]{Example}
\newtheorem{remark}[proposition]{Remark}

\numberwithin{equation}{section}

\def\Reff#1; #2; #3; #4; #5; #6; #7\par{%
\bibitem{#1} #2, {\it #3}, #4 {\bf #5} (#6) #7}

\def\Ref#1; #2; #3; #4\par{%
\bibitem{#1} #2, {\it #3}, #4}

\definecolor{Grey}{cmyk}{0,0,0,.5}%
\definecolor{Light}{cmyk}{.5,0,1,0}%
\definecolor{Medium}{cmyk}{.5,0,1,.3}%
\definecolor{Dark}{cmyk}{.5,0,1,.6}%
\definecolor{Contrast}{rgb}{1,0,0}%

\newcommand\AAA{\pmb{A}}
\newcommand\AAAh{\widehat\AAA}
\newcommand\act{\mathbin{*}}
\newcommand\add[2]{\mathrm{add}_{#1}(#2)}%

\newcommand\Blue[1]{c_{#1}}
\newcommand\Bluee[1]{c'_{#1}}

\newcommand\CC{C}
\newcommand\CbL[1]{\widetilde{C}_{#1}}
\newcommand\CbR[1]{C_{#1}}
\newcommand\cl[1]{\overline{#1}}

\newcommand\COV[1]{\vartriangleright_{#1}}
\newcommand\COVe[1]{\trianglerighteq_{#1}}
\newcommand\COVs[1]{\vartriangleright^\#_{#1}}

\newcommand\dist{\mathrm{dist}}
\newcommand\distp{\mathrm{dist}^{\scriptscriptstyle+}\!}
\newcommand\Div{\mathrm{Div}}
\newcommand\dive{\preccurlyeq\nobreak}
\newcommand\diveR{\mathrel{\widetilde{%
\raisebox{5pt}{\hspace{1pt}}\hspace{-1.5pt}\smash\preccurlyeq}}}
\newcommand\DL{D_\smallL}
\newcommand\DLR{D_{\smallL\,\smallR}}
\newcommand\DR{D_\smallR}

\newcommand\ea{\emptyset}
\newcommand\EE{E}
\newcommand\eps{\varepsilon}
\newcommand\equivp{\equiv^{\scriptscriptstyle+}}
\newcommand\et{\mathord\bullet}

\newcommand\ff{f}
\newcommand\Fp{F^{\scriptscriptstyle+}}
\newcommand\Fs{F_{\hspace{-1.5pt}\raise 1.3pt\hbox{\tiny sym}}^{\scriptscriptstyle+}}

\newcommand\ga[1]{a_{#1}}
\newcommand\gah[2]{\hat{a}_{#1,#2}}
\newcommand\GG{G}
\newcommand\gLex{>^{\scriptscriptstyle\mathrm{Lex}}}
\newcommand\GR[2]{\langle#1\,\vert\, #2\rangle}

\newcommand\hh{h}

\newcommand\ii{i}
\newcommand\II{I}
\newcommand\inv{^{-1}}
\newcommand\Is{I}
\newcounter{ITEM}
\newcommand\ITEM[1]{\setcounter{ITEM}{#1}\leavevmode\hbox{\rm(\roman{ITEM})}}

\newcommand\jj{j}

\newcommand\kk{k}

\newcommand\lab[1]{#1^{\mathtt{\#}}}
\newcommand\leLex{\leqslant^{\scriptscriptstyle\mathrm{Lex}}}
\newcommand\leT{\leqslant_{\mathcal{T}}}
\newcommand\Lg[1]{\vert{#1}\vert}
\newcommand\LGA[1]{\Vert{#1}\Vert_{\!\AAA}}
\newcommand\LGAp[1]{\Vert{#1}\Vert_{\!\AAA}^{\scriptscriptstyle+}}
\newcommand\Lgu[1]{\vert{#1}\vert_1}
\newcommand\lLex{<^{\scriptscriptstyle\mathrm{Lex}}}
\newcommand\lLR{<}

\newcommand\ma{a}
\newcommand\mA{A}
\newcommand\mb{b}
\newcommand\mc{c}
\newcommand\mg{g}
\newcommand\mm{m}
\newcommand\MM{M}
\newcommand\MON[2]{\langle#1\,\vert\, #2\rangle^{\scriptscriptstyle\!+}\!}
\newcommand\ms{s}
\newcommand\mt{t}
\newcommand\MZ[1]{\mu(#1)}

\newcommand\NL{N_\smallL}
\newcommand\NLR{N_{\smallL\,\smallR}}
\newcommand\nn{n}
\newcommand\NN{N}
\newcommand\NNNN{\mathbb{N}}
\newcommand\NR{N_\smallR}

\newcommand\OP{\mathord{^{\wedge}}}
\newcommand\OPP{\mathord{\circ}}

\newcommand\pdots{\,...\,}
\newcommand\Pol[1]{\langle#1\rangle}
\newcommand\pp{p}

\newcommand\qq{q}

\newcommand\resp{{\it resp}}
\newcommand\rev{\curvearrowright}
\newcommand\revL{\mathrel{\raisebox{5pt}{\rotatebox{180}{$\curvearrowleft$}}}}
\newcommand\rr{r}
\newcommand\RR{R}
\newcommand\RRR{\pmb{R}}
\newcommand\RRRh{{\widehat\RRR}}

\newcommand\sh[1]{\mathrm{sh}_{#1}}
\newcommand\smallL{{_{\!\ell\!}}}
\newcommand\smallR{{_{\!r\!}}}
\newcommand\Sol[2]{s(#1, #2)}
\newcommand\SOL[2]{S(#1, #2)}
\newcommand\Sp[1]{\langle#1\rangle}
\newcommand\Sym[1]{\mathfrak{S}_{#1}}

\newcommand\Tam[1]{\mathcal{T}_{#1}}
\newcommand\Tree{\mathcal{B}}
\newcommand\Treelab{\Tree^{\mathtt{\#}}}
\newcommand\trm[1]{{#1}_-}
\newcommand\trp[1]{{#1}_+}
\newcommand\TT{T}

\newcommand\TTm{T_{\vee}}
\newcommand\TTp{T_\infty}

\newcommand\uu{u}

\newcommand\vv{v}

\newcommand\wdots{, ...\hspace{0.2ex},}
\newcommand\wit{\lambda}
\newcommand\ww{w}

\newcommand\xx{x}

\newcommand\yy{y}

\newcommand\zz{z}
\newcommand\ZZ{Z}

\begin{document}

\title{Tamari Lattices and the symmetric Thompson monoid
}

\author{Patrick Dehornoy}

\address{Laboratoire de Math\'ematiques Nicolas Oresme, UMR6139, Universit\'e de Caen, F-14032 Caen, France\\ 
Laboratoire de Math\'ematiques Nicolas Oresme, UMR6139, CNRS, F-14032 Caen, France}
 
\email{patrick.dehornoy@unicaen.fr}


\thanks{Work partially supported by the ANR grant TheoGar ANR-08-BLAN-0269-02}


\maketitle

\begin{abstract}
We investigate the connection between Tamari lattices and the Thompson group~$F$, summarized in the fact that $F$ is a group of fractions for a certain monoid~$\Fs$ whose Cayley graph includes all Tamari lattices. Under this correspondence, the Tamari lattice operations are the counterparts of the least common multiple and greatest common divisor operations in~$\Fs$. As an application, we show that, for every~$\nn$, there exists a length~$\ell$ chain in the $\nn$th Tamari lattice whose endpoints are at distance at most~$12\ell/\nn$.
\end{abstract}


\section{Introduction}

The aim of this text is to show the interest of using monoid techniques to investigate Tamari lattices. More precisely, we shall describe the very close connection existing between Tamari lattices and a certain submonoid~$\Fs$ of Richard Thompson's group~$F$: equipped with the left-divisibility relation, the monoid~$\Fs$ is a lattice that includes all Tamari lattices. Roughly speaking, the principle is to attribute to the edges of the Tamari lattices names that live in the monoid~$\Fs$. By using the subword reversing method, a general technique from the theory of monoids, we then obtain a very simple way of reproving the existence of the lattice operations, computing them, and establishing further properties.

The existence of a connection between Tamari lattices, associativity, and the Thompson group~$F$ has been known for decades and belongs to folklore. What is specific here is the role of the monoid~$\Fs$, which is especially suitable for formalizing the connection. Some of the results already appeared, implicitly in~\cite{Dfg} and explicitly in~\cite{Dhb}. Several new results are established in the current text, in particular the construction of a unique normal form in the monoid~$\Fs$ and the group~$F$ (Subsection~\ref{SS:NormalForm}) and the (surprising) result that the embedding of the monoid~$\Fs$ in the Thompson group~$F$ is not a quasi-isometry (Proposition~\ref{P:NotQuasi}). In the language of binary trees, this implies that, for every constant~$\CC$, there exist chains of length~$\ell$ whose endpoints can be connected by a path of length at most~$\ell/\CC$ (Corollary~\ref{C:Chains}).

Let us mention that a connection between the Tamari lattices and the group~$F$ is described in~\cite{Sun}. However both the objects and the technical methods are disjoint from those developed below. In particular, the approach of~\cite{Sun} does not involve the symmetric monoid~$\Fs$, which is central here, but it uses instead the standard Thompson monoid~$\Fp$, which is not directly connected with the Tamari ordering.

The text is organized as follows. In Section~\ref{S:Framework}, we recall the definition of Tamari lattices and Thompson's group~$F$, and we establish a presentation of~$F$ in terms of some specific, non-standard generators~$\ga\alpha$ indexed by binary addresses. In Section~\ref{S:Lattice}, we investigate the submonoid~$\Fs$ of~$F$ generated by the elements~$\ga\alpha$, we prove that $\Fs$ equipped with divisibility has the structure of a lattice, and we describe the (close) connection between this lattice and Tamari lattices. Then, in Section~\ref{S:Polish}, we use the Polish encoding of trees to construct an algorithm that computes common upper bounds for trees in the Tamari ordering and we deduce a unique normal form for the elements of~$F$ and~$\Fs$. Finally, in Section~\ref{S:Distance}, we gather results about the length of the elements of~$F$ with respect to the generators~$\ga\alpha$ or, equivalently, about the distance in Tamari lattices, with a specific interest on lower bounds.


\section{The framework}
\label{S:Framework}

The aim of this section is to set our notation and basic definitions. In Subsections~\ref{SS:Expressions} and~\ref{SS:Trees}, we briefly recall the definition of the Tamari lattices in terms of parenthesized expressions and of binary trees, whereas  Subsections~\ref{SS:Thompson} and~\ref{SS:Action} contain an introduction to Richard Thompson's group~$F$ and its action by rotation on trees. This leads us naturally to introducing in Subsection~\ref{SS:AAA} a new family of generators of~$F$ indexed by binary addresses, and giving in Subsection~\ref{SS:Pres} a presentation of~$F$ in terms of these generators. 


\subsection{Parenthesized expressions and associativity}
\label{SS:Expressions}

\index{Tamari!lattice}
\index{Tamari!order}
Introduced by Dov Tamari in his 1951 PhD thesis, and appearing in the 1962 article~\cite{Tam}---also see~\cite{FrT} and~\cite{HuT}---the $\nn$th Tamari lattice, here denoted by~$\Tam\nn$, is, for every positive integer~$\nn$, the poset (partially ordered set) obtained by considering all well-formed parenthesized expressions involving $\nn+1$ fixed variables and declaring that an expression~$\EE$ is smaller than another one~$\EE'$, written $\EE \leT \EE'$, if $\EE'$ may be obtained from~$\EE$ by applying the associative law $\xx(\yy\zz) = (\xx\yy)\zz$ in the left-to-right direction. As established in~\cite{Tam}, the poset~$(\Tam\nn, \leT)$ is a lattice, that is, any two elements admit a least upper bound and a greatest lower bound. Moreover, $(\Tam\nn, \leT)$ admits a top element, namely the expression in which all left parentheses are gathered on the left, and a bottom element, namely the expression in with all right parentheses are gathered on the right. 

As associativity does not change the order of variables, we may forget about their names, and use~$\et$ everywhere. So, for instance, there exist five parenthesized expressions involving four variables, namely $\et(\et(\et\et))$, $\et((\et\et)\et)$, $(\et(\et\et))\et$, $(\et\et)(\et\et)$, and $((\et\et)\et)\et$, and we have $\et((\et\et)\et) \leT (\et(\et\et))\et$ in the Tamari order as one goes from the first expression to the second by applying the associativity law with $\xx = \et$, $\yy = \et\et$, and~$\zz = \et$. The Hasse diagrams of the lattices~$\Tam3$ and~$\Tam4$ respectively are the pentagon and the $14$~vertex polyhedron displayed in Figures~\ref{F:Tam3} and~\ref{F:Tam4} below. As is well known, the number of elements of~$\Tam\nn$ is the $\nn$th Catalan number~$\frac1{\nn+1}{{2\nn} \choose \nn}$. 

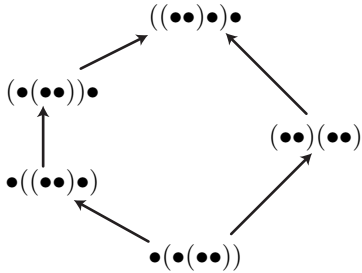
\begin{figure}[htb]
\begin{picture}(47,34)(0,0)
\put(0,0){\includegraphics{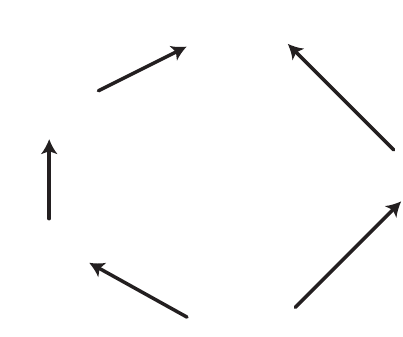}}
\put(19,-1){$\et(\et(\et\et))$}
\put(0,9){$\et((\et\et)\et)$}
\put(0,21){$(\et(\et\et))\et$}
\put(19,31){$((\et\et)\et)\et$}
\put(35,15){$(\et\et)(\et\et)$}
\end{picture}
\caption{\sf\smaller The Tamari lattice~$\Tam3$ made by the five ways of bracketing a four variable parenthesized expression.}
\label{F:Tam3}
\end{figure}

The Tamari lattice~$\Tam\nn$ is connected with a number of usual objects. For instance, its Hasse diagram is the $1$-skeleton---that is, the graph made of the $0$- and $1$-cells---of the $\nn$th Mac Lane--Stasheff associahedron~\cite{Mac, Sta}. Also $\Tam\nn$ embeds in the lattice made by the symmetric group~$\Sym\nn$ equipped with the weak order: $\Tam\nn$ identifies with the sub-poset made by all $312$-avoiding permutations (Bj\"orner \& Wachs \cite{BjW}).

For every~$\nn$, replacing in a parenthesized expression the last (rightmost) symbol~$\et$ with $\et\et$ defines an embedding~$\iota_\nn$ of~$\Tam\nn$ into~$\Tam{\nn+1}$. We denote by~$\Tam\infty$ the limit of the direct system~$(\Tam\nn, \iota_\nn)$ so obtained. Note that $\Tam\infty$ has a bottom element, namely the class of~$\et$, which is also that of~$\et\et$, $\et(\et\et)$, $\et(\et(\et\et))$, etc., but no top element.


\subsection{Trees and rotations}
\label{SS:Trees}

\index{tree}
\index{$\OP$ (tree composition)}
There exists an obvious one-to-one correspondence between parenthesized expressions involving $\nn+1$~variables and size~$\nn$ binary rooted trees, that is, trees with $\nn$~interior nodes and $\nn+1$ leaves, see Figure~\ref{F:Trees}. In this text, we shall use both frameworks equivalently. We denote by~$\TT_0 \OP \TT_1$ the tree whose left-subtree is~$\TT_0$ and whose right-subtree is~$\TT_1$, but skip the symbol~$\OP$ in concrete examples involving~$\et$. 

\begin{figure}[htb]
\begin{picture}(70,12)(0, 0)
\put(0,4){\includegraphics[scale=0.7]{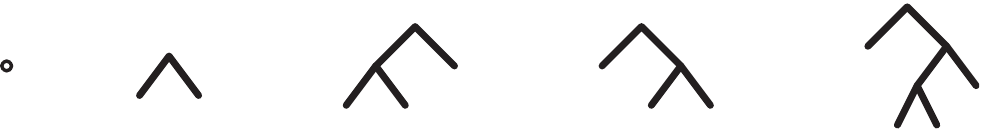}}
\put(0,0){$\et$}
\put(10,0){$\et\et$}
\put(25,0){$(\et\et)\et$}
\put(43,0){$\et(\et\et)$}
\put(60,0){$\et((\et\et)\et)$}
\end{picture}
\caption{\sf\smaller Correspondence between parenthesized expressions and trees.}
\label{F:Trees}
\end{figure}
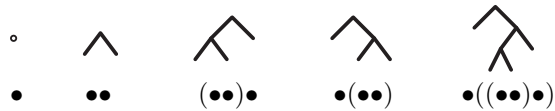

\index{left-rotation}
When translated in terms of trees, the operation of applying associativity in the left-to-right direction corresponds to performing one left-rotation, namely replacing some subtree of the form $\TT_0 \OP (\TT_1 \OP \TT_2)$ with the corresponding tree $(\TT_0 \OP \TT_1) \OP \TT_2$, see Figure~\ref{F:Rotation}. So the Tamari lattice~$\Tam\nn$ is also the poset of size~$\nn$ trees ordered by the transitive closure of left-rotation. We naturally use~$\leT$ for the latter partial ordering.
    
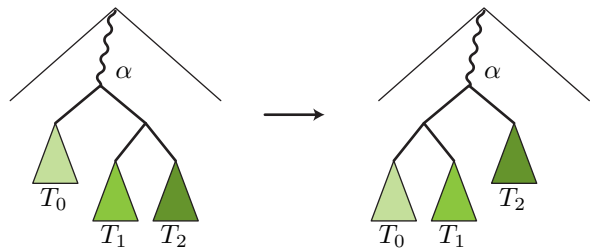
\begin{figure}[htb]
\begin{picture}(75,30)(3,0)
\put(0,0){\includegraphics{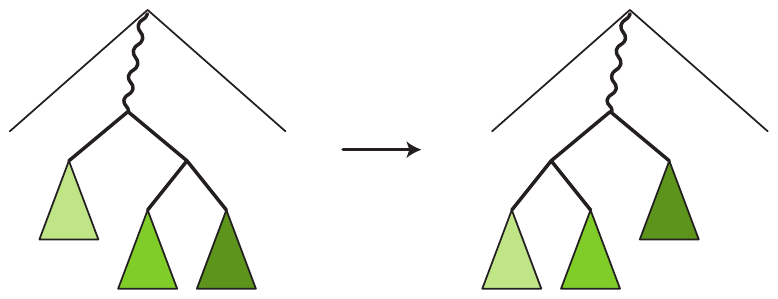}}
\put(5,3){$\TT_0$}
\put(13,-2){$\TT_1$}
\put(21,-2){$\TT_2$}
\put(50,-2){$\TT_0$}
\put(58,-2){$\TT_1$}
\put(66,3){$\TT_2$}
\put(15,20){$\alpha$}
\put(64,20){$\alpha$}
\end{picture}
\caption{\sf\smaller Applying a left rotation in a tree: replacing some subtree of the form $\TT_0 \OP (\TT_1 \OP \TT_2)$ with the corresponding tree $(\TT_0 \OP \TT_1) \OP \TT_2$.}
\label{F:Rotation}
\end{figure}

\index{left-comb}
\index{right-comb}
In terms of trees, the bottom element of the Tamari lattice~$\Tam\nn$ is the size~$\nn$ \emph{right-comb} (or right-vine)~$\CbR\nn$ recursively defined by~$\CbR0 = \et$ and $\CbR\nn = \et \OP \CbR{\nn-1}$ for $\nn \geqslant 1$, whereas the top element is the size~$\nn$ \emph{left-comb} (or left-vine)~$\CbL\nn$ recursively defined by~$\CbL0 = \et$ and $\CbL\nn = \CbL{\nn-1} \OP \et$ for $\nn \geqslant 1$.

\begin{figure}[htb]
\begin{picture}(62,108)(20,0)
\put(0,0){\includegraphics[scale=1]{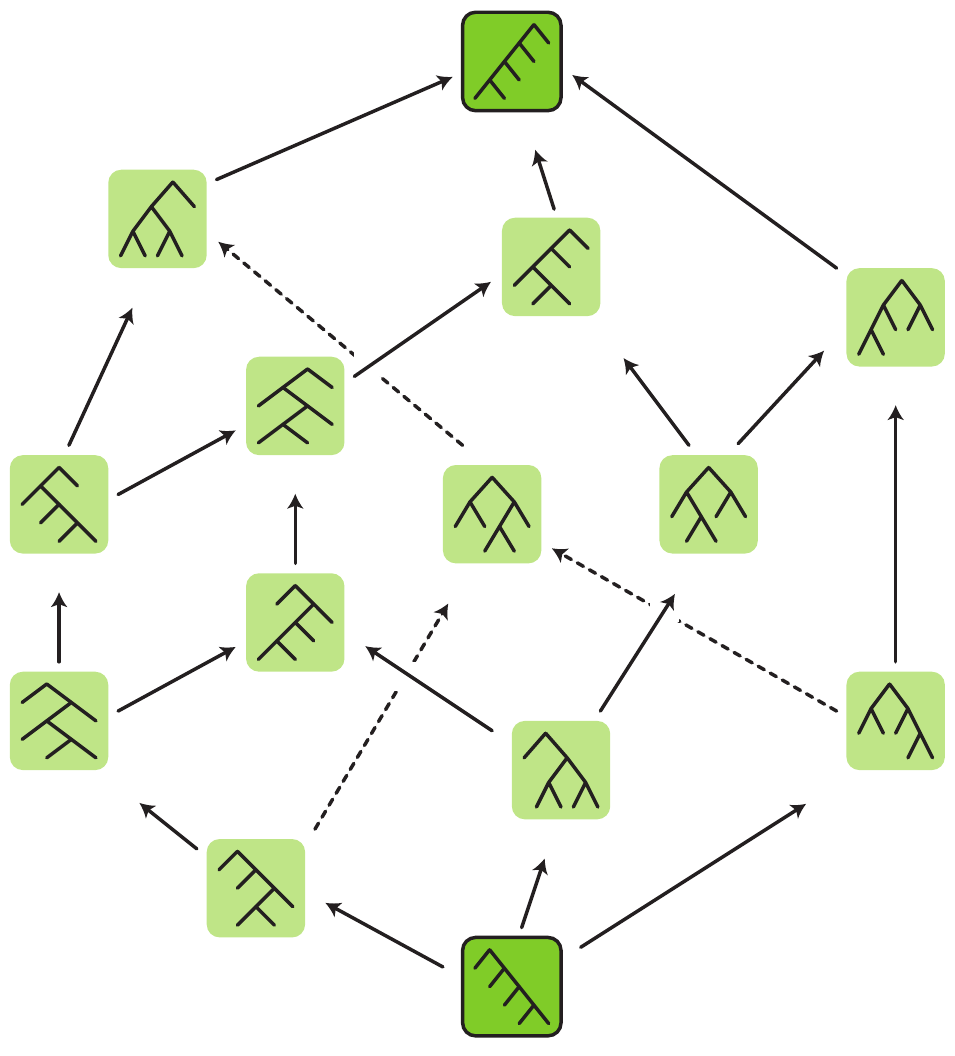}}
\put(44,-1){$\et(\et(\et(\et\et)))$}
\put(48,21){$\et((\et\et)(\et\et))$}
\put(17,9){$\et(\et((\et\et)\et))$}
\put(-3,26){$\et((\et(\et\et))\et)$}
\put(-3,48){$(\et(\et(\et\et)))\et$}
\put(21,36){$\et(((\et\et)\et)\et)$}
\put(21,58){$(\et((\et\et)\et))\et$}
\put(8,77){$((\et\et)(\et\et))\et$}
\put(41,47){$(\et\et)((\et\et)\et)$}
\put(48,72){$((\et(\et\et))\et)\et$}
\put(44,93){$(((\et\et)\et)\et)\et)$}
\put(82,26){$(\et\et)(\et(\et\et))$}
\put(83,67){$((\et\et)\et)(\et\et)$}
\put(64,48){$(\et(\et\et))(\et\et)$}
\end{picture}
\caption{\sf\smaller The Tamari lattice~$\Tam4$, both in terms of parenthesized expressions and binary trees.}
\label{F:Tam4}
\end{figure}
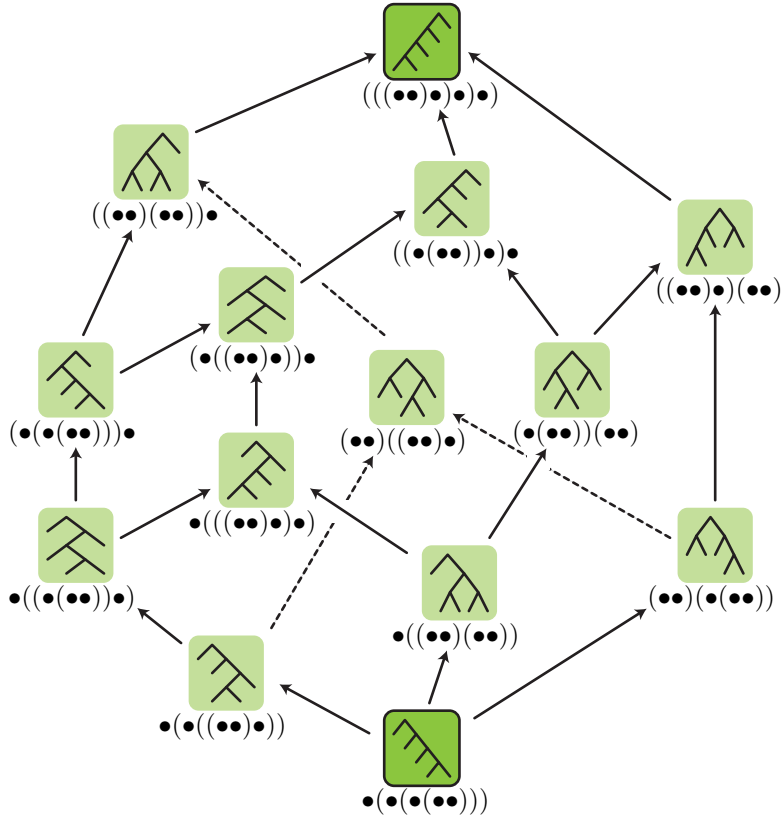


\subsection{Richard Thompson's group $F$}
\label{SS:Thompson}

Introduced by Richard Thompson in~1965, the group~$F$ appeared in print only later, in~\cite{McT} and~\cite{Tho}. The most common approach is to define~$F$ as a group of piecewise linear self-homeomorphisms of the unit interval~$[0,1]$.

\index{Thompson!group}
\index{$F$ (Thompson!group)}
\begin{definition}
The \emph{Thompson group}~$F$ is the group of all dyadic order-preserving self-homeomorphisms of~$[0,1]$, where a homeomorphism~$\ff$ is called \emph{dyadic} if it is piecewise linear with only finitely many breakpoints, every breakpoint of~$\ff$ has dyadic rational coordinates, and every slope of~$\ff$ is an integral power of~$2$. 
\end{definition}

Typical elements of~$F$ are displayed in Figure~\ref{F:Thompson}. In this paper, it is convenient to equip~$F$ with reversed composition, that is, $\ff \mg$ stands for~$\ff$ followed by~$\mg$---using the other convention simply amounts to reversing all expressions. The notation~$\xx_0$ is traditional for the element of~$F$ defined by
$$\xx_0(t) = 
\begin{cases}
\frac{t}2&\mbox{for $0 \leqslant t \leqslant \frac12$},\\
t-\frac14&\mbox{for $\frac12 \leqslant t \leqslant \frac34$},\\
2t-1&\mbox{for $\frac34 \leqslant t \leqslant 1$},
\end{cases}$$
and $\xx_\ii$ is used for the element that is the identity on~$[0, 1 - \frac1{2^\ii}]$ and is a rescaled copy of~$\xx_0$ on~$[1 - \frac1{2^\ii}, 1]$---see Figure~\ref{F:Thompson} again. It is easy to check that $F$ is generated by the sequence of all elements~$\xx_\ii$, with the presentation
\begin{equation}
\label{E:Pres1}
\langle \xx_0, \xx_1,... \mid \xx_{\nn+1} \xx_\ii = \xx_\ii \xx_\nn \mbox{ for $\ii < \nn$}\rangle.
\end{equation}
One deduces that $F$ is also generated by~$\xx_0$ and~$\xx_1$, with the (finite) presentation
\begin{equation}
\label{E:Pres2}
\langle \xx_0, \xx_1 \mid [\xx_0\inv \xx_1, \xx_0 \xx_1 \xx_0\inv], [\xx_0\inv \xx_1, \xx_0^2 \xx_1 \xx_0^{-2}]\rangle,
\end{equation}
where $[\xx, \yy]$ denotes the commutator~$\xx \yy \xx\inv \yy\inv$.

\begin{figure}[htb]
\begin{picture}(75,40)(0,0)
\put(-47,-35){\includegraphics{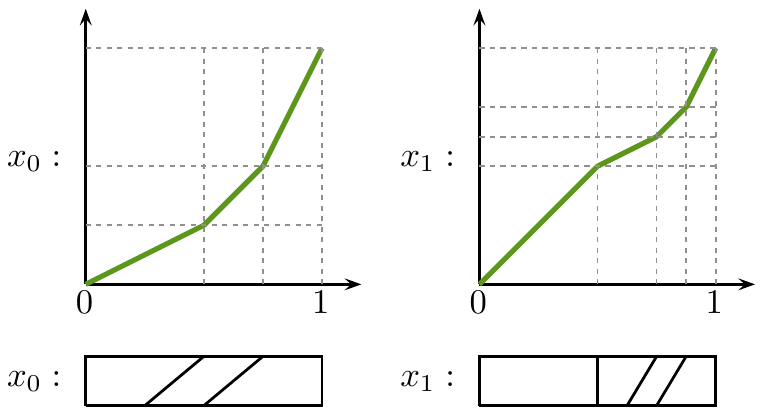}}
\end{picture}
\caption{\sf\smaller Two representations of the elements~$\xx_0$ and~$\xx_1$ of the Thompson group~$F$: above, the usual graph of a function of~$[0,1]$ into itself, below, a diagram displaying the two involved dyadic decompositions of~$[0,1]$, with the source above and the target below: this simplified diagram specifies the function entirely.}
\label{F:Thompson}
\end{figure}

The group~$F$ has many interesting algebraic and geometric properties, see~\cite{CFP}. Its center is trivial, the derived group~$[F, F]$ is a simple group, $F$ includes no free group of rank more than~$1$ (Brin--Squier \cite{BrS}), its Dehn function is quadratic (Guba \cite{Gub}). It is not known whether $F$ is automatic, nor whether $F$ is amenable. The latter question has received lot of attention as $F$ seems to lie very close to the border between amenability and non-amenability.

Owing to the developments of Section~\ref{S:Lattice} below, we mention one more (simple) algebraic result, namely that $F$ is a group of (left)-fractions, that is, there exists a submonoid of~$F$ such that every element of~$F$ can be expressed as~$\ff\inv \mg$ with $\ff, \mg$ in the considered submonoid. 

\index{Thompson!monoid}
\begin{proposition}\cite{CFP}
\label{P:Fraction}
Define the \emph{Thompson monoid}~$\Fp$ to be the submonoid of~$F$ generated by the elements~$\xx_\ii$ with $\ii \geqslant 1$. Then, as a monoid, $\Fp$ admits the presentation~\eqref{E:Pres1}, and $F$ is a group of left-fractions for~$\Fp$.
\end{proposition}  

Thus $\Fp$ consists of the elements of~$F$ that admit at least one expression in terms of the elements~$\xx_\ii$ in which no factor~$\xx_\ii\inv$ occurs. Although easy, Proposition~\ref{P:Fraction} is technically significant as its leads to a unique normal form for the elements of~$F$. 


\subsection{The action of~$F$ on trees}
\label{SS:Action}

An element of~$F$ is determined by a pair of dyadic decompositions of the interval~$[0, 1]$ specifying the intervals on which the slope has a certain value, and, from there, by a pair of trees. 

\index{dyadic decomposition}
To make the description precise, define a \emph{dyadic decomposition} of~$[0,1]$ to be an increasing sequence $(\mt_0 \wdots \mt_\nn)$ of dyadic numbers with $\mt_0 = 0$ and $\mt_\nn = 1$, such that no interval~Ê$[\mt_\ii, \mt_{\ii+1}]$ may contain a dyadic number with denominator less that those of~$\mt_\ii$ and~$\mt_{\ii+1}$: for instance, $(0, \frac12, \frac34, 1)$ is legal, but $(0, \frac34, 1)$ is not. Then dyadic decompositions are in one-to-one correspondence with binary rooted trees: the decomposition associated with~$\et$ is~$(0,1)$, whereas the one associated with~$\TT_0 \OP \TT_1$ is the concatenation of those associated with~$\TT_0$ and~$\TT_1$ rescaled to fit in~$[0,\frac12]$ and $[\frac12,1]$.

As the diagram representation of Figure~\ref{F:Thompson} shows, every element of the group~$F$ is entirely specified by a pair of dyadic decompositions, hence by a pair of trees. Provided adjacent intervals are gathered, this pair of decompositions (hence of trees) is unique. We shall denote by~$(\trm\ff, \trp\ff)$ the pair of trees associated with~$\ff$. For instance, we have $\trm1 = \trp1 = \et$, and, as illustrated in Figure~\ref{F:Seed}, $\trm{(\xx_0)} = \et(\et\et)$, $\trp{(\xx_0)} = (\et\et)\et$, $\trm{(\xx_1)} = \et(\et(\et\et))$, and $\trp{(\xx_1)} = \et((\et\et)\et)$. By construction, the trees~$\trm\ff$ and~$\trp\ff$ have the same size. Moreover, we have $\trm{(\ff\inv)} = \trp\ff$ and $\trp{(\ff\inv)} = \trm\ff$ as taking the inverse amounts to exchanging source and target in the diagram.

\begin{figure}[htb]
\begin{picture}(75,18)(0,0)
\put(-45,-35){\includegraphics{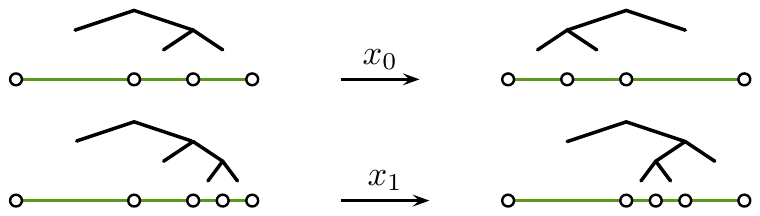}}
\end{picture}
\caption{\sf\smaller Canonical pair of trees associated with an element of the group~$F$.}
\label{F:Seed}
\end{figure}

We now define a partial action of the group~$F$ on finite trees. Hereafter, we denote by~$\Tree$ the family of all finite, binary, rooted trees, and by~$\Treelab$ the family of all (finite, binary, rooted) labeled trees whose leaves wear labels in~$\NNNN$. Thus $\Tree$ identifies with the family of all parenthesized expressions involving the single variable~$\et$, and $\Treelab$ with the family of all parenthesized expressions involving variables from the list~$\{\et_0, \et_1, ...\}$. Forgetting the labels (or the indices of variables) defines a projection of~$\Treelab$ onto~$\Tree$; by identifying~$\et$ with~$\et_0$, we can see~$\Tree$ as a subset of~$\Treelab$. If $\TT$ is a tree of~$\Tree$, we denote by~$\lab\TT$ the tree of~$\Treelab$ obtained by attaching to the leaves of~$\TT$ labels~$0, 1, ...$ starting from the left.

\index{substitution}
\begin{definition}
A \emph{substitution} is a map from~$\NNNN$ to~$\Treelab$. If $\sigma$ is a substitution and $\TT$ is a tree in~$\Treelab$, we define~$\TT^\sigma$ to be the tree obtained from~$\TT$ by replacing every $\ii$-labeled leaf of~$\TT$ by the tree~$\sigma(\ii)$.
\end{definition}

Formally, $\TT^\sigma$ is recursively defined by the rules 
$$(\et_\ii)^\sigma = \sigma(\ii), \qquad (\TT_0 \OP \TT_1)^\sigma = \TT_0^\sigma \OP \TT_1^\sigma.$$
For instance, if $\TT$ is $\et_3(\et_0\et_2)$ and we have $\sigma(0) = \et\et$ and $\sigma(2) = \sigma(3) = \et$, then $\TT^\sigma$ is~$\et((\et\et)\et)$.

\index{action!(of $F$ on trees)}
\begin{definition}
\label{D:Action}
If $\TT, \TT'$ are labeled trees and $\ff$ is an element of the Thompson group~$F$, we say that $\TT\act \ff = \TT'$ holds if we have $\TT = (\lab{\trm\ff})^\sigma$ and $\TT' = (\lab{\trp\ff})^\sigma$ for some substitution~$\sigma$.
\end{definition}

\begin{example}
\label{X:Action}
First consider $\ff = 1$. Then we have $\trm1 = \trp1 = \et$, whence $\lab{\trm1} = \lab{\trp1} =\nobreak \et_0$. For every tree~$\TT$, we have $\TT = (\lab{\trm1})^\sigma$ for any substitution satisfying $\sigma(0) = \TT$, and, in this case, we have $(\lab{\trm1})^\sigma = \TT$. So $\TT \act 1$ is always defined and it is equal to~$\TT$.

Consider now~$\xx_0$. Then we have $\trm{\xx_0} = \et(\et\et)$, whence $\lab{\trm{\xx_0}} = \et_0(\et_1\et_2)$. For a tree~$\TT$, there exists a substitution satisfying $\TT = (\et_0(\et_1\et_2))^\sigma$ if and only if $\TT$ can be expressed as $\TT_0 \OP (\TT_1 \OP \TT_2)$. In this case, the tree $((\et_0\et_1)\et_2)^\sigma$ is $(\TT_0 \OP \TT_1) \OP \TT_2$. So $\TT \act \xx_0$ is defined if and only if $\TT$ is eligible for a left-rotation and, in this case, $\TT \act \xx_0$ is the tree obtained from~$\TT$ by that left-rotation, see Figure~\ref{F:Rotation}.

Consider finally~$\xx_1$. Arguing as above, we see that $\TT \act \xx_1$ is defined if and only if $\TT$ can be expressed as $\TT_0 \OP (\TT_1 \OP (\TT_2 \OP \TT_3))$, in which case $\TT \act \ff$ is the tree $\TT_0 \OP ((\TT_1 \OP \TT_2) \OP \TT_3)$, that is, the tree obtained from~$\TT$ by a left-rotation at the right-child of the root.
\end{example}

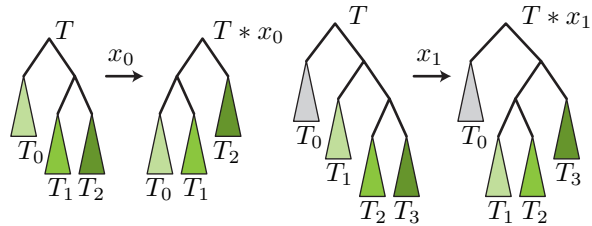
\begin{figure}[htb]
\begin{picture}(75,24)(2,0)
\put(0,0){\includegraphics{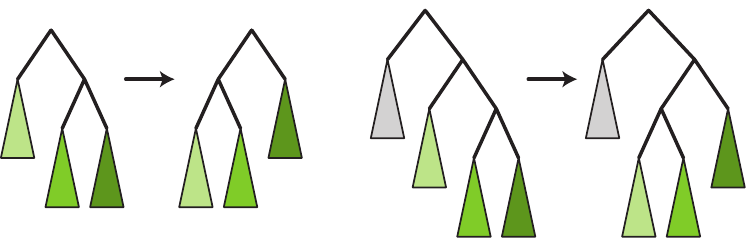}}
\put(1,6){$\TT_0$}
\put(5,1){$\TT_1$}
\put(9,1){$\TT_2$}
\put(18,1){$\TT_0$}
\put(23,1){$\TT_1$}
\put(27,6){$\TT_2$}
\put(6,22){$\TT$}
\put(27,22){$\TT \act \xx_0$}
\put(13,19){$\xx_0$}
\put(38,8){$\TT_0$}
\put(42,3){$\TT_1$}
\put(46.5,-2){$\TT_2$}
\put(51,-2){$\TT_3$}
\put(60,8){$\TT_0$}
\put(63.5,-2){$\TT_1$}
\put(68,-2){$\TT_2$}
\put(72.5,3){$\TT_3$}
\put(45,24){$\TT$}
\put(68,24){$\TT \act \xx_1$}
\put(54,19){$\xx_1$}
\end{picture}
\caption{\sf\smaller Action of~$\xx_0$ and~$\xx_1$ on a tree: respectively applying a left-rotation at the root, and at the right-child of the root.}
\label{F:Action}
\end{figure}

The above definition specifies what can naturally be called a partial action of the group~$F$ on (labeled) trees---labels are not important here as rotations do not change their order or repeat them, but they are needed for a clean definition of substitutions.

\begin{proposition}
\label{P:Action}
\ITEM1 For every (labeled) tree~$\TT$ and and every element~$\ff$ of~$F$, there exists at most one~$\TT'$ satisfying $\TT' = \TT \act \ff$.

\ITEM2 For every (labeled) tree~$\TT$, we have $\TT \act 1 = \TT$.

\ITEM3 For every (labeled) tree~$\TT$ and all $\ff, \mg$ in~$F$, we have $(\TT \act \ff) \act \mg = \TT \act \ff\mg$, this meaning that either both terms are defined and they are equal, or neither is defined.

\noindent Moreover, for all~$\ff_1 \wdots \ff_\nn$ in~$F$, there exists~$\TT$ such that $\TT \act \ff_\ii$ is defined for each~$\ii$.
\end{proposition}

\begin{proof}[Sketch, see~\cite{Dhb} for details]
\ITEM1 For~$\ff$ in~$F$, a given tree~$\TT$ can be expressed in at most one way as~$(\lab{\trm\ff})^\sigma$ and, as the same variables occur on both sides of the associativity law, there is in turn at most one corresponding tree~$(\lab{\trp\ff})^\sigma$.

Point~\ITEM2 has been established in Example~\ref{X:Action}. For~\ITEM3, the point is that there exists a simple rule for determining the pair of trees associated with~$\ff \mg$. Indeed, we have $\trm{(\ff \mg)} = \trm\ff^\sigma$ and~$\trp{(\ff \mg)} = \trp\mg^\tau$, where $\sigma$ and~$\tau $ are minimal substitutions satisfying $\trp\ff^\sigma = \trm\mg^\tau$---that is, $(\sigma, \tau)$ is a minimal identifier for~$\trp\ff$ and~$\trm\mg$.

As for the final point, it comes from the fact that, by construction, every tree~$\lab{\trm\ff}$ has pairwise distinct labels and, therefore, a tree~$\TT$ can be expressed as~$(\lab{\trm\ff})^\sigma$ if and only if the skeleton of~$\TT$ (as defined in Definition~\ref{D:Address} below) includes the skeleton of~$\lab{\trm\ff}$. Then, for $\ff_1 \wdots \ff_\nn$ in~$F$, one can always find a tree~$\TT$ whose skeleton includes those of~$\lab{\trm{(\ff_1)}} \wdots \lab{\trm{(\ff_\nn)}}$.
\end{proof}

\begin{proposition}
\label{P:Geometry}
For all (labeled) trees~$\TT, \TT'$, the following are equivalent:

\ITEM1 One can go from~$\TT$ to~$\TT'$ using a finite sequence of rotations---that is, by applying associativity;

\ITEM2 The trees~$\TT$ and~$\TT'$ have the same size, and the left-to-right enumerations of the labels in~$\TT$ and~$\TT'$ coincide;

\ITEM3 There exists~$\ff$ in~$F$ satisfying $\TT' = \TT \act \ff$.

\noindent In this case, the element~$\ff$ involved in~\ITEM3 is unique.
\end{proposition}

\begin{proof}[Sketch, see~\cite{Dhb} for details]
The equivalence of~\ITEM1 and~\ITEM2 follows from the syntactic properties of the terms occurring in the associativity law, namely that the same variables occur on both sides, in the same order.

Next, assume that $\TT$ and~$\TT'$ are equal size trees. Then $\TT$ and~$\TT'$ determine dyadic decompositions of~$[0,1]$, and there exists a dyadic homeomorphism~$\ff$, hence an element of~$F$, that maps the first onto the second. Provided the enumerations of the labels in~$\TT$ and~$\TT'$ coincide, we have $\TT' = \TT \act \ff$. So \ITEM2 implies~\ITEM3. 

Conversely, we saw in Example~\ref{X:Action} that the action of~$\xx_0$ and~$\xx_1$ is a rotation. On the other hand, we know that $\xx_0$ and~$\xx_1$ generate~$F$. Therefore, the action of an arbitrary element of~$\ff$ is a finite product of rotations. So \ITEM3 implies~\ITEM2.

Finally, the uniqueness of the element~$\ff$ possibly satisfying $\TT' = \TT \act \ff$ follows from the fact that the pair~$(\TT, \TT')$ determines a unique pair of dyadic decompositions of~$[0, 1]$, so it directly determines the graph of the dyadic homeomorphism~$\ff$.
\end{proof}

Proposition~\ref{P:Geometry} states that $F$ is the \emph{geometry group} of associativity in the sense of~\cite{Dhb}. A similar approach can be developed for every algebraic law, and more generally every family of algebraic laws, leading to a similar geometry monoid (a group in good cases). In the case of associativity together with commutativity, the geometry group happens to be the Thompson group~$V$, whereas, in the case of the left self-distributivity law $\xx(\yy\zz) = (\xx\yy)(\xx\zz)$, the geometry group is a certain ramified extension of Artin's braid group~$B_\infty$ \cite{Dgd}---also see the case of $\xx(\yy\zz) = (\xx\yy)(\yy\zz)$ in~\cite{Dgj}. In the latter cases, the situation is more complicated than with associativity as, in particular, the counterparts of~\ITEM1 and~\ITEM2 in Proposition~\ref{P:Geometry} fail to be equivalent.


\subsection{The generators~$\ga\alpha$}
\label{SS:AAA}

Considering the action of the group~$F$ on trees invites us to introducing, beside the standard generators~$\xx_\ii$, a new, more symmetric family of generators for~$F$. 

In order to define these elements, we need an index system for the subtrees of a tree. A common solution consists in describing the path connecting the root of the tree to the root of the considered subtree using (for instance) $0$ for ``forking to the left'' and~$1$ for ``forking to the right''.

\index{address!(binary)}
\index{skeleton!(of a tree)}
\begin{definition}
\label{D:Address}
A finite sequence of~$0$'s and~$1$'s is called an {\it address}; the empty address is denoted by~$\ea$. For~$\TT$ a tree and $\alpha$ a short enough address, the $\alpha$-{\it subtree} of~$\TT$ is the part of~$\TT$ that lies below~$\alpha$. The set of all~$\alpha$'s for which the $\alpha$-subtree of~$\TT$ exists is called the \emph{skeleton} of~$\TT$.
\end{definition}

Formally, the $\alpha$-subtree is defined by the following recursive rules: the $\ea$-subtree of~$\TT$ is~$\TT$, and, for $\alpha = 0\beta$ (\resp.\ $1\beta$), the $\alpha$-subtree of~$\TT$ is the $\beta$-subtree of~$\TT_0$ (\resp.\ $\TT_1$) if $\TT$ is~$\TT_0 \OP \TT_1$, and it is undefined otherwise. 

\begin{example}
For $\TT = \et((\et\et)\et)$ (the rightmost example in Figure~\ref{F:Trees}), the $10$-subtree of~$\TT$ is $\et\et$, while the $01$- and $111$-subtrees are undefined. The skeleton of~$\TT$ consists of the seven addresses~$\ea$, $0$, $1$, $10$, $100$, $101$, and~$11$.
\end{example}
 
By definition, applying associativity in a parenthesized expression or, equivalently, applying a rotation in a tree~$\TT$ consists in choosing an address~$\alpha$ in the skeleton of~$\TT$ and either replacing the $\alpha$-subtree of~$\TT$, supposed to have the form $\TT_0 \OP (\TT_1 \OP \TT_2)$, by the corresponding $(\TT_0 \OP \TT_1) \OP \TT_2$, or vice versa, see Figure~\ref{F:Rotation} again. By Proposition~\ref{P:Geometry}, this rotation corresponds to a (unique) element of~$F$.

\begin{definition}
For every address~$\alpha$, we denote by~$\ga\alpha$ the element of~$F$ whose action is a left-rotation at~Ê$\alpha$. We denote by~$\AAA$ the family of all elements~$\ga\alpha$ for~$\alpha$ an address.
\end{definition}

According to Example~\ref{X:Action} and Figure~\ref{F:Action}, the action of~$\xx_0$ is a left-rotation at the root of the tree, and, therefore, we have $\xx_0 = \ga\ea$. Similarly, $\xx_1$ is left-rotation at the right-child of the root, that is, at the node with address~$1$, and, therefore, we have $\xx_1 = \ga1$. More generally, all elements~$\ga\alpha$ can be expressed in terms of the generators~$\xx_\ii$, as will be done in Subsection~\ref{SS:Pres} below. For the moment, we simply note that iterating the argument for~$\xx_1$ gives for every~$\ii \geqslant 1$ the equality $\xx_\ii = \ga{1^{\ii-1}}$, where $1^{\ii-1}$ denotes $11...1$, $\ii-1$~times~$1$.

The trees~$\TT$ such that $\TT \act \ga\alpha$ is defined are easily characterized. Indeed, a necessary and sufficient for~$\TT \act \ga\alpha$ to exist is that the $\alpha$-subtree of~$\TT$ is defined and a left-rotation can be applied to that subtree, that is, it can be expressed as $\TT_0 \OP (\TT_1 \OP \TT_2)$. This is true if and only if the addresses~$\alpha0$, $\alpha10$, and~$\alpha11$ lie in the skeleton of~$\TT$, hence actually if and only if $\alpha10$ lies in the skeleton of~$\TT$ since $\beta0$ may lie in the skeleton of a tree only if $\beta1$ and $\beta$ do. Symmetrically, $\TT \act \ga\alpha\inv$ is defined if and only if $\alpha01$ lies in the skeleton of~$\TT$. As a tree has a finite skeleton, there exist for every tree~$\TT$ finitely many addresses~$\alpha$ such that $\TT \act \ga\alpha^{\pm1}$ is defined, see Figure~\ref{F:ExAA}.

\begin{figure}[htb]
\begin{picture}(75,26)(1,2)
\put(0,0){\includegraphics{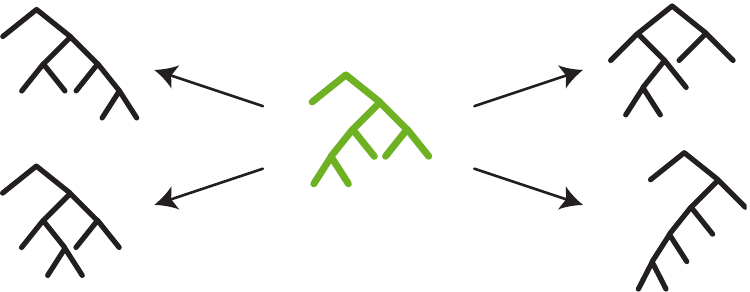}}
\put(20,7){$\ga{10}\inv$}
\put(20,23){$\ga1\inv$}
\put(51,7){$\ga1$}
\put(51,23){$\ga\ea$}
\put(35,23){$\scriptstyle\ea$}
\put(39,20){$\scriptstyle1$}
\put(32.5,16.5){$\scriptstyle10$}
\end{picture}
\caption{\sf\smaller The four elements~$\ga\alpha^{\pm1}$ such that $\TT \act \ga\alpha^{\pm1}$ is defined in the case $\TT = \et (((\et \et) \et) (\et \et))$.}
\label{F:ExAA}
\end{figure}

Before proceeding, we note that the forking nature of the family~$\AAA$ naturally gives rise to a large family of shift endomorphisms of the group~$F$.

\begin{lemma}
For every address~$\alpha$, there exists a (unique) shift endomorphism~$\sh\alpha$ of~$F$ that maps~$\ga\beta$ to~$\ga{\alpha\beta}$ for every~$\beta$.
\end{lemma}

\begin{proof}
For~$\ff$ in~$F$, let $\sh1(\ff)$ denote the homeomorphism obtained by rescaling~$\ff$, applying it in the interval~$[\frac12, 1]$, and completing with the identity on~$[0,\frac12]$. Then $\sh1$ is an endomorphism of~$F$, and it  maps~$\xx_\ii$ to~$\xx_{\ii+1}$ for every~$\ii$. Moreover, for every~$\beta$, the element~$\sh1(\ga\beta)$ is the rescaled version of~$\ga\beta$ applied in the interval~$[\frac12, 1]$. By definition, this is~$\ga{1\beta}$.

Symmetrically, for~$\ff$ in~$F$, let $\sh0(\ff)$ denote the homeomorphism obtained by rescaling~$\ff$, applying it in the interval~$[0, \frac12]$, and completing with the identity on~$[\frac12, 1]$. Then $\sh0$ is an endomorphism of~$F$, and, for every~$\beta$, the element~$\sh0(\ga\beta)$ is the rescaled version of~$\ga\beta$ applied in the interval~$[0, \frac12]$, hence it is~$\ga{0\beta}$.

Finally, we recursively define~$\sh\alpha$ for every~$\alpha$ by $\sh\ea = \mathrm{id}_F$ and,  for $\ii = 0,1$,  $\sh{\ii\alpha}(\ff) = \sh\ii(\sh\alpha(\ff))$. By construction, $\sh\alpha(\ga\beta) = \ga{\alpha\beta}$ holds for all~$\alpha, \beta$.
\end{proof}


\subsection{Presentation of~$F$ in terms of the elements~$\ga\alpha$}
\label{SS:Pres}

As the family $\{\xx_0, \xx_1\}$, which is $\{\ga\ea, \ga1\}$, generates the group~$F$, the family~$\AAA$ generates~$F$ as well. By using the presentation~\eqref{E:Pres1} or~\eqref{E:Pres2}, we could easily deduce a presentation of~$F$ in terms of the elements~$\ga\alpha$. However, we can obtain a more natural and symmetric presentation by coming back to trees and associativity, and exploiting the geometric meaning of the elements of~$\AAA$.

\begin{lemma}
\label{L:Pres3}
Say that two addresses~$\alpha, \beta$ are \emph{orthogonal}, written $\alpha \perp \beta$, if there exists~$\gamma$ such that $\alpha$ begins with~$\gamma0$ and $\beta$ begins with~$\gamma1$, or vice versa. Then all relations of the following family~$\RRR$ are satisfied in~$F$:
\begin{gather}
\label{E:Comm}
\ga\alpha \, \ga\beta = \ga\beta \, \ga\alpha \mbox{\qquad for $\alpha \perp \beta$},\\
\label{E:QComm}
\ga{\alpha11\beta} \, \ga\alpha = \ga\alpha \, \ga{\alpha1\beta},
\quad
\ga{\alpha10\beta} \, \ga\alpha = \ga\alpha \, \ga{\alpha01\beta},
\quad
\ga{\alpha0\beta} \, \ga\alpha = \ga\alpha \, \ga{\alpha00\beta},\\
\label{E:Pentagon}
\ga\alpha^2 =\ga{\alpha1} \, \ga\alpha \, \ga{\alpha0}.
\end{gather}
\end{lemma}

\begin{proof}
By Proposition~\ref{P:Geometry}, in order to prove that two elements~$\ff, \ff'$ of~$F$ coincide, it is enough to exhibit a tree~$\TT$ such that $\TT \act \ff$ and~$\TT \act \ff'$ are defined and equal. 

The commutation relations of type~\eqref{E:Comm} are trivial. If $\alpha$ and~$\beta$ are orthogonal, the $\alpha$- and $\beta$-subtrees are disjoint, and the result of applying rotations (as well as any transformations) in each of these subtrees does not depend on the order. So we have 
$$\sh\alpha(\ff) \, \sh\beta(\mg) = \sh\beta(\mg) \, \sh\alpha(\ff)$$ for all transformations~$\ff, \mg$ and, in particular, $\ga\alpha \, \ga\beta = \ga\beta \, \ga\alpha$. 

The quasi-commutation relations of type~\eqref{E:QComm} are more interesting. Assume that $\TT, \TT'$ are trees and $\ga\ea$ maps~$\TT$ to~$\TT'$. Then, by definition, the $1$-subtree of~$\TT'$ is a copy of the $11$-subtree of~$\TT$. Now, assume that $\ff$ is a (partial) mapping of~$\Tree$ to itself. Then, starting from~$\TT$, first applying~$\ga\ea$ and then applying~$\ff$ to the $11$-subtree leads to the same result as first applying~$\ff$ to the $1$-subtree and then applying~$\ga\ea$, see Figure~\ref{F:QuasiComm}. Moreover, if $\ff$ is a partial mapping, the result of one operation is defined if and only if the result of the other is. So, in all cases, we have
$$\ga\ea \, \sh{11}(\ff) = \sh1(\ff) \, \ga\ea.$$
Applying this to $\ff = \ga\beta$ then gives $\ga\ea \, \ga{11\beta} = \ga{1\beta} \, \ga\ea$. Shifting by~$\alpha$ this relation, we obtain $\ga\alpha \, \ga{\alpha11\beta} = \ga{\alpha1\beta} \, \ga\alpha$, the first relation of~\eqref{E:QComm}. Arguing similarly with the $0$- and $10$-subtrees in place of the $11$-subtree, one obtains the other relations of~\eqref{E:QComm}.

Finally, the relations of~\eqref{E:Pentagon} stem from the pentagon of Figure~\ref{F:Tam3}. As Figure~\ref{F:Pentagon} shows, the relation $\ga\ea^2 = \ga{1} \,  \ga\ea \, \ga{0}$ is satisfied in~$F$ and, therefore, so is its shifted version $\ga\alpha^2 = \ga{\alpha1} \,  \ga\alpha \, \ga{\alpha0}$ for every address~$\alpha$.
\end{proof}

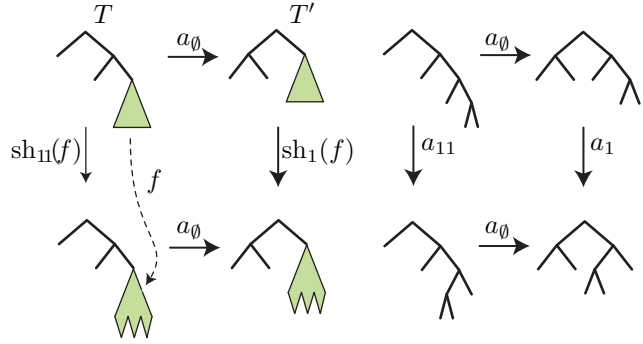
\begin{figure}[htb]
\begin{picture}(65,40)(5,4)
\put(0,0){\includegraphics{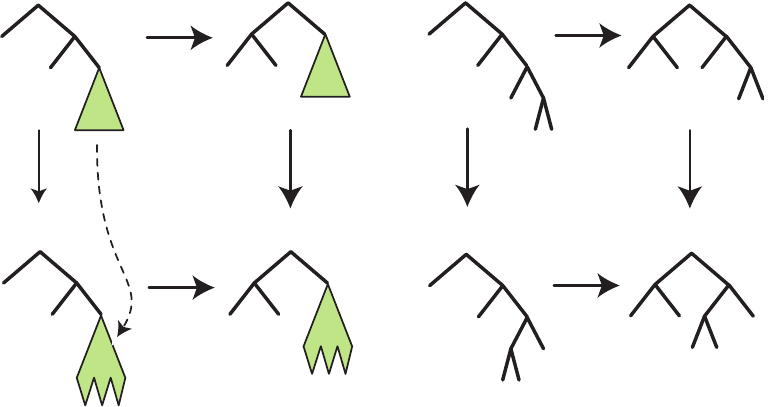}}
\put(5,42){$\TT$}
\put(31,42){$\TT'$}
\put(12,20){$\ff$}
\put(-6,24){$\sh{1\!1}\!(\!\ff)$}
\put(30,24){$\sh1(\ff)$}
\put(16,39){$\ga\ea$}
\put(16,14){$\ga\ea$}
\put(57,39){$\ga\ea$}
\put(57,14){$\ga\ea$}
\put(48.5,25){$\ga{11}$}
\put(71,25){$\ga1$}
\end{picture}
\caption{\sf\smaller Quasi-commutation relation in~$F$: the general scheme and one example.}
\label{F:QuasiComm}
\end{figure}

\begin{figure}[htb]
\begin{picture}(70,25)(0,4)
\put(0,0){\includegraphics{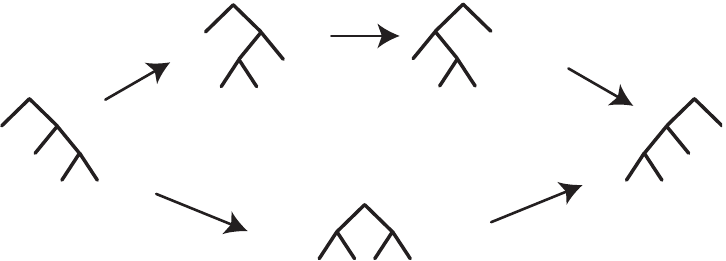}}
\put(11.5,20){$\ga1$}
\put(35,25){$\ga\ea$}
\put(60,20){$\ga0$}
\put(20,7){$\ga\ea$}
\put(51,7){$\ga\ea$}
\end{picture}
\caption{\sf\smaller Pentagon relation in the group~$F$.}
\label{F:Pentagon}
\end{figure}
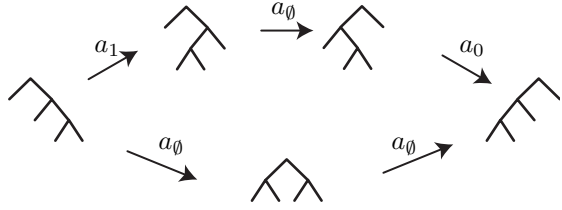

It is then easy to check that the above relations actually exhaust the relations connecting the elements~$\ga\alpha$ in the group~$F$.

\begin{proposition}\cite{Dfg, Dhb}
\label{P:Pres3}
The group~$F$ admits the presentation~$\langle\AAA\mid\RRR\rangle$.
\end{proposition}

\begin{proof}
By Lemma~\ref{L:Pres3}, the relations of~$\RRR$ are valid in~$F$. Conversely, to prove that these relations make a presentation, it is sufficient to show that they include the relations of a previously known presentation. This is what happens as, for $1 \leqslant \ii < \nn$, the relation $\ga{1^\nn} \ga{1^{\ii-1}} = \ga{1^{\ii-1}} \ga{1^{\nn-1}}$, which is a reformulation of the relation $\xx_{\nn+1} \xx_\ii = \xx_\ii \xx_\nn$ of~\eqref{E:Pres1}, occurs in~$\RRR$ as the first relation of~\eqref{E:QComm} with $\alpha = 1^{\ii-1}$ and~$\beta = 1^{\nn-\ii}$.
\end{proof}

As an application, we compute the elements~$\ga\alpha$ in terms of the generators~$\xx_\ii$.

\begin{proposition}
If $\alpha$ is an address containing at least one~$0$, say $\alpha = 1^\ii0^{1+\ii_0} 10^{\ii_1}...1 0^{\ii_\mm}$ with $\mm \geqslant 0$ and $\ii, \ii_0 \wdots \ii_\mm \geqslant 0$, then, putting $\mg = \xx_{\ii+\mm+1}^{\ii_\mm + 1} \pdots \xx_{\ii+2}^{\ii_1+1}  \xx_{\ii+1}^{\ii_0+1}$, we have
\begin{equation}
\label{E:Comput}
\ga\alpha =  \mg\inv \xx_{\ii+\mm+2}\inv  \xx_{\ii+\mm+1} \mg.
\end{equation}
\end{proposition}

\begin{proof}
It is sufficient to establish the formula in the case $\ii = 0$ as, then, applying $\sh1^\ii$ gives the general case. We use induction on~$(\mm, \ii_0)$ with respect to the lexicographical (well)-order, that is, $(\mm', \ii'_0)$ is smaller than $(\mm, \ii_0)$ if and only if we have either $\mm' < \mm$, or $\mm' = \mm$ and~$\ii'_0 < \ii_0$. 

Assume first $(\mm, \ii_0) = (0, 0)$, that is, $\alpha = 0$. Then the pentagon relation at~$\ea$ gives
$$\ga\alpha = \ga0 = \ga\ea\inv \ga1\inv \ga\ea^2 = \xx_1\inv (\xx_2\inv \xx_1) \xx_1,$$
which is the expected instance of~\eqref{E:Comput}. 
Assume now $\mm \geqslant 1$ and $\ii_0 = 0$, that is $\alpha = 010^{\ii_1} \pdots 10^{\ii_\mm}$. Then the quasi-commutation relation for $\ea$ and~$01\beta$ gives 
\begin{equation}
\label{E:Comput2}
\ga\alpha = \ga{010^{\ii_1} \pdots 10^{\ii_\mm}} = \ga\ea\inv \ga{10^{1+\ii_1}10^{\ii_2} \pdots 10^{\ii_\mm}} \ga\ea = \xx_1\inv  (\ga{10^{1+\ii_1}10^{\ii_2} \pdots 10^{\ii_\mm}}) \xx_1.
\end{equation}
The number of non-initial symbols~$1$ in $0^{1+\ii_1}10^{\ii_2} \pdots 10^{\ii_\mm}$ is $\mm-1$. As $(\mm-1, \ii_1)$ is smaller than~$(\mm, 0)$, the induction hypothesis gives $\ga{0^{1+\ii_1}10^{\ii_2} \pdots 10^{\ii_\mm}} = \mg\inv \xx_{\mm+1}\inv \xx_{\mm} \mg$ with $\mg = \xx_{\mm}^{\ii_\mm + 1} \pdots  \xx_{2}^{\ii_2+1} \xx_{1}^{\ii_1+1}$. Using~$\sh1$, we get $\ga{10^{1+\ii_1}10^{\ii_2} \pdots 10^{\ii_\mm}} = \hh\inv \xx_{\mm+2}\inv \xx_{\mm+1} \hh$ with $\hh = \xx_{\mm+1}^{\ii_\mm + 1} \pdots  \xx_3^{\ii_2+1} \xx_2^{\ii_1+1}$. Merging with~\eqref{E:Comput2}, we deduce the expected value for~$\ga\alpha$.  

Assume finally $\ii_0 \geqslant 1$. Then the quasi-commutation relation for $\ea$ and~$00\beta$ gives 
$$\ga\alpha = \ga{0^{1+\ii_0}10^{\ii_1} \pdots 10^{\ii_\mm}} = \ga\ea\inv \ga{0^{\ii_0}10^{\ii_1} \pdots 10^{\ii_\mm}} \ga\ea = \xx_1\inv \ga{0^{\ii_0}10^{\ii_1} \pdots 10^{\ii_\mm}} \xx_1.$$
The pair $(\mm, \ii_0-1)$ is smaller than the pair~$(\mm, \ii_0)$, so the induction hypothesis gives $\ga\alpha = \xx_1\inv (\mg\inv \xx_{\mm+2}\inv  \xx_{\mm+1} \mg) \xx_1$ with $\mg = \xx_{\mm+1}^{\ii_\mm + 1} \pdots  \xx_2^{\ii_1+1} \xx_1^{\ii_0}$, again the expected instance of~\eqref{E:Comput}. So the induction is complete.
\end{proof}

\begin{example}
Consider $\alpha = 01100$, which corresponds to $\mm = 2$, and $\ii = \ii_0 = \ii_1 =\nobreak 0$, and $\ii_2 = 2$. Then we find $\ga\alpha = \mg\inv \xx_4\inv \xx_3 \mg$ with $\mg = \xx_3^3 \xx_2 \xx_1$, that is, $\ga{01100}$ is equal to~$\xx_1\inv \xx_2\inv \xx_3^{-3} \xx_4\inv \xx_3^4 \xx_2 \xx_1$.
\end{example}


\section{A lattice structure on the Thompson group~$F$}
\label{S:Lattice}

Here comes the core of our study, namely the investigation of the submonoid~$\Fs$ of~$F$ generated by the elements~$\ga\alpha$. The main result is that $\Fs$ has the structure of a lattice when equipped with its divisibility relation, and that this lattice is closely connected with the Tamari lattices, which occur as initial sublattices. 

These results are not trivial, as, in particular, determining a presentation of~$\Fs$ is not so easy. Our approach relies on using subword reversing, a general method of combinatorial group theory that turns out to be well suited for~$\Fs$. One of the outcomes is a new proof (one more!) of the fact that Tamari posets are lattices. 

The section is organized as follows. The symmetric Thompson monoid~$\Fs$ is introduced in Subsection~\ref{SS:Fs}, and it is investigated in Subsection~\ref{SS:Rev} using subword reversing. The lattice structure on~$\Fs$ and its connection with the Tamari lattices are described in Subsection~\ref{SS:Lattice}. Finally, a few results about the algorithmic complexity of the reversing process are gathered in Subsection~\ref{SS:Algo}.


\subsection{The symmetric Thompson monoid~$\Fs$}
\label{SS:Fs}

Once new generators~$\ga\alpha$ of the Thompson group~$F$ have been introduced, it is natural to investigate the submonoid generated by these elements.

\index{symmetric!Thompson!monoid}
\begin{definition}
The \emph{symmetric Thompson monoid}~$\Fs$ is the submonoid of~$F$ generated by the elements~$\ga\alpha$ with~$\alpha$ a binary address.
\end{definition}

The family~$\AAA$ of all elements~$\ga\alpha$ is a sort of closure of the family of standard generators~$\xx_\ii$ under all local left--right symmetries, so the above terminology is natural. Another option could be to call~$\Fs$ the \emph{dual Thompson monoid} as the relation of~$\Fp$ and~$\Fs$ is reminiscent of the relation of the standard braid monoids and the dual braid monoids generated by the Birman--Ko--Lee braids.

Although straightforward, the following connection is essential for our purpose:

\begin{lemma}
\label{L:Compat}
For all trees $\TT, \TT'$, the following are equivalent

\ITEM1 We have $\TT \leT \TT'$ in the Tamari order;

\ITEM2 There exists~$\ff$ in~$\Fs$ satisfying $\TT' = \TT \act \ff$.
\end{lemma}

\begin{proof}
By definition, $\TT \leT \TT'$ holds if there exists a finite sequence of left-rotations transforming~$\TT$ into~$\TT'$. Now applying the left-rotation at~$\alpha$ is letting~$\ga\alpha$ act.
\end{proof}

In order to investigate the monoid~$\Fs$ and its connection with the Tamari lattices, it will be necessary to first know a presentation of~$\Fs$. Owing to Propositions~\ref{P:Fraction} and~\ref{P:Pres3}, the following result should not be a surprise.

\begin{proposition}\cite{Dfg, Dhb}
\label{P:DualFraction}
The monoid~$\Fs$ admits the presentation~$\langle\AAA\mid\RRR\rangle^{\scriptscriptstyle+}$, and $F$ is a group of right-fractions for~$\Fs$ (that is, every element of~$F$ can be expressed as $\ff \mg\inv$ with $\ff, \mg$ in~$\Fs$).
\end{proposition}

However, the proof of Proposition~\ref{P:DualFraction} is more delicate than the proof of Proposition~\ref{P:Fraction}, and no very simple argument is known. 

\begin{proof}[Sketch of the proof developed in~\cite{Dfg, Dhb}]
In order to prove that the relations of~$\RRR$ generate all relations connecting the elements~$\ga\alpha$ in the monoid~$\Fp$, one introduces, for every size~$\nn$ tree~$\TT$, an explicit sequence~$\Blue\TT$ of elements~$\ga\alpha$ satisfying $\CbR\nn \act \Blue\TT = \TT$---as will be made in the proof of Proposition~\ref{P:Connection} below. The point is then to show that, if $\TT' = \TT \act \ww$ holds, then the relations of~$\RRR$ are sufficient to establish the equivalence of~$\Blue{\TT'}$ and~$\Blue\TT \ww$. Then, if two $\AAA$-words $\uu, \vv$ represent the same element of~$\Fs$, and $\TT$ is a tree such that both $\TT \act \uu$ and~$\TT \act \vv$ are defined, the above argument shows that $\Blue\TT \uu$ and $\Blue\TT \vv$ are $\RRR$-equivalent, since both are $\RRR$-equivalent to~$\Blue{\TT \act \uu}$. Provided $\RRR$-equivalence is known to allow left-cancellation, one deduces that $\uu$ and~$\vv$ are $\RRR$-equivalent, as expected. 
\end{proof}

Here we shall propose a new proof, which is more lattice-theoretic in that it exclusively relies on the so-called subword reversing method, which we shall see below is directly connected with the Tamari lattice operations. Instead of working with~$\Fs$, we investigate the abstract monoid~$\MON\AAA\RRR$ defined by the presentation~$(\AAA, \RRR)$ of Proposition~\ref{P:Pres3}. A priori, as $\Fs$ is generated by~$\AAA$ and satisfies the relations of~$\RRR$, we only know that $\Fs$ is a quotient of~$\MON\AAA\RRR$. 


\index{left-divisor}
\index{right-multiple}
\begin{definition}
Assume that $\MM$ is a monoid. For $\ff, \mg$ in~$\MM$, we say that $\ff$ \emph{left-divides}~$\mg$, or that $\mg$ is a \emph{right-multiple} of~$\ff$, written $\ff \dive \mg$, if $\ff \mg' = \mg$ holds for some~$\mg'$ of~$\MM$. We use $\Div(\ff)$ for the family of all left-divisors of~$\ff$.
\end{definition}

It is standard that the left-divisibility relation is a partial pre-ordering. Moreover, if $\MM$ contains no invertible element except~$1$, this partial pre-ordering is a partial ordering, that is, the conjunction of $\ff \dive \mg$ and $\mg \dive \ff$ implies~$\ff = \mg$.

\begin{lemma}
\label{L:Strategy}
In order to establish Proposition~\ref{P:DualFraction}, it is sufficient to prove that the monoid~$\MON\AAA\RRR$ is cancellative and any two elements admit a common right-multiple.
\end{lemma}

\begin{proof}
A classical result of Ore (see for instance~\cite{ClP}) says that, if a monoid~$\MM$ is cancellative and any two elements of~$\MM$ admit a common right-multiple, then $\MM$ embeds in a group of right-fractions~$\GG$. Moreover, if $\MM$ admits the presentation $\MON\mA\RR$, then $\GG$ admits the presentation $\GR\mA\RR$. So, if the hypotheses of the lemma are satisfied, then the monoid~$\MON\AAA\RRR$ embeds in a group of fractions that admits the presentation~$\GR\AAA\RRR$. By Proposition~\ref{P:Pres3}, the group~$\GR\AAA\RRR$ is the group~$F$. Therefore, $\MON\AAA\RRR$ is isomorphic to the submonoid of~$F$ generated by~$\AAA$, that is, to~$\Fs$. Hence $\Fs$ admits the expected presentation, and $F$ is a group of right-fractions for~$\Fs$.
\end{proof} 

\subsection{Subword reversing}
\label{SS:Rev}

In order to apply the strategy of Lemma~\ref{L:Strategy}, we have to prove that the presented monoid~$\MON\AAA\RRR$ is cancellative and any two elements of~$\MON\AAA\RRR$ admit a common right-multiple. The subword reversing method \cite{Dff, Dgp} proves to be relevant. We recall below the basic notions, and refer to~\cite{Dia} or~\cite[Section~II.4]{Garside} for a more complete description. 

Hereafter, words in an alphabet~$\mA$ are called \emph{(positive) $\mA$-words}, whereas words in the alphabet~$\mA \cup \mA\inv$, where $\mA\inv$ consists of a copy~$\ma\inv$ for each letter~$\ma$ of~$\mA$, are called \emph{signed $\mA$-words}. We say that a group presentation~$(\mA, \RR)$ is \emph{positive} if all relations in~$\RR$ have the form~$\uu = \vv$ where $\uu$ and~$\vv$ are nonempty positive $\mA$-words. We denote by~$\MON\mA\RR$ and by~$\GR\mA\RR$ the monoid and the group presented by~$(\mA, \RR)$, respectively, and we use~$\equivp_\RR$ (\resp.\ $\equiv_\RR$) for the congruence on positive $\mA$-words (\resp.\ on signed $\mA$-words) generated by~$\RR$. Finally, for~$\ww$ a signed $\mA$-word, we denote by~$\cl\ww$ the element of~$\GR\mA\RR$ represented by~$\ww$, that is, the $\equiv_\RR$-class of~$\ww$.

\index{right-reversing}
\begin{definition}
Assume that $(\mA, \RR)$ is a positive presentation. If $\ww, \ww'$ are signed $\mA$-words, we say that $\ww$ is \emph{right-$\RR$-reversible} to~$\ww'$ in one step if $\ww'$ is obtained from~$\ww$ either by deleting some length~$2$ subword~$\ma\inv \ma$ or by replacing some length~$2$ subword~$\ma\inv \mb$ with a word~$\vv \uu\inv$ such that $\ma \vv = \mb \uu$ is a relation of~$\RR$. We write $\ww \rev_\RR \ww'$ if $\ww$ is right-$\RR$-reversible to~$\ww'$ in finitely many steps.
\end{definition}

The principle of right-$\RR$-reversing is to use the relations of~$\RR$ to push the negative letters (those with exponent~$-1$) to the right, and the positive letters (those with exponent~$+1$) to the left. The process can be visualized in diagrams as in Figure~\ref{F:Rev}.

\begin{example}
Consider the presentation~$(\AAA, \RRR)$, which is positive. Let~$\ww$ be the signed $\AAA$-word $\ga1\inv \ga\ea \ga{00}\inv \ga1$. Then $\ww$ contains two negative--positive length~$2$ subwords, namely~$\ga1\inv \ga\ea$ and $\ga{00}\inv \ga1$.  There exists in~$\RRR$ a unique relation of the form $\ga1 \pdots = \ga\ea \pdots$, namely $\ga1 \ga\ea \ga0 = \ga\ea^2$, and a unique relation $\ga{00} \pdots = \ga1 \pdots$, namely $\ga{00} \ga1 = \ga1 \ga{00}$. Therefore, there exists two ways to right-$\RRR$-reverse~$\ww$, namely replacing~$\ga1\inv \ga\ea$ with $\ga\ea \ga{00}\inv \ga\ea\inv$ and obtaining $\ww_1 = \ga\ea \ga0 \ga\ea\inv \ga1$, or replacing $\ga{00}\inv \ga1$ with $\ga1 \ga{00}\inv$ and obtaining $\ww'_1 = \ga1\inv \ga\ea \ga1 \ga{00}\inv$. The words~$\ww_1$ and~$\ww'_1$ each contain a unique negative--positive length~$2$ subword, and reversing it leads in both cases to $\ww_2 = \ga\ea \ga0 \ga\ea\inv \ga1 \ga{00}\inv$. The word~$\ww_2$ contains a unique negative--positive length two subword and reversing it leads to $\ww_3 = \ga\ea \ga0 \ga\ea \ga0\inv \ga\ea\inv \ga{00}\inv$. As the latter word contains no negative--positive subword, no further right-reversing is possible. See Figure~\ref{F:Rev}.
\end{example}

\begin{figure}[htb]
\begin{picture}(36,40)(0,0)
\put(-45,-46){\includegraphics{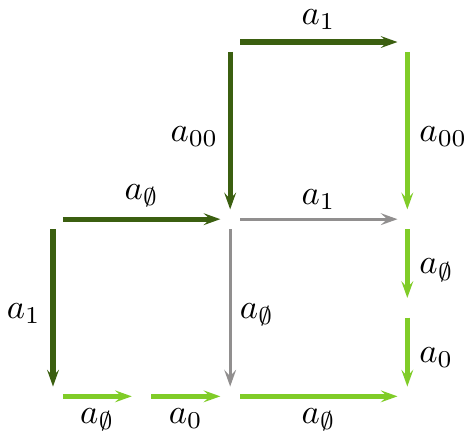}}
\end{picture}
\caption{\sf\smaller Right-$\RRR$-reversing of the signed $\AAA$-word $\ga1\inv \ga\ea \ga{00}\inv \ga1$: we draw the initial word as a zigzag path (here in dark green) from SW to NE by associating with every letter~$\ma$ a horizontal arrow labeled~$\ma$ and every letter~$\ma\inv$ a vertical arrow labeled~$\ma$ (crossed in the wrong direction); then reversing $\ma\inv \mb$ to~$\vv \uu\inv$ corresponds to closing the open pattern made by a vertical $\ma$-arrow and a horizontal $\mb$-arrow with the same source by adding horizontal arrows labeled~$\vv$ and vertical arrows labeled~$\uu$; the final word corresponds to the rightmost path from the SW corner to the NE corner, here $\ga\ea \ga0 \ga\ea \ga0\inv \ga\ea\inv \ga{00}\inv$\break (light green).}
\label{F:Rev}
\end{figure}

\vspace{-2mm}

It is easy to see that, if $(\mA, \RR)$ is a positive presentation and $\ww, \ww'$ are signed $\mA$-words, then $\ww \rev_\RR \ww'$ implies $\ww \equiv_\RR \ww'$ and that, if $\uu, \vv, \uu', \vv'$ are positive $\mA$-words, then $\uu\inv \vv \rev_\RR \vv' \uu'{}\inv$ implies $\uu \vv' \equivp_\RR\nobreak \vv \uu'$. In particular, using~$\eps$ for the empty word, 
\begin{equation}
\label{E:RevEquiv}
\uu\inv \vv \rev_\RR \eps \mbox{\quad  implies \quad} \uu \equivp_\RR \vv.
\end{equation} 
In general, \eqref{E:RevEquiv} need not be an equivalence, but it turns out that this is the interesting situation, in which case the presentation~$(\mA, \RR)$ is said to be \emph{complete with respect to right-reversing}. Roughly speaking, a presentation is complete with respect to right-reversing if right-reversing always detects equivalence. The important point here is that the presentation~$(\AAA, \RRR)$ has this property.

\begin{lemma}\cite{Dfg}
\label{L:Complete}
The presentation $(\AAA, \RRR)$ is complete with respect to right-reversing.
\end{lemma}

\begin{proof}[Sketch]
By~\cite[Proposition~2.9]{Dia}, a sufficient condition for a positive presentation $(\mA, \RR)$ to be complete with respect to right-reversing is that $(\mA, \RR)$ satisfies
\begin{equation}
\label{E:Complete}
\parbox{100mm}{\ITEM1 There exists a $\equivp_\RR$-invariant map $\wit$ from positive $\mA$-words to~$\NNNN$ satisfying $\wit(\uu \vv) \geqslant \wit(\uu) + \wit(\vv)$ for all~$\uu, \vv$ and $\wit(\ma) \geqslant 1$ for~$\ma$ in~$\mA$, and

\ITEM2 For all~$\ma, \mb, \mc$ in~$\mA$ and all positive $\mA$-words~$\uu, \vv$, if $\ma\inv \mc \mc\inv \mb \rev_\RR \vv \uu\inv$ holds, then $\vv\inv \ma\inv \mb \uu \rev_\RR \eps$ holds as well.}
\end{equation}
We claim that $(\AAA, \RRR)$ satisfies~\eqref{E:Complete}. As for~\ITEM1, we cannot use for~$\wit$ the length of words, as it is not $\equivp_{\RRR}$-invariant: in the pentagon relation, the length~$2$ word~$\ga\ea^2$ is $\equivp_{\RRR}$-equiv\-alent to the length~$3$ word~$\ga1 \ga\ea \ga0$. Now, for~$\TT$ a tree, let $\MZ\TT$ be the total number of~$0$'s occurring in the addresses of the leaves of~$\TT$: for instance, we have $\MZ{\et(\et\et)} = 2$ and $\MZ{(\et\et)\et} = 3$, as the leaves of~$\et(\et\et)$ have addresses~$0$, $10$,~$11$, with two~$0$'s, and those of~$(\et\et)\et$ have addresses~$00$, $01$,~$1$, with three~$0$'s. Then put
\begin{equation}
\label{E:Wit}
\wit(\ww) = \MZ{\trp{\cl\ww}} - \MZ{\trm{\cl\ww}}.
\end{equation}
For instance, if $\ww$ is~$\ga\ea$, the trees~$\trm{\cl\ww}$ and~$\trp{\cl\ww}$ are $\et(\et\et)$ and $(\et\et)\et$, and one finds $\wit(\ga\ea) = 3-2 = 1$. A similar argument gives $\wit(\ga\alpha) = 1$ for every address~$\alpha$. More generally, one easily checks that $\TT \leT \TT'$ implies $\MZ{\TT'} \geqslant \MZ\TT$. Hence the function~$\wit$ takes values in~$\NNNN$. Moreover, a counting argument shows that, in the previous situation, $\MZ{\TT'{}^\sigma} - \MZ{\TT^\sigma} \geqslant \MZ{\TT'} -\nobreak \MZ\TT$ holds for every substitution~$\sigma$. If $\uu$ and~$\vv$ are positive $\AAA$-words, then, as seen in the proof of Proposition~\ref{P:Action}, we have $\trm{\cl{(\uu \vv)}} = \trm{\cl\uu}^\sigma$ and $\trp{(\cl{\uu \vv})} = \trp{\cl\vv}^\tau$ for some substitutions~$\sigma, \tau$ satisfying $\trp{\cl\uu}^\sigma = \trm{\cl\vv}^\tau$. We deduce
\begin{align*}
\wit(\uu \vv) 
= \MZ{\trp{\cl{\uu \vv}}} - \MZ{\trm{\cl{\uu \vv}}}
&= \MZ{\trp{\cl\vv}^\tau} - \MZ{\trm{\cl\uu}^\sigma}
= \MZ{\trp{\cl\vv}^\tau} - \MZ{\trm{\cl\vv}^\tau} + \MZ{\trp{\cl\uu}^\sigma}- \MZ{\trm{\cl\uu}^\sigma}\\
&\geqslant \MZ{\trp{\cl\vv}} - \MZ{\trm{\cl\vv}} + \MZ{\trp{\cl\uu}}- \MZ{\trm{\cl\uu}}
= \wit(\vv) + \wit(\uu).
\end{align*}

As for~\ITEM2, the problem is to check that, whenever $\alpha, \beta, \gamma$ are addresses and the signed word $\ga\alpha\inv \ga\gamma \ga\gamma\inv \ga\beta$ is right-$\RRR$-reversible to some positive--negative word~$\vv \uu\inv$, then $\vv\inv \ga\alpha\inv \ga\beta \uu$ is right-$\RRR$-reversible to the empty word. The systematic verification seems tedious. Actually it is not. First, what matters is the mutual position of the addresses~$\alpha, \beta, \gamma$ with respect to the prefix ordering, and only finitely many patterns may occur. Next, for every pair of addresses~$\alpha, \beta$, there exists in~$\RRR$ exactly one relation of the form~$\ga\alpha ... = \ga\beta ...$~, which implies that, for every signed $\AAA$-word~$\ww$, there exists at most one pair of positive $\AAA$-words~$\uu, \vv$ such that $\ww$ is right-$\RRR$-reversible to~$\vv \uu\inv$. Finally, all instances involving quasi-commutation relations turn out to be automatically verified. So, the only critical cases are those corresponding to the triple of addresses~$\ea, 1, 11$ and its translated and permuted copies, and a direct verification is then easy. For instance, the reader can see on Figure~\ref{F:Cube} that we have $\ga\ea\inv \ga1 \ga1\inv \ga{11} \rev_{\RRR} \ga\ea^2 \ga{00}\inv \ga0\inv \ga\ea\inv \ga{01}\inv \ga1\inv$ and $\ga\ea^{-3} \ga{11} \ga1 \ga{10} \ga\ea \ga0 \ga{00} \rev_{\RRR} \eps$. 
\end{proof}

\begin{figure}[htb]
\begin{picture}(75,36)(0,0)
\put(-47,-45){\includegraphics[scale=0.97]{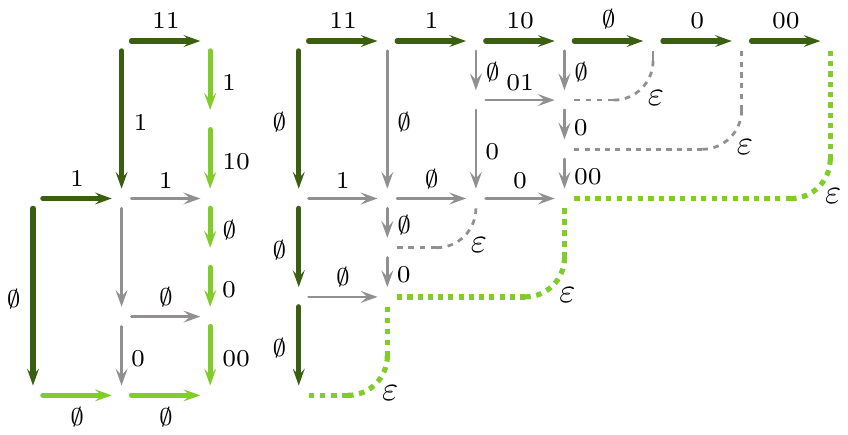}}
\end{picture}
\caption{\sf\smaller Proof of Lemma~\ref{L:Complete}: $\ga\ea\inv \ga1 \ga1\inv \ga{11}$ is right-$\RRR$-reversible to $\ga\ea^2 \ga{00}\inv \ga0\inv \ga\ea\inv \ga{01}\inv \ga1\inv$ (left), and $\ga\ea^{-3} \ga{11} \ga1 \ga{10} \ga\ea \ga0 \ga{00}$ is right-$\RRR$-reversible to the empty word; dotted lines represent the empty word\break that appears when a pattern $\ga\alpha\inv \ga\alpha$ is reversed.}
\label{F:Cube}
\end{figure}

Once a positive presentation is known to be complete with respect to right-reversing, it is easy to deduce properties of the associated monoid. 

\begin{proposition}
\label{P:LeftCancel}
The monoid~$\MON\AAA\RRR$ is left-cancellative.
\end{proposition}

\begin{proof}
By~\cite[Proposition 3.1]{Dia}, if $(\mA, \RR)$ is a positive presentation that is complete with respect to right-reversing, a sufficient condition for the monoid~$\MON\mA\RR$ to be left-cancellative is that 
\begin{equation}
\label{E:LeftCancel}
\parbox{100mm}{$\RR$ contains no relation of the form $\ma \uu = \ma \vv$ with $\ma$ in~$\mA$ and $\uu \not= \vv$.}
\end{equation}
By definition, $\RRR$ satisfies~\eqref{E:LeftCancel}. Hence the monoid~$\MON\AAA\RRR$ is left-cancellative.
\end{proof}

As for right-cancellation, no new computation is needed as we can exploit the symmetries of~$\RRR$. First, we introduce a counterpart of right-reversing where the roles of positive and negative letters are exchanged.

\index{left-reversing}
\begin{definition}
Assume that $(\mA, \RR)$ is a positive presentation. If $\ww, \ww'$ are signed $\mA$-words, we say that $\ww$ is \emph{left-$\RR$-reversible} to~$\ww'$ in one step if $\ww'$ is obtained from~$\ww$ either by deleting some length~$2$ subword~$\ma \ma\inv$ or by replacing some length~$2$ subword~$\ma \mb\inv$ with a word~$\uu\inv \vv$ such that $\uu \ma = \vv \mb$ is a relation of~$\RR$. We write $\ww \revL_\RR \ww'$ if $\ww$ is left-$\RR$-reversible to~$\ww'$ in finitely many steps.
\end{definition}

Of course, properties of left-reversing are symmetric to those of right-reversing. 

\begin{proposition}
\label{P:RightCancel}
The monoid~$\MON\AAA\RRR$ is right-cancellative.
\end{proposition}

\begin{proof}
The argument is symmetric to the one for Proposition~\ref{P:LeftCancel}, and relies on first proving that $(\AAA, \RRR)$ is, in an obvious sense, complete with respect to left-reversing. Due to the symmetries of~$\RRR$, this is easy. Indeed, for~$\ww$ a signed $\AAA$-word, let $\widetilde\ww$ denote the word obtained by reading the letters of~$\ww$ from right to left, and exchanging~$0$ and~$1$ everywhere in the indices of the letters~$\ga\alpha$. For instance, $\widetilde{\ga{110} \ga\ea}$ is $\ga\ea \ga{001}$. A direct inspection shows that the family~$\widetilde\RRR$ of all relations $\widetilde\uu = \widetilde\vv$ for $\uu = \vv$ a relation of~$\RRR$ is $\RRR$ itself. It follows that, for all signed $\AAA$-words~$\ww, \ww'$,  the relations $\ww \rev_{\RRR} \ww'$ and $\widetilde\ww \revL_{\RRR} \widetilde{\ww'}$ are equivalent. Then, as $\ww \mapsto \widetilde\ww$ is an alphabetical anti-automorphism, the completeness of~$(\AAA, \RRR)$ with respect to right-reversing implies the completeness of~$(\AAA, \widetilde\RRR)$, hence of~$(\AAA, \RRR)$, with respect to left-reversing. As the right counterpart of~\eqref{E:LeftCancel} is satisfied, we deduce that the monoid~$\MON\AAA\RRR$ is right-cancellative.
\end{proof}

In order to complete the proof of Proposition~\ref{P:DualFraction} using the strategy of Lemma~\ref{L:Strategy}, we still need to know that any two elements of the monoid~$\MON\AAA\RRR$ admit a common right-multiple. Using the action on trees, it is easy to prove that result in~$\Fs$. But this is not sufficient here as we do not know yet that $\Fs$ is isomorphic to~$\MON\AAA\RRR$. We appeal to right-reversing once more.

\begin{proposition}
\label{P:Common}
Any two elements of $\MON\AAA\RRR$ admit a common right-multiple.
\end{proposition}

\begin{proof}
If $(\mA, \RR)$ is a positive presentation, saythat right-$\RR$-reversing is \emph{terminating} if, for all positive $\mA$-words~$\uu, \vv$, there exist positive $\mA$-words~$\uu', \vv'$ satisfying $\uu\inv \vv \rev_{\RRR} \vv' \uu'{}\inv$. We noted that the latter relation implies $\uu \vv' \equivp_{\RR} \vv \uu'$, thus implying that, in the monoid~$\MON\mA\RR$, the elements represented by~$\uu$ and~$\vv$ admit a common right-multiple. So, in order to establish the proposition, it is sufficient to prove that right-$\RRR$-reversing is terminating, a non-trivial question as, because of the pentagon relations, the length of the words may increase under right-reversing, and there might exist infinite reversing sequences---try right-reversing of~$\ma\inv \mb \ma$ in the presentation $(\ma, \mb, \ma \mb = \mb^2 \ma)$.

Now, by~\cite[Proposition~3.11]{Dia}, if $(\mA, \RR)$ is a positive presentation, a sufficient condition for right-$\RR$-reversing to be terminating is that $(\mA, \RR)$ satisfies
\begin{equation}
\label{E:Terminating}
\parbox{100mm}{\ITEM1 For all~$\ma, \mb$ in~$\mA$, there is exactly one relation $\ma \pdots = \mb \pdots$ in~$\RR$, and

\ITEM2 There exists a family~$\widehat\mA$ of positive $\mA$-words that includes~$\mA$ and is \emph{closed under right-$\RR$-reversing}, this meaning that, for all~$\uu, \vv$ in~$\widehat\mA$, there exist~$\uu', \vv'$ in~$\lab\mA \cup \{\eps\}$ satisfying $\uu\inv \vv \rev_{\RRR} \vv' \uu'{}\inv$.}
\end{equation}
We claim that $(\AAA, \RRR)$ satisfies~\eqref{E:Terminating}. Indeed, \ITEM1 follows from an inspection of~$\RRR$. As for~\ITEM2, let us put 
$$\gah\alpha\rr = \ga\alpha \ga{\alpha0} \pdots \ga{\alpha0^{\rr-1}}$$ for~$\alpha$ an address and $\rr \geqslant 1$, see Figures~\ref{F:ExtPentagon} and~\ref{F:Solution} for an illustration of the action of~$\gah\alpha\rr$ on trees. Then the family~$\AAAh$ of all words~$\gah\alpha\rr$ includes~$\AAA$ as we have $\ga\alpha = \gah\alpha1$ for every~$\alpha$, and it is closed under right-$\RRR$-reversing as we find
\begin{equation}
\label{E:Closure}
\gah\beta\ms\inv \, \gah\alpha\rr \ \rev_{\RRR}\ 
\begin{cases}
\gah{0^{\ms-\rr}}{\ms - \rr}
&\mbox{for $\beta = \alpha$ with $\rr < \ms$},\\
\gah\alpha\rr \, \gah\beta\ms\inv
&\mbox{for $\beta \perp \alpha$},\\
\gah\alpha\rr \, \gah{\alpha0^{\rr+1}\gamma}\ms\inv
&\mbox{for $\beta = \alpha0\gamma$},\\
\gah\alpha\rr \, \gah{\alpha0^\rr1\gamma}\ms\inv
&\mbox{for $\beta = \alpha10^\rr\gamma$},\\
\gah\alpha\rr \, \gah{\alpha0^\ii1\gamma}\ms\inv
&\mbox{for $\beta = \alpha10^\ii1\gamma$ with $\ii < \rr$},\\
\gah\alpha{\rr+\ms} \, \gah{\alpha0^\ii}\ms\inv
&\mbox{for $\beta = \alpha10^\ii$ with $\ii < \rr$},
\end{cases}
\end{equation}
see Figure~\ref{F:Closure}. Note that $\AAAh$ is the smallest family that includes~$\AAA$ and is closed under right-$\RRR$-reversing as the last type of relation in~\eqref{E:Closure} inductively forces any such family to contain~$\gah\alpha\rr$ for every~$\rr$.
\end{proof}

\begin{figure}[htb]
\begin{picture}(75,32)(0,0)
\put(-43,-48){\includegraphics{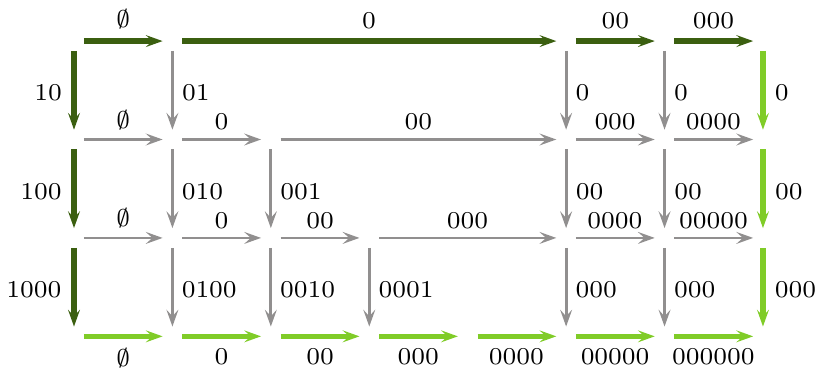}}
\end{picture}
\caption{\sf\smaller Closure of the family~$\AAAh$ under right-reversing: $\gah{10}3\inv \gah\ea4$ reverses to~$\gah\ea7 \gah03\inv$, which corresponds to the last relation in~\eqref{E:Closure} with $\alpha = \ea$, $\rr = 4$, $\ms = 3$, $\ii = 1$ (the letter ``$a$'' has been skipped everywhere).}
\label{F:Closure}
\end{figure}

In terms of the generators~$\gah\alpha\rr$, the pentagon relation can be expressed as $\ga\ea^2 = \ga1 \gah\ea2$, with both sides of length~$2$. The last type in~\eqref{E:Closure} corresponds to an extended pentagon relation $\gah\alpha\rr \, \gah{\alpha0^\ii}\ms = \gah{\alpha10^\ii}\ms \, \gah\alpha{\rr + \ms}$ for all~$\rr, \ms, \ii$ with~$\ii < \rr$, whose counterpart in terms of tree rotation is displayed in Figure~\ref{F:ExtPentagon}.

\begin{figure}[htb]
\begin{picture}(75,40)(0,2)
\put(0,0){\includegraphics{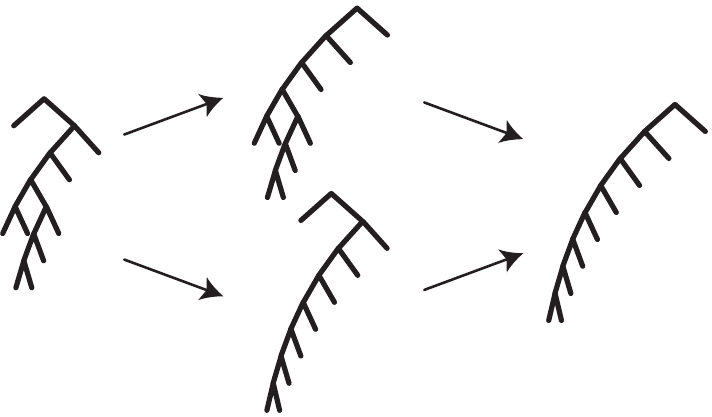}}
\put(13,10){$\gah{10^\ii}\ms$}
\put(43,10){$\gah\ea{\rr+\ms}$}
\put(14,33){$\gah\ea\rr$}
\put(45,33){$\gah{0^\ii}\ms$}
\end{picture}
\caption{\sf\smaller Extended pentagon relation $\gah\alpha\rr \, \gah{\alpha0^\ii}\ms = \gah{\alpha10^\ii}\ms \, \gah\alpha{\rr + \ms}$, here for $\alpha = \ea$, $\rr = 4$, $\ms = 3$, $\ii = 2$.}
\label{F:ExtPentagon}
\end{figure}

We thus established that the monoid~$\MON\AAA\RRR$ satisfies the conditions of  Lemma~\ref{L:Strategy} and, therefore, the proof of Proposition~\ref{P:DualFraction} is complete. 


\subsection{The lattice structure of~$\Fs$}
\label{SS:Lattice}

Here comes the central point, namely the connection between the right-divisibility relation of the monoid~$\Fs$, which we now know admits the presentation~$\MON\AAA\RRR$, and the Tamari posets. We recall that $\Div(\ff)$ denotes the family of all left-divisors of~$\ff$. 

\begin{proposition}
\label{P:Connection}
For every~$\nn \geqslant 1$, the subposet~$(\Div(\ga\ea^{\nn-1}), \dive)$ of~$(\Fs, \dive)$ is isomorphic to the Tamari poset~$(\Tam\nn, \leT)$. The poset $(\bigcup_\nn\Div(\ga\ea^\nn), \dive)$ is isomorphic to the Tamari poset~$(\Tam\infty, \leT)$. 
\end{proposition}

\begin{proof}
An immediate induction gives the equality $\CbR\nn \act \ga\ea^{\nn-1} = \CbL\nn$ for every~$\nn$, that is, the element~$\ga\ea^{\nn-1}$ of~$\Fs$ maps the right-comb~$\CbR\nn$ to the left-comb~$\CbL\nn$. Hence $\CbR\nn \act \ff$ is defined for every element~$\ff$ of~$\Fs$ that left-divides~$\ga\ea^{\nn-1}$. Thus, as $\CbR\nn$ belongs to~$\Tam\nn$ and the action of~$\Fs$ preserves the size of the trees, we obtain a well defined map
\begin{equation}
\label{E:Is}
\Is_\nn: \ff \mapsto \CbR\nn \act \ff
\end{equation}
of~$\Div(\ga\ea^{\nn-1})$ into~$\Tam\nn$. By Proposition~\ref{P:Geometry}, the map~$\Is_\nn$ is injective. On the other hand, we claim that $\Is_\nn$ is surjective. To prove it, it suffices to exhibit, for every size~$\nn$ tree~$\TT$, an element of~$\Fs$ that maps the right-comb~$\CbR\nn$ to~$\TT$. Now, for every tree~$\TT$, define two elements~$\Blue\TT, \Bluee\TT$ of~$\Fs$ by the recursive rules: 
\begin{equation}
\label{E:Blue}
\Blue\TT = \begin{cases}
1\\
\Bluee{\TT_0} \, \sh1(\Bluee{\TT_1}) \, \ga\ea
\end{cases}\ 
\Bluee\TT = \begin{cases}
1&\mbox{\quad for $\TT$ of size~$0$,}\\
\Bluee{\TT_0} \, \sh1(\Bluee{\TT_1}) \, \ga\ea
&\mbox{\quad for $\TT = \TT_0 \OP \TT_1$.}
\end{cases}
\end{equation}
For every size~$\nn$ tree~$\TT$ and every~$\pp \ge 1$, we have $\CbR\nn \act \Blue\TT = \TT$ and $\CbR{\nn + \pp} \act \Bluee\TT = \TT \OP \CbR\pp$, as shows an induction on~$\TT$: everything is obvious for~$\TT = \et$, and, for $\TT = \TT_0 \OP \TT_1$, it suffices to follow the diagrams of Figure~\ref{F:Blue}. Note that introducing both~$\Blue\TT$ and~$\Bluee\TT$ is necessary for the induction. However, the connection $\Bluee\TT = \Blue\TT \, \ga{1^{\ii-1}} \pdots \ga1 \ga\ea$, where $\ii$ is the length of the rightmost branch in~$\TT$, is easy to check.

Thus $\Is_\nn$ is a bijection of~$\Div(\ga\ea^{\nn-1})$ onto~$\Tam\nn$. Moreover, $\Is_\nn$ is compatible with the orderings. Indeed, assume $\ff \dive \mg$, say $\ff \mg' = \mg$. Then, by Proposition~\ref{P:Action}, we have $(\CbR\nn \act \ff) \act \mg' = \CbR\nn \act \mg$, whence $\CbR\nn \act \ff \leT \CbR\nn \act \mg$ by Lemma~\ref{L:Compat}. This completes the proof that $(\Div(\ga\ea^{\nn-1}), \dive)$ is isomorphic to the Tamari poset~$(\Tam\nn, \leT)$.

As for~$\Tam\infty$, we observe that, for every~$\nn$, we have $\CbR{\nn+1} = \lab{\CbR\nn}{}^\sigma$ where $\sigma$ is the substitution that maps $0 \wdots \nn-1$ to~$\et$ and $\nn$ to~$\et\et$. On the other hand, by definition, $\Div(\ga\ea^{\nn-1})$ is an initial segment of~$\Div(\ga\ea^\nn)$ and, for every~$\ff$ in~$\Div(\ga\ea^{\nn-1})$, we have 
$$\CbR{\nn+1} \act \ff = \lab{\CbR\nn}{}^\sigma \act \ff = \lab{(\CbR\nn \act \ff)}{}^\sigma,$$
hence $\Is_{\nn+1}(\ff) = \iota_\nn(\Is_\nn(\ff))$. It follows that the family~$(\Is_\nn)_{\nn \geqslant 1}$ induces a well defined map~$\Is_\infty$ of~$\bigcup_\nn \Div(\ga\ea^\nn)$ into~$\Tam\infty$. The map~$\Is_\infty$ is injective because~$\Is_\nn$ is, it is surjective as, by definition, $\Tam\infty$ is the limit of the directed system~$(\Tam\nn, \iota_\nn)$, and it preserves the orderings as $\Is_\nn$ does. 
\end{proof} 

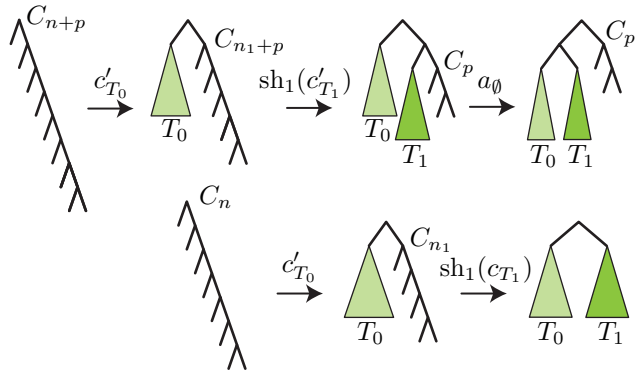
\begin{figure}[htb]
\begin{picture}(75,47)(8,2)
\put(0,0){\includegraphics{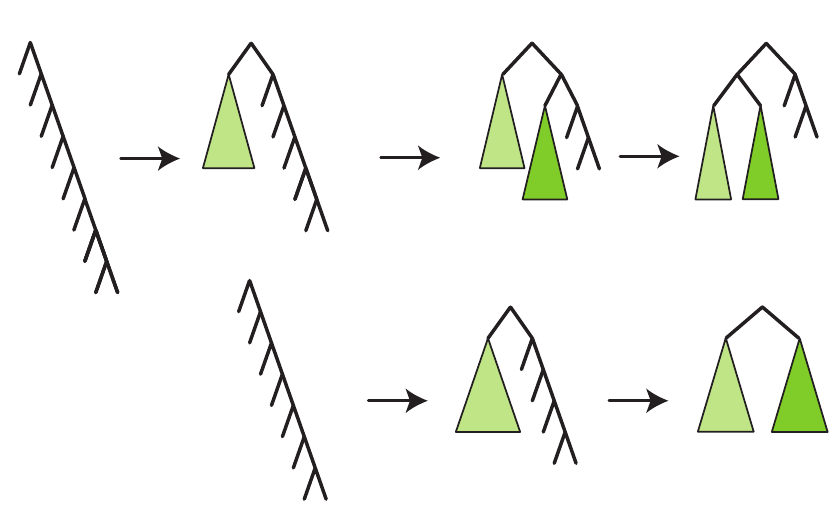}}
\put(27,23){$\CbR\nn$}
\put(55,17.5){$\CbR{\nn_1\!}$}
\put(48,5){${\TT_0}$}
\put(80,5){${\TT_1}$}
\put(72,5){${\TT_0}$}
\put(38,14){$\Bluee{\TT_0}$}
\put(59,13.5){$\sh1(\Blue{\TT_1})$}
\put(4,47){$\CbR{\nn{+}\pp}$}
\put(29,44){$\CbR{\nn_1{+}\pp}$}
\put(59.5,41){$\CbR\pp$}
\put(81,45){$\CbR\pp$}
\put(22,32){${\TT_0}$}
\put(49,32){${\TT_0}$}
\put(54,28.5){${\TT_1}$}
\put(71,28.5){${\TT_0}$}
\put(76,28.5){${\TT_1}$}
\put(13,38.5){$\Bluee{\TT_0}$}
\put(35,38.5){$\sh1(\Bluee{\TT_1})$}
\put(64,38.5){$\ga\ea$}
\end{picture}
\caption{\sf\smaller For~$\TT$ a size~$n$ tree,\break $\Blue\TT$ describes how to construct~$\TT$ from the right-comb~$\CbR\nn$, and $\Bluee\TT$ describes how to construct $\TT \OP \CbR\pp$ from~$\CbR{\nn+\pp}$; the figure illustrates the recursive definition of~$\Bluee\TT$ (above) and $\Blue\TT$ (below) for $\TT = \TT_0 \OP \TT_1$, with $\nn_1$ denoting the size of~$\TT_1$.}
\label{F:Blue}
\end{figure}

\begin{remark}
The subset $\bigcup_\nn\Div(\ga\ea^\nn)$ involved in Proposition~\ref{P:Connection} is a proper subset of~$\Fs$ as, for instance, it contains no~$\ga\alpha$ such that $0$ occurs in~$\alpha$: indeed, in this case, $\CbR\nn \act \ga\alpha$ is not defined, whereas $\CbR\nn \act \ff$ is defined for every~$\ff$ left-dividing~$\ga\ea^{\nn-1}$.
\end{remark}

The connection of Proposition~\ref{P:Connection} can be used in both directions. If we take for granted that the Tamari posets are lattices, we deduce that the subsets $(\Div(\ga\ea^{\nn-1}), \dive)$ of~$(\Fs, \dive)$ must be lattices as well, that is, with the usual terminology of left-divisibility relation, that any two elements of~$\bigcup_\nn\Div(\ga\ea^\nn)$ admit a \emph{least common right-multiple}, or \emph{right-lcm}, and a \emph{greatest common left-divisor}, or \emph{left-gcd}.  

On the other hand, if we have a direct proof that $(\Fs, \dive)$ is a lattice, then the isomorphism of Proposition~\ref{P:Connection} provides a new proof of the lattice property for the Tamari posets. This is what happens.

\begin{proposition}
\label{P:Lcm}
The poset~$(\Fs, \dive)$ is a lattice.
\end{proposition}

\begin{corollary}
For every~$\nn$, the Tamari poset~$(\Tam\nn, \leqslant)$ is a lattice.
\end{corollary}

To establish Proposition~\ref{P:Lcm}, we once again appeal to subword reversing.

\begin{proof}[of Proposition~\ref{P:Lcm}]
By~\cite[Proposition~3.6]{Dia}, if $(\mA, \RR)$ is a positive presentation that is complete with respect to right-reversing, a sufficient condition for any two elements of~$\MON\mA\RR$ that admit a common right-multiple to admit a right-lcm is that $(\mA, \RR)$ satisfies Condition~\ITEM1 of~\eqref{E:Terminating}; moreover, in this case, the right-lcm of the elements represented by two $\mA$-words~$\uu, \vv$ is represented by~$\uu \vv'$ and~$\vv \uu'$, where $\uu', \vv'$ are the positive $\mA$-words for which $\uu\inv \vv \rev_\RR \vv' \uu'{}\inv$ holds. 

Now, as already noted, $(\AAA, \RRR)$ satisfies~\eqref{E:Terminating}. Hence any two elements of~$\Fs$ that admit a common right-multiple admit a right-lcm. On the other hand, by Proposition~\ref{P:Common}, any two elements of~$\MON\AAA\RRR$, that is, of~$\Fs$, admit a common right-multiple. Hence any two elements of~$\Fs$ admit a right-lcm. In other words, any two elements in the poset~$(\Fs, \dive)$ admit a least upper bound.

As for left-gcd's, we can argue as follows. Let $\diveR$ denote the right-divisibility relation, which is the binary relation such that $\ff \diveR \mg$ holds if and only if we have $\mg' \ff  = \mg$ for some~$\mg'$ (the difference with~$\dive$ is that, here, $\ff$ appears on the right and not on the left). Then we have the derived notions of a left-lcm and a right-gcd. An easy general result says that, if $\ff, \mg, \ff', \mg'$ are elements of a monoid and satisfy $\ff \mg' = \mg \ff'$, then $\ff$ and~$\mg$ admit a left-gcd if and only if $\ff'$ and~$\mg'$ admit a left-lcm. By Proposition~\ref{P:Common}, any two elements of~$\Fs$ admit a common right-multiple and so, it suffices to show that any two elements of~$\Fs$ admit a left-lcm to deduce that they admit a left-gcd. Now, the existence of left-lcm's in~$\Fs$ follows from the properties of left-$\RRR$-reversing, which we have seen in the proof of Proposition~\ref{P:RightCancel} are similar to those of right-$\RRR$-reversing. 
\end{proof}

\begin{remark}
Another way of deducing the existence of left-gcd's from that of right-lcm's is to use Noetherianity properties. The existence of the function~$\wit$ of~\eqref{E:Wit} implies that a set~$\Div(\ff)$ contains no infinite increasing sequence~$\ff_1 \prec \ff_2 \prec \pdots$ in~$\Fs$. For all $\ff, \mg$, the family $\Div(\ff) \cap \Div(\mg)$ is nonempty as it contains~$1$, and, by Noetherianity, it contains a $\dive$-maximal element, which must be a left-gcd of~$\ff$~and~$\mg$.
\end{remark}


\subsection{Computing the operations}
\label{SS:Algo}

We conclude this section with results about the algorithmic complexity of subword reversing in~$\Fs$. Here we concentrate on space complexity, namely bounds on the length of words; it would be easy to state analogous bounds on the number of reversing steps, hence for time complexity.

\begin{proposition}
\label{P:NumDen}
If $\ww, \ww'$ are signed $\AAA$-words, $\ww \rev_{\RRR} \ww'$ implies $\Lg{\ww'} \leqslant \Lg\ww^2/4 + \Lg\ww$. More precisely, we have $\Lg{\ww'} \leqslant \pp + \qq + \pp \qq$ if $\ww$ contains $\pp$~positive letters and $\qq$~negative letters. These bounds are sharp.
\end{proposition}

\begin{proof}
By construction, the $\RRR$-reversing steps in the right-$\RRR$-reversing of~$\ww$ to~$\ww'$ can be gathered into $\RRRh$-reversing steps, which are at most~$\pp \qq$ in number. Consider the sum of the indices~$\rr$ of the involved generators~$\gah\alpha\rr$. Each $\RRR$-reversing step increases this sum by~$1$ at most (in the case of a pentagon relation), so the total sum in the final~$\pp + \qq$ generators~$\gah\alpha\rr$ is at most $\pp + \qq + \pp\qq$. So, when the generators~$\gah\alpha\rr$ are decomposed as products of~$\ga\alpha$'s, at most $\pp + \qq + \pp\qq$ of the latter occur.

The bound is sharp, as an easy  induction gives
$$(\ga{1^{\pp-1}} \pdots \ga1 \ga\ea)\inv \ 
\ga{1^\pp}^\qq \ \rev_{\RRR}\  \ga\ea^\qq \ (\gah{1^{\pp-1}}\qq \pdots \gah1\qq \gah\ea\qq)\inv,$$
a word of length~$\pp + \qq$ that
is right-$\RRR$-reversible to a word of length~$\pp + \qq + \pp\qq$.
\end{proof}

Other upper bounds can be obtained by using the action of~$\Fs$ on trees. To state the result, it is convenient to introduce the following natural terminology.

\index{denominator}
\index{numerator}
\begin{definition}
For every signed $\AAA$-word~$\ww$, the \emph{right-numerator}~$\NR(\ww)$ and the \emph{right-denominator}~$\DR(\ww)$ of~$\ww$ are the unique $\AAA$-words satisfying $\ww \rev_{\RRR} \NR(\ww) \DR(\ww)\inv$. Symmetrically, the \emph{left-numerator}~$\NL(\ww)$ and the \emph{left-denominator}~$\DR(\ww)$ of~$\ww$ are the unique $\AAA$-words satisfying $\ww \revL_{\RRR} \DL(\ww)\inv \NL(\ww)$.
\end{definition}

As left- and right-$\RRR$-reversings are terminating, the positive $\AAA$-words~$\NR(\ww)$, $\DR(\ww)$, $\NL(\ww)$, and~$\DL(\ww)$ exist for every signed $\AAA$-word~$\ww$. 

\begin{proposition}
\label{P:NumDenBis}
Assume that $\ww$ is a signed $\AAA$-word and $\TT \act \ww$ is defined for some size~$\nn$ tree~$\TT$. Then we have
\begin{equation} 
\label{E:NumDenBis}
\max(\Lg{\NL(\ww)} + \Lg{\DR(\ww)}, \Lg{\NL(\ww)} + \Lg{\DR(\ww)}) \leqslant (n-1)(n-2)/2.
\end{equation}
\end{proposition}

In order to establish Proposition~\ref{P:NumDenBis}, we need a preliminary result about the action of $\AAA$-words on trees. First, if $\TT$ is a tree and $\ww$ is a signed~$\AAA$-word, we say that $\TT \act \ww$ is defined if $\TT \act \cl\uu$ is defined for every prefix~$\uu$ of~$\ww$. Now, if two signed $\AAA$-words~$\ww, \ww'$ represent the same element of~$F$, the hypothesis that $\TT \act \ww$ is defined for some tree~$\TT$ does not guarantee that $\TT \act \ww'$ is also defined: for instance, $\TT \act \eps$ is always defined, but $\TT \act \ga\alpha\inv \ga\alpha$ is not. However, this cannot happen with reversing.

\begin{lemma} 
\label{L:ActRev}
Assume that $\ww, \ww'$ are signed-$\AAA$-words and $\ww$ is right- or left-$\RRR$-reversible to~$\ww'$. Then, for every tree~$\TT$ such that $\TT \act \ww$ is defined, $\TT\act \ww'$ is defined as well.
\end{lemma}

\begin{proof}
The problem with arbitrary equivalences is that new pairs~$\ga\alpha\inv \ga\alpha$ or~$\ga\alpha \ga\alpha\inv$ may be created. This however is impossible in the case of (right- or left-) reversing, as we can only delete such pairs, but not create them. A complete formal proof requires to check all possible cases: this is easy, and we skip the details.
\end{proof}

\begin{proof}[of Proposition~\ref{P:NumDenBis}]
Let $\TT' = \TT \act \ww$. By definition, $\ww$ is right-$\RRR$-reversible to~$\NR(\ww) \DR(\ww)\inv$, and left-$\RRR$-reversible to $\DL(\ww)\inv \NL(\ww)$. By Lemma~\ref{L:ActRev}, this implies that $\TT \act \NR(\ww) \DR(\ww)\inv$ and $\TT \act \DL(\ww)\inv \NL(\ww)$ are defined. Put $\TT_\smallL = \TT \act \DL(\ww)\inv$ and $\TT_\smallR = \TT \act \NR(\ww)$. By hypothesis, the terms~$\TT, \TT'$, $\TT_\smallL$, and~$\TT_\smallR$ all have size~$\nn$. Hence there exists a positive $\AAA$-word~$\uu$ (namely~$\Blue{\TT_\smallL}$) mapping the right comb~$\CbR\nn$ to~$\TT_\smallL$. By symmetry, there exists a positive $\AAA$-word~$\vv$ mapping~$\TT_\smallR$ to the left comb~$\CbL\nn$. Then $\uu \NL(\ww) \DR(\ww) \vv$ and $\uu \DL(\ww) \NR(\ww) \vv$ are $\RRR$-equivalent positive $\AAA$-words, and both map~$\CbR\nn$ to~$\CbL\nn$, see Figure~\ref{F:Bound}. Now $\ga\ea^{n-2}$ also maps~$\CbR\nn$ to~$\CbL\nn$. Hence, by Proposition~\ref{P:Geometry}, we must have
\begin{equation}
\ga\ea^{n-2} \equivp_{\RRR} \uu \NL(\ww) \DR(\ww) \vv \equivp_{\RRR} \uu \DL(\ww) \NR(\ww) \vv.
\end{equation}
Then the function~$\wit$ of~\eqref{E:Wit} provides an upper bound for the lengths of the words $\RRR$-equivalent to a given word. In the current case, we have $\wit(\ga\ea^{\nn-2}) = (\nn-1)(\nn-2)/2$, and \eqref{E:NumDenBis} follows.
\end{proof}

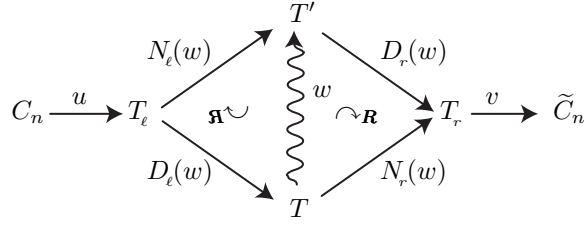
\begin{figure}[htb]
\begin{picture}(70,25)(3,3)
\put(7,3){\includegraphics{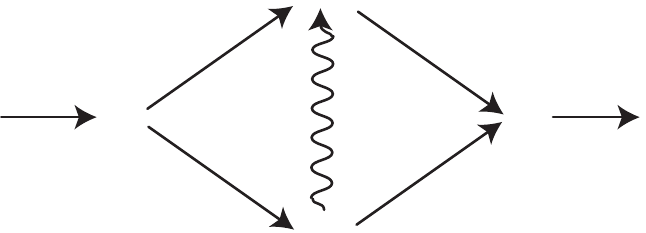}}
\put(39, 1){$\TT$}
\put(39, 27){$\TT'$}
\put(2, 14){$\CbR\nn$}
\put(74, 14){$\CbL\nn$}
\put(42, 17){$\ww$}
\put(45, 14){$\rev_{\RRR}$}
\put(28, 14){$\reflectbox{${}_{\RRR}$}\!\!\mathrel{\raisebox{5pt}{\rotatebox{180}{$\curvearrowright$}}}$}
\put(20, 22){$\NL(\ww)$}
\put(51, 22){$\DR(\ww)$}
\put(51, 6){$\NR(\ww)$}
\put(20, 6){$\DL(\ww)$}
\put(10, 16){$\uu$}
\put(65, 16){$\vv$}
\put(17.5, 14){$\TT_\smallL$}
\put(59, 14){$\TT_\smallR$}
\end{picture}
\caption{\sf\smaller Bounding the lengths of the left- and right-numerators and denominators of a signed $\AAA$-word~$\ww$ in terms of the size of a term~$\TT$ such that $\TT \act \ww$ is defined.}
\label{F:Bound}
\end{figure}

The upper bound of~\eqref{E:NumDenBis} is close to sharp: for $\ww = (\ga{1^{\pp-1}} \pdots \ga1 \ga\ea)\inv \, \ga{1^\pp}^\qq$, the word~$\DR(\ww)$ is $\gah{1^{\pp-1}}\qq \pdots \gah1\qq \gah\ea\qq$, which has length~$\pp\qq$ in the alphabet~$\AAA$, so the sum of the lengths of~$\NL(\ww)$ and~$\DR(\ww)$ is~$\pp + \pp \qq$, while the minimal size of a term~$\TT$ such that $\TT \act \ww$ is defined is~$\pp + \qq + 2$.

To conclude with subword reversing, we mention one more result that involves both left- and right-reversing. The example of~$\eps$ and~$\ga\ea \ga\ea\inv$ shows that $\RRR$-equivalent words need not have $\RRR$-equivalent numerators and denominators: the right-numera\-tor of~$\eps$ is~$\eps$, whereas the right-numerator of~$\ga\ea \ga\ea\inv$ is~$\ga\ea$. This cannot happen when left- and right-numerators are mixed in a double reversing.

\begin{proposition}
\label{P:MinFrac}
For~$\ww$ a signed $\AAA$-word, define $\NLR(\ww)$ to be $\NR(\DL(\ww)\inv \NL(\ww))$ and $\DLR(\ww)$ to be $\DR(\DL(\ww)\inv \NL(\ww))$. Then $\ww \equiv_{\RRR} \ww'$ implies $\NLR(\ww') \equivp_{\RRR} \NLR(\ww)$ and~$\DLR(\ww') \equivp_{\RRR} \DLR(\ww)$.
\end{proposition} 

We first observe that $\NLR(\ww) \DLR(\ww)\inv$ is a minimal fractionary expression of~$\cl\ww$:

\begin{lemma}
\label{L:MinFrac}
If $\ww, \ww'$ are $\RRR$-equivalent signed $\AAA$-words, there exist a positive $\AAA$-word~$\uu$ satisfying 
\begin{equation}
\label{E:MinFrac}
\NR(\ww') \equivp_{\RRR} \NLR(\ww) \, \uu \mbox{\quad and \quad} \DR(\ww') \equivp_{\RRR} \DLR(\ww) \, \uu.
\end{equation}
\end{lemma} 

\begin{proof}
By construction, the word~$\ww$ is $\RRR$-equivalent to $\DL(\ww)\inv \NL(\ww)$, and the latter word is right-$\RRR$-reversible to~$\NLR(\ww) \DLR(\ww)\inv$. Hence we have 
\begin{equation}
\label{E:MinFrac1}
\DL(\ww) \NLR(\ww) \equivp_{\RRR} \NL(\ww) \DLR(\ww),
\end{equation}
and, moreover, as mentioned in the proof of Proposition~\ref{P:Lcm}, the element of~$\Fs$ represented by~$\DL(\ww) \NLR(\ww)$ and~$\NL(\ww) \DLR(\ww)$ is the right-lcm of~$\cl{\NL(\ww)}$ and~$\cl{\DL(\ww)}$.

On the other hand, $\NR(\ww') \DR(\ww')\inv$ is $\RRR$-equivalent to~$\ww'$, hence to~$\ww$, and therefore to~$\DL(\ww)\inv \NL(\ww)\inv$. We deduce $\DL(\ww) \NR(\ww') \equiv_{\RRR} \NL(\ww) \DR(\ww')$, whence 
\begin{equation}
\label{E:MinFrac2}
\DL(\ww) \NR(\ww') \equivp_{\RRR} \NL(\ww) \DR(\ww')
\end{equation} 
since $\Fs$ embeds in~$F$. As $\cl{\DL(\ww) \NLR(\ww)}$ is the right-lcm of~$\cl{\NL(\ww)}$ and~$\cl{\DL(\ww)}$, comparing~\eqref{E:MinFrac1} and~\eqref{E:MinFrac2} implies the existence of~$\uu$ satisfying~\eqref{E:MinFrac}.
\end{proof}

\begin{proof}[of Proposition~\ref{P:MinFrac}]
By Lemma~\ref{L:MinFrac}, there exist positive $\AAA$-words~$\uu$ and~$\uu'$ satisfying $\NLR(\ww') \equivp_{\RRR} \NLR(\ww) \uu$ and $\NLR(\ww) \equivp_{\RRR} \NLR(\ww') \uu'$, whence $\NLR(\ww) \equivp_{\RRR} \NLR(\ww) \uu \uu'$. As $\Fs$ is left-cancellative, we deduce $\eps \equivp_{\RRR} \uu \uu'$. The only possibility is then that $\uu$ and~$\uu'$ are empty.
\end{proof}

We now return to Tamari lattices, and show how to use right-reversing to compute lowest upper bounds in the Tamari poset~$\Tam\nn$ appealing to the words~$\Blue\TT$ of~\eqref{E:Blue}. Of course, left-reversing can be used symmetrically to compute greatest lower bounds.

\begin{proposition}
\label{P:Lub}
Assume that $\TT, \TT'$ are size~$\nn$ trees. Then the least upper bound~$\TT''$ of~$\TT$ and~$\TT'$ in the Tamari lattice~$\Tam\nn$ is determined by
$$\TT'' =  \TT \act \NR(\Blue\TT\inv \Blue{\TT'}) = \TT' \act \DR(\Blue\TT\inv \Blue{\TT'}).$$
\end{proposition}

\begin{proof}
As mentioned in the proof of Proposition~\ref{P:Lcm}, the words~$\Blue\TT \, \NR(\Blue\TT\inv \Blue{\TT'})$ and $\Blue{\TT'} \, \DR(\Blue\TT\inv \Blue{\TT'})$ both represent the right-lcm of~$\cl{\Blue\TT}$ and~$\cl{\Blue{\TT'}}$ in~$\Fs$. Hence, owing to Proposition~\ref{P:Connection}, the image of $\cl{\Blue\TT \, \NR(\Blue\TT\inv \Blue{\TT'})}$ under~$\II_\nn$, which, by definition, is $\CbR\nn \act \Blue\TT \, \NR(\Blue\TT\inv \Blue{\TT'})$, that is, $\TT \act \NR(\Blue\TT\inv \Blue{\TT'})$, is the least upper bound in~$\Tam\nn$ of~$\II_\nn(\cl{\Blue\TT})$, that is, of~$\CbR\nn \act \Blue\TT$, which is~$\TT$, and $\II_\nn(\cl{\Blue{\TT'}})$, that is, of~$\CbR\nn \act \Blue{\TT'}$, which is~$\TT'$.
\end{proof}

\begin{example}
\label{X:Lub}
Let $\TT = \et(((\et(\et\et))\et)\et)$ and $\TT' = (\et(\et\et))(\et(\et\et))$. Using the method explained in Lemma~\ref{L:PolBlue} below, one obtains $\Blue\TT = \ga{11} \ga1^2$ and~$\Blue{\TT'} = \ga1 \ga\ea$. Right-reversing $\ga1^{-2} \ga{11}\inv \ga1 \ga\ea$ leads to $\ga{100} \ga\ea \ga0 \ga{00} \ga\ea^{-2}$ (see Figure~\ref{F:Lub}), and we deduce that the least upper bound of~$\TT$ and~$\TT'$ in the Tamari poset is the tree $\TT \act \ga{100} \ga\ea \ga0 \ga{00}$, namely $(((\et(\et\et))\et)\et)\et$ (which is also $\TT' \act \ga\ea^2$).
\end{example}

\begin{figure}[htb]
\begin{picture}(65,28)(0,0)
\put(-52,-50){\includegraphics[scale=1]{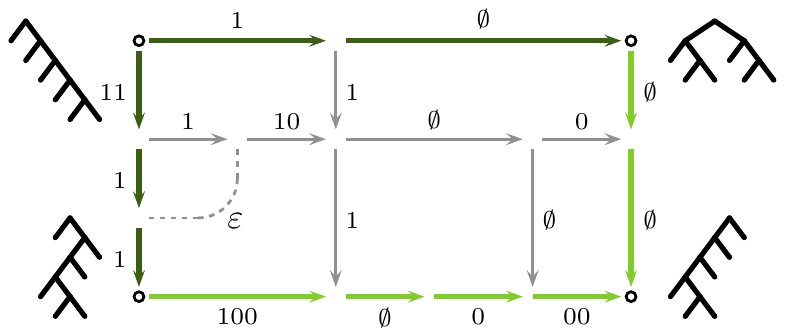}}
\end{picture}
\caption{\sf\smaller Computing the right-lcm of~$\Blue\TT$ and~$\Bluee\TT$ by right-reversing determines the least upper bound of~$\TT$ and~$\TT'$ in the Tamari lattice, here for $\TT = \et(((\et(\et\et))\et)\et)$ and $\TT' = (\et(\et\et))(\et(\et\et))$.}
\label{F:Lub}
\end{figure}


\section{The Polish normal form on~$F$}
\label{S:Polish}

We now develop another approach for determining least common multiples in the monoid~$\Fs$, whence, equivalently, least upper bounds in the Tamari lattices, namely using what is known as the Polish algorithm. Initially introduced in the case of the self-distributivity law \cite[Chapter~IX]{Dgd}, the latter is easily adapted to our current context, where it provides a unique normal form for the elements of~$\Fs$ and a method for determining the upper bound of two trees in the Tamari lattice. The main technical tool here is the covering relation of~\cite{Dhw}, a variant of the weight sequences of~\cite{Pal1}---also see~\cite{Pal0} and~\cite{BaP}.

The section is organized as follows. In Subsection~\ref{SS:Encoding} we recall the standard Polish encoding of trees and its connection with the Tamari ordering. In Subsection~\ref{SS:PolAlgo} we describe an algorithm that, starting with the Polish encoding of two trees, determines a common upper bound of the latter in the Tamari lattice together with a distinguished way of performing the rotations. Then, in Subsection~\ref{SS:Covering}, we use the covering relation to control the previous algorithm and, in particular, prove that it always determines the least upper bound of the initial trees. Finally, in Subsection~\ref{SS:NormalForm}, we deduce a unique normal form for the elements of~$F$ that enjoys a sort of weak rationality property.


\subsection{The Polish encoding of trees}
\label{SS:Encoding}

As is well known, trees or, equivalently, parenthesized expressions can be encoded without parentheses using the Polish notation. Here we consider the right version, and use~$\OPP$ as the operation symbol.

\index{Polish!encoding}
\begin{definition}
For~$\TT$ a tree, the \emph{(right)-Polish encoding} of~$\TT$ is the word~$\Pol\TT$ recursively defined by $\Pol\TT = \TT$ if $\TT$ has size~$0$, and $\Pol\TT = \Pol{\TT_0} \, \Pol{\TT_1} \OPP$ for $\TT = \TT_0 \OP \TT_1$.
\end{definition}

\index{origin}
For~$\TT$ of size~$\nn$, the Polish encoding~$\Pol\TT$ is a word of length~$2\nn+1$, which we consider as a map of~$\{1 \wdots 2\nn + 1\}$ into~$\{\et, \OPP\}$ (when we restrict to unlabeled trees): thus $\Pol\TT(\kk)$ refers to the $\kk$th letter of the word~$\Pol\TT$. There exists a natural one-to-one \emph{origin} function from the positions of the letters of~$\Pol\TT$ to the addresses of the nodes of~$\TT$, recursively defined for~$\TT = \TT_0 \OP \TT_1$ with $\TT_0$ of size~$\nn_0$ by the rule that the origin of~$\kk$ in~$\TT$ is $0\alpha$ where $\alpha$ is the origin of~$\kk$ in~$\TT_0$ for $\kk \leqslant 2\nn_0+1$, it is $1\alpha$ where $\alpha$ is the origin of~$\kk-2\nn_0-1$ in~$\TT_1$ for $2\nn_0+1 < \kk \leqslant 2\nn$, and it is $\ea$ for $\kk = 2\nn+1$. For instance, the Polish encoding of the tree $\et((\et\et)\et)$ of Figure~\ref{F:Trees} is $\et \et \et \OPP \et \OPP \OPP$, and the corresponding origins are $\et_{(0)} \et_{(100)} \et_{(101)} \OPP_{(10)} \et_{(11)} \OPP_{(1)} \OPP_{(\ea)}$.

For our current purpose, it is important to note the following connection between the Polish encoding and the Tamari order.

\begin{lemma}
Let $\lLex$ denote the lexicographical extension of the ordering~$\et < \OPP$ to $\{\et, \OPP\}$-words. Then, for all trees~$\TT, \TT'$, the relation $\TT \leT \TT'$ implies $\Pol\TT \leLex \Pol{\TT'}$.
\end{lemma}

\begin{proof}
When translated to the right Polish notation, applying a left-rotation in a tree amounts to replacing some subword of the form $\Pol{\TT_0} \, \Pol{\TT_1} \, \Pol{\TT_2} \OPP \OPP$ with the corresponding word $\Pol{\TT_0} \, \Pol{\TT_1} \OPP \Pol{\TT_2} \OPP$. The latter word is $\lLex$-larger than the former, as the beginning of the word is preserved, until the first letter~$\et$ associated with~$\Pol{\TT_2}$, which is replaced with~$\OPP$.  
\end{proof}

When the initial letter~$\et$ is erased, the words that are the Polish encoding of a trees identify with Dyck words, defined as those words in the alphabet~$\{\et, \OPP\}$ such that no initial segment has more~$\OPP$'s than~$\et$'s, see for instance~\cite{Stn}. Using the standard correspondence between such words and random walks in~$\NNNN^2$, we obtain a simple receipe for determining the elements~$\Blue\TT$ and~$\Bluee\TT$ of~\eqref{E:Blue} from~$\Pol\TT$. 

\begin{lemma}
\label{L:PolBlue}
(See Figure~\ref{F:PolBlue}.) Assume that $\TT$ is a size~$\nn$ tree. For~$\kk$ in~$\{1 \wdots 2\nn+1\}$ recursively define~$\nu_\TT(\kk)$ by $\nu_\TT(1) = -1$ and, for $\kk \geqslant 2$,
\begin{equation}
\label{E:Level}
\nu_\TT(\kk) = \begin{cases}
\nu_\TT(\kk-1) + 1 
&\mbox{for $\Pol\TT(\kk-1) = \Pol\TT(\kk) = \et$},\\
\nu_\TT(\kk-1) - 1 
&\mbox{for $\Pol\TT(\pp-1) = \Pol\TT(\pp) = \OPP$},\\
\nu_\TT(\kk-1)
&\mbox{otherwise}.
\end{cases}
\end{equation}
Then $\Bluee\TT$ is obtained from~$\Pol\TT$ by replacing each letter~$\Pol\TT(\kk)$ with~$\eps$ if it is~$\et$ and with~$\ga{1^\ii}$ with $\ii = \nu_\TT(\kk)$ if is it~$\OPP$; the word~$\Blue\TT$ is obtained similarly after erasing the last block of~$\OPP$.
\end{lemma}

\begin{figure}[htb]
\begin{picture}(70,32)(0,2)
\put(-44,-43){\includegraphics{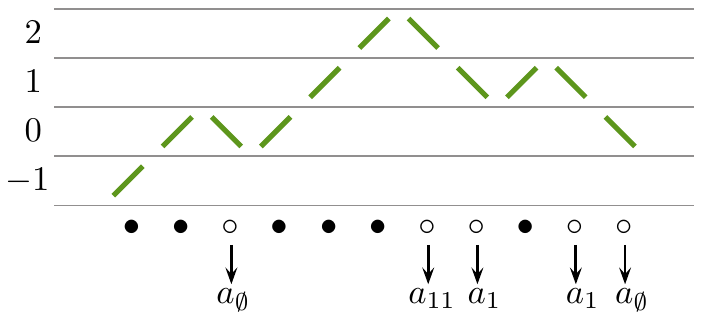}}
\end{picture}
\caption{\sf\smaller Computing $\Blue\TT$ and~$\Bluee\TT$ from the Polish encoding~$\Pol\TT$ of~$\TT$: write the $\kk$th letter of~$\Pol\TT$ at level $\nu_\TT(\kk)$; then $\Bluee\TT$ is read from the levels of the letters~$\OPP$. Here, for $(\et\et)((\et(\et\et))\et)$, we read $\Bluee\TT = \ga\ea \ga{11}\ga1^2\ga\ea$, and, discarding the last two symbols~$\OPP$, $\Blue\TT = \ga\ea \ga{11}\ga1$.}
\label{F:PolBlue}
\end{figure}

We skip the verification, a comparison of the recursive definitions of~$\Pol\TT$ and~$\Bluee\TT$.


\subsection{The Polish algorithm}
\label{SS:PolAlgo}

Assume that $\TT, \TT'$ are trees of size~$\nn$ and we look for a (minimal) tree~$\TT''$ that is an upper bound of~$\TT$ and~$\TT'$ in the Tamari order. If $\TT$ and~$\TT'$ do not coincide, then one of the words~$\Pol\TT$, $\Pol{\TT'}$ is lexicographically smaller than the other, say for instance~$\Pol\TT$. This means means that there exists~$\kk$ such that $\Pol\TT(\kk)$ is~$\et$, whereas $\Pol{\TT'}(\kk)$  is~$\OPP$. In this case, we shall say that $\TT$ and~$\TT'$ have a \emph{clash at~$\kk$}. Here is the point. 

\begin{lemma}
\label{L:Solution}
Assume that $\TT$ is a tree and that the $\kk$th letter in~$\Pol\TT$ is~$\et$. Then there exists at most one pair~$(\alpha, \rr)$ such that $\TT \act \gah\alpha\pp$ is defined and $\TT$ and $\TT \act \gah\alpha\rr$ have a clash at~$\kk$. Moreover, if there exists~$\TT''$ such that $\TT$ and~$\TT''$ have a clash at~$\kk$, there exists exactly one pair as above.
\end{lemma}

\begin{proof}
As Figure~\ref{F:Solution} shows, if we have $\TT' = \TT \act \gah\alpha\rr$, then the words~$\Pol\TT$ and~$\Pol{\TT'}$ coincide up to the first letter coming from the $\alpha10^{\rr-1}1$-subtree of~$\TT$: the latter is~$\et$ (as is always the first letter of a Polish encoding), whereas, in~$\Pol{\TT'}$, we have a letter~$\OPP$ at this position. Thus, the action of~$\gah\alpha\rr$ on~$\Pol\TT$ is to replace~$\et$ by~$\OPP$ at a position whose origin in~$\TT$ has the form~$\alpha10^{\rr-1}10^\ii$ for some~$\ii$. 

Consider now the $\kk$th letter in~$\Pol\TT$, supposed to be a letter~$\et$. The origin of~$\kk$ in~$\TT$ is a certain address of leaf in~$\TT$, say~$\beta$. By the above argument, a pair~$(\alpha, \rr)$ may result in a clash at~$\kk$ only if we can write $\beta = \alpha10^{\rr-1}10^\ii$ for some~$\rr \geqslant 1$ and~$\ii \geqslant 0$. For every~$\beta$, this happens for at most one pair~$(\alpha, \rr)$, and this happens if and only if $\beta$ contains at least two digits~$1$. 

Assume now that $\TT$ and $\TT''$ have a clash at~$\kk$, and consider the value of~$\nu_\TT(\kk)$ as defined in~\eqref{E:Level}. By construction (and by the standard properties of Dyck words), we have $\nu_{\TT''}(\kk) \geqslant 0$ as $\Pol{\TT''}(\kk)$ is~$\OPP$. By construction, we have $\nu_\TT(\kk) > \nu_{\TT''}(\kk)$ since $\Pol\TT(\kk)$ is~$\et$, whence $\nu_\TT(\kk) \geqslant 1$. This implies (actually an equivalence) that the address~$\beta$ contains at least two digits~$1$. Hence there exists a pair~$(\alpha, \rr)$ as above.
\end{proof}

\begin{figure}[htb]
\begin{picture}(75,48)(2,-2)
\put(0,0){\includegraphics{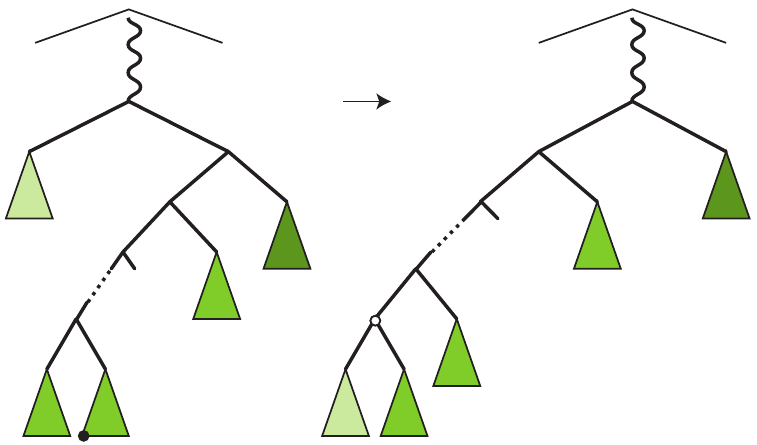}}
\put(14,45){$\TT$}
\put(65,45){$\TT'$}
\put(1,20){${\TT_0}$}
\put(2.5,-2){${\TT_1}$}
\put(9,-2){${\TT_2}$}
\put(20,9.5){${\TT_\rr}$}
\put(26,15){${\TT_{\rr{+}1}}$}
\put(35,38){$\gah\alpha\rr$}
\put(15,35){$\scriptstyle\alpha$}
\put(66,35){$\scriptstyle\alpha$}
\put(25,30){$\scriptstyle\alpha1$}
\put(11,24){$\scriptstyle\alpha10$}
\put(-1.5,7){$\scriptstyle\alpha10^\rr$}
\put(-0.5,13){$\scriptstyle\alpha10^{\rr-1}$}
\put(11.5,7){$\scriptstyle\alpha10^{\rr-1}1$}
\put(33,-2){${\TT_0}$}
\put(39.5,-2){${\TT_1}$}
\put(44.5,3){${\TT_2}$}
\put(58.5,14.5){${\TT_\rr}$}
\put(71,20){${\TT_{\rr{+}1}}$}
\put(7.6,3){$\pmb{\downarrow}$}
\put(37.2,14.5){$\pmb{\downarrow}$}
\color{Medium}
\put(4.5,20.5){$\scriptstyle\uparrow$}
\put(4,18){$\scriptstyle\jj_0$}
\put(6.2,-1.5){$\scriptstyle\uparrow$}
\put(5.7,-4){$\scriptstyle\jj_1$}
\put(12.2,-1.5){$\scriptstyle\uparrow$}
\put(12,-4){$\scriptstyle\jj_2$}
\put(23.5,10.5){$\scriptstyle\uparrow$}
\put(23.5,8){$\scriptstyle\jj_\rr$}
\put(36.7,-1.5){$\scriptstyle\uparrow$}
\put(36.5,-4){$\scriptstyle\jj_0$}
\put(42.7,-1.5){$\scriptstyle\uparrow$}
\put(42.5,-4){$\scriptstyle\jj_1$}
\put(48,3.5){$\scriptstyle\uparrow$}
\put(47.7,1){$\scriptstyle\jj_2$}
\put(62.2,15.5){$\scriptstyle\uparrow$}
\put(62,13){$\scriptstyle\jj_\rr$}
\end{picture}
\caption{\sf\smaller Action of~$\gah\alpha\rr$: the Polish encodings coincide up to the first~$\et$ corresponding to~$\TT_2$ in~$\Pol\TT$ (black arrow); the latter is replaced with~$\OPP$ in~$\Pol{\TT'}$ because, in~$\TT'$, there is one more right-edge after~$\TT_1$ than in~$\TT$, and the clash occurs between the marked letters.}
\label{F:Solution}
\end{figure}

Now the principle of an algorithm should be clear: starting with two trees~$\TT, \TT'$ such that the Polish encoding~$\Pol\TT$ and~$\Pol{\TT'}$ coincide up to position~$\kk-1$, we have found a unique way of applying an iterated left-rotation~$\gah\alpha\rr$ to one of the trees so that the clash is moved further to the right. By iterating the process, we obtain after finitely many steps two trees whose Polish encodings coincide, that is, we obtain a common upper bound for the initial trees~$\TT, \TT'$. 

\begin{definition}
\label{D:Solution}
Assume that $\TT, \TT'$ are trees of equal size.

\ITEM1 If $\Pol\TT \lLex \Pol{\TT'}$ holds, we denote by~$\Sol\TT{\TT'}$ the unique element~$\gah\alpha\rr$ such that $\TT \act \gah\alpha\rr$ and~$\TT'$ have no clash at the position where $\TT$ and~$\TT'$ have one.

\ITEM2 We denote by~$\SOL\TT{\TT'}$ the signed $\AAAh$-word recursively defined by the rules
\begin{equation}
\SOL\TT{\TT'} = 
\begin{cases}
\eps&\mbox{for $\TT = \TT'$},\\
\Sol\TT{\TT'} \, \SOL{\TT \act \Sol\TT{\TT'}}{\TT'}
&\mbox{for $\Pol\TT \lLex \Pol{\TT'}$},\\
\SOL{\TT}{\TT' \act \Sol{\TT'}\TT} \, \Sol{\TT'}\TT\inv
&\mbox{for $\Pol\TT \gLex \Pol{\TT'}$}.
\end{cases} 
\end{equation}
\end{definition}

\begin{example}
\label{X:Solution}
Let us consider the trees of Example~\ref{X:Lub} again, namely
$\TT_0 = \et(((\et(\et\et))\et)\et)$ and $\TT'_0 = (\et(\et\et))(\et(\et\et))$. We find

$\Pol{\TT_0} = \et \et \et \underline{\color{Medium}\et}\OPP \OPP \et \OPP \et \OPP \OPP$,

$\Pol{\TT'_0} = \et\et\et \OPP \OPP \et \et \et \OPP \OPP \OPP$.

\noindent Thus we have $\Pol{\TT_0} \lLex \Pol{\TT'_0}$, with a clash at~$4$ (underlined). The origin of~$4$ in~$\TT_0$ is $10011$, whence $\Sol{\TT_0}{\TT'_0} = \ga{100}$, and $\SOL{\TT_0}{\TT'_0} = \ga{100} \; \SOL{\TT_1}{\TT'_1}$ with $\TT_1 = \TT_0 \act \ga{100}$ and $\TT'_1 = \TT'_0 $, corresponding to 

$\Pol{\TT_1} = \et \et \et \OPP \underline{\color{Medium}\et} \OPP \et \OPP \et \OPP \OPP$,

$\Pol{\TT'_1} = \et\et\et \OPP \OPP \et \et \et \OPP \OPP \OPP$.

\noindent  We have now $\Pol{\TT_1} \lLex \Pol{\TT'_1}$, with a clash at~$5$. The origin of~$5$ in~$\TT_1$ is $1001$, whence $\Sol{\TT_1}{\TT'_1} = \gah\ea3$, and $\SOL{\TT_1}{\TT'_1} = \gah\ea3 \; \SOL{\TT_2}{\TT'_2}$ with $\TT_2 = \TT_1 \act \gah\ea3$ and $\TT'_2 = \TT'_1$, hence

$\Pol{\TT_2} = \et \et \et \OPP \OPP \et \OPP \et \OPP \et \OPP$,

$\Pol{\TT'_2} = \et\et\et \OPP \OPP \et \underline{\color{Medium}\et} \et \OPP \OPP \OPP$.

\noindent This time, we have $\Pol{\TT'_2} \lLex \Pol{\TT_2}$, with a clash at~$7$. The origin of~$7$ in~$\TT'_2$ is $110$, so $\Sol{\TT'_2}{\TT_2}$ is~$\ga\ea$, and $\SOL{\TT_2}{\TT'_2} = \SOL{\TT_3}{\TT'_3} \; \ga\ea\inv$ with $\TT_3 = \TT_2$ and $\TT'_3 = \TT'_2 \act \ga\ea$, that is, 

$\Pol{\TT_3} = \et \et \et \OPP \OPP \et \OPP \et \OPP \et \OPP$,

$\Pol{\TT'_3} = \et\et\et \OPP \OPP \et \OPP \et \underline{\color{Medium}\et} \OPP\OPP$.

\noindent We find now $\Pol{\TT'_3} \lLex \Pol{\TT_3}$, with a clash at~$9$. The origin of~$9$ in~$\TT'_3$ is $11$, whence $\Sol{\TT'_3}{\TT_3} = \ga\ea$, and $\SOL{\TT_3}{\TT'_3} = \SOL{\TT_4}{\TT'_4} \; \ga\ea\inv$ with $\TT_4 = \TT_3$ and $\TT'_4 = \TT'_3 \act \ga\ea$, that is, 

$\Pol{\TT_4} = \et \et \et \OPP \OPP \et \OPP \et \OPP \et \OPP$,

$\Pol{\TT'_4} = \et \et \et \OPP \OPP \et \OPP \et \OPP \et \OPP$.

\noindent We have $\TT_4 = \TT'_4$, so the algorithm halts. The tree~$\TT_4$ is a common upper bound of~$\TT_0$ and~$\TT'_0$, and the word~$\SOL{\TT_0}{\TT'_0}$ is $\ga{100} \gah\ea3 \ga\ea^{-2}$.
 \end{example}
 
Thus, for all equal size trees~$\TT, \TT'$, we obtained a distinguished signed $\AAAh$-word~$\SOL\TT{\TT'}$, and, by construction, the relation $\TT' = \TT \act \SOL\TT{\TT'}$ is satisfied. 

\begin{remark}
As mentioned in the beginning of the section, an entirely similar algorithm can be defined with the self-distributivity law $\xx(\yy\zz) = (\xx\yy) (\xx\zz)$ replacing the associativity law $\xx(\yy\zz) = (\xx\yy) \zz$. Then tree rotations are replaced with distributions, which consist in replacing subtrees $\TT_0 \OP (\TT_1 \OP \TT_2)$ with $(\TT_0 \OP \TT_1) \OP (\TT_0 \OP \TT_2)$. In this case, the size of the trees is changed by the transformations, and termination becomes problematic. Actually, in spite of experimental evidence~\cite{Dei} and positive partial results~\cite{Dgd}, the question, which seems to be extremely difficult, remains open.
\end{remark}


\subsection{The covering relation}
\label{SS:Covering}

For the moment, we have no connection between the common upper bound of two trees provided by the Polish algorithm of Subsection~\ref{SS:PolAlgo} and their least upper bound in the Tamari lattice. In particular, if $\TT, \TT'$ are trees satisfying $\TT \leT \TT'$, it is not a priori clear that the Polish algorithm terminates with the pair~$(\TT', \TT')$, that is, the clashes always occur on the first of the two current trees. We shall see now that this is actually true. The main tool will be the covering relation, a binary relation that provides a description of the shape of a tree in terms of the addresses of its leaves. We recall that, if $\TT$ is a size~$\nn$ tree, $\lab\TT$ denotes the labeled tree obtained by attributing to the leaves of~$\TT$ labels~$0$ to~$\nn$ from left to right. So, for instance, for $\TT = \et((\et\et)\et)$, we have $\lab\TT = \et_0((\et_1\et_2)\et_3)$, and $\Pol{\lab\TT} = \et_0 \et_1 \et_2 \OPP \et_3 \OPP \OPP$.

\index{covering!relation}
\begin{definition}
\label{D:Covering}
(See Figure~\ref{F:Covering}.) Assume that $\TT$ is a size~$\nn$ tree. For $0 \leqslant \ii \leqslant \nn$, we define $\add\TT\ii$ to be the origin of~$\et_\ii$ in~$\Pol{\lab\TT}$. Then, for $\jj > \ii$, we say that $\jj$ \emph{covers}~$\ii$ in~$\TT$, written $\jj \COV\TT \ii$, if there exists an address~$\gamma$ such that $\add\TT\jj$ has the form~$\gamma1^\pp$ for some positive~$\pp$ and $\add\TT\ii$ begins with~$\gamma0$. We write $\jj \COVe\TT \ii$ for ``$\jj \COV\TT \ii$ or $\jj = \ii$''.
\end{definition}

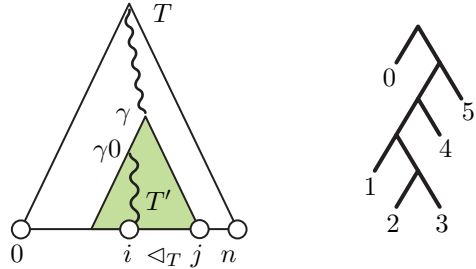
\begin{figure}[htb]
\begin{picture}(65,32)(0,0)
\put(0,0){\includegraphics{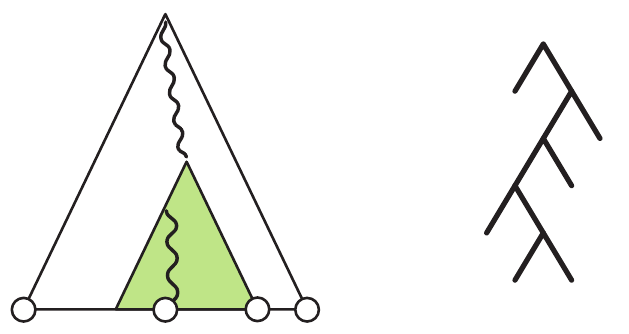}}
\put(20,30){$\TT$}
\put(19,5){$\TT'$}
\put(15,17){$\gamma$}
\put(12,12){$\gamma0$}
\put(1,-2){$0$}
\put(16,-2){$\ii$}
\put(25,-2){$\jj$}
\put(19,-2){$\vartriangleleft_\TT$}
\put(29,-2){$\nn$}
\put(50.5,21.5){$0$}
\put(48,7){$1$}
\put(51,2){$2$}
\put(57.5,2){$3$}
\put(58,12){$4$}
\put(61,17){$5$}
\end{picture}
\caption{\sf\smaller Covering relation of~$\TT$: the leaves are numbered~$0$ to~$\nn$, and $\jj$ covers~$\ii$ in~$\TT$ if there exists a subtree~$\TT'$ such that $\jj$ is the last (rightmost) label in~$\TT'$, whereas $\ii$ is a non-final label in~$\TT'$. For instance, in the right hand tree, $4$ covers~$1$, $2$, $3$, but does not cover~$0$, and $3$ covers~$1$ and~$2$, whereas $2$ covers nobody.}
\label{F:Covering}
\end{figure}

It is easily seen~\cite{Dhw} that, for every~$\jj$ occurring in a tree~$\TT$, the set of all~$\ii$'s covered by~$\jj$ is either empty or is an interval ending in~$\jj-1$: if $\jj \COV\TT \ii$ and $\jj > \ii' \geqslant \ii$ hold, then so does $\jj \COV\TT \ii'$. Also, the relation~$\COV\TT$ is transitive, and it determines~$\TT$. We shall need the more precise result that every initial fragment of the covering relation determines the corresponding initial fragment of the Polish encoding of~$\TT$.

\begin{lemma}
\label{L:Weight}
Assume that $\TT$ is a size~$\nn$ tree. Then, for $1 \leqslant \jj \leqslant \nn+1$, the number of symbols~$\OPP$ following the $\jj$th letter~$\et$ in~$\Pol\TT$ is the number of~$\ii$'s satisfying $\jj \COV\TT \ii$ and $\kk \not\COV\TT \ii$ for $\jj > \kk > \ii$.
\end{lemma}

\begin{proof}
Write $\jj \COVs\TT \ii$ if we have $\jj \COV\TT \ii$ and $\kk \not\COV\TT \ii$ for $\jj > \kk > \ii$. Then, by definition, $\jj \COVs\TT \ii$ holds if and only we have $\add\TT\jj = \alpha1^\qq$ and $\add\TT\ii = \alpha01^\pp$ for some~$\alpha$ and some~$\pp, \qq \geqslant 0$. Indeed, if we have $\add\TT\ii = \alpha0\beta$ with~$\beta$ containing at least one~$0$, say $\beta = 1^\pp0\gamma$, then we have $\kk \COV\TT \ii$ for~$\kk$ satisfying $\add\TT\kk = \alpha01^\pp01^\rr$.

On the other hand, an induction shows that the $\jj$th letter~$\et$ in~$\Pol\TT$ is followed by $\rr$~letters~$\OPP$ if and only if the address~$\add\TT\jj$ has the form~$\gamma1^\rr$ for some~$\gamma$ that does not finish with~$1$, that is, is empty or finishes with~$0$. 

Now, assume that $\add\TT\jj$ is $\gamma1^\rr$. For $0 \leqslant \mm < \rr$, let $\ii_\mm$ be the (unique) position whose address in~$\TT$ has the form~$\gamma1^\mm01^\qq$ for some~$\qq$. By the above characterization, we have $\jj \COVs\TT \ii_\mm$. So the number of~$\ii$'s satisfying $\jj \COVs\TT \ii$ is at least~$\rr$. 

Conversely, assume that there are $\rr$ different positions $\ii_0 < \pdots < \ii_{\rr-1}$ satisfying $\jj \COVs\TT \ii_\mm$. By the above characterization, there exist $\alpha_0 \wdots \alpha_{\rr-1}$ satisfying $\add\TT{\ii_\mm} = \alpha_\mm01^*$ and $\add\TT\jj = \alpha1^*$ As the numbers~$\ii_\mm$ are pairwise distinct, so are the addresses~$\alpha_\mm$ and, therefore, we have $\add\TT\jj = \alpha_0 1^{\rr'}$ with $\rr' \geqslant \rr$.
\end{proof}

It directly follows from Lemma~\ref{L:Weight} that the covering relation of a tree~$\TT$ determines the Polish encoding of~$\TT$, hence $\TT$~itself. Actually, the lemma shows more.

\begin{lemma}
\label{L:CovLex}
Assume that $\TT, \TT'$ are equal size trees, and (as a set of pairs) $\COV\TT$ is properly included in~$\COV{\TT'}$. Then $\Pol\TT \lLex \Pol{\TT'}$ holds.
\end{lemma}

\begin{proof}
Let $\jj$ be minimal such that there exists~$\ii$ satisfying $\jj \COV{\TT'} \ii$ but not $\jj \COV\TT \ii$. For $\kk < \jj$, the restriction of the covering relations~$\COV\TT$ and~$\COV{\TT'}$ to the interval~$[1, \kk]$ coincide and, therefore, by Lemma~\ref{L:Weight}, the numbers of symbols~$\OPP$ following~$\et_{\kk-1}$ in~$\Pol{\lab\TT}$ and~$\Pol{\lab{\TT'}}$ are equal. So, up to~$\et_{\jj-1}$, the words~$\Pol\TT$ and~$\Pol{\TT'}$ coincide. 

Consider now~$\et_{\jj-1}$. We claim that the number~$\rr'$ of~$\OPP$ following~$\et_{\jj-1}$ in~$\Pol{\lab{\TT'}}$ is larger than its counterpart~$\rr$ in~$\Pol{\lab\TT}$, resulting in a clash between~$\Pol\TT$ and~$\Pol{\TT'}$ and in the inequality~$\Pol\TT \lLex \Pol{\TT'}$. To see that $\rr' > \rr$ holds, we use Lemma~\ref{L:Weight} again. Using~$\COVs\TT$ as in the proof of Lemma~\ref{L:Weight}, we note that $\jj \COVs\TT \ii$ implies $\jj \COVs{\TT'}\ii$ as the restrictions of~$\COV\TT$ and~$\COV{\TT'}$ to~$[1, \jj-1]$ coincide. By hypothesis, there are $\rr$~values of~$\ii$ satisfying $\jj \COVs\TT \ii$, and these values also satisfy~$\jj \COVs{\TT'} \ii$. Now, by hypothesis, there exists~$\ii'$ satisfying $\jj \not\COV\TT \ii'$ and $\jj \COV{\TT'} \ii'$. If $\ii'$ is chosen maximal, we have $\jj \COVs{\TT'} \ii'$. Hence there are strictly more than~$\rr$ values of~$\ii$ satisfying $\jj \COVs{\TT'} \ii$, as expected.
\end{proof}

When a left-rotation transforms a tree~$\TT$ into a tree~$\TT$', the covering relation of~$\TT$ is included in that of~$\TT'$. The precise relation is as follows. Hereafter we use $\{0, 1\}^*$ (\resp.\ $\{1\}^*$) for the set of all addresses (\resp.\ all addresses of the form~$1^\ii$).

\begin{lemma}
\label{L:NewCovering}
Assume $\TT' = \TT \act \gah\alpha\rr$. Then $\COV{\TT'}$ is obtained by adding to~$\COV\TT$ the pairs~$(\jj, \ii)$ that satisfy
\begin{equation}
\label{E:NewCovering}
\exists \mm \in \{1 \wdots \rr\}\ (\add\TT\jj \in \alpha10^{\rr+1-\mm}\{1\}^*)\mbox{\quad and\quad} \add\TT\ii \in \alpha0\{0,1\}^* .
\end{equation}
\end{lemma}

\begin{proof}
Consider Figure~\ref{F:Solution} again. Let $\jj_1 \wdots \jj_\rr$ denote the last variable in the subtrees~$\TT_1 \wdots \TT_\rr$. A direct inspection shows that every covering pair in~$\TT$ is still a covering pair in~$\TT'$, and that the new covering pairs are the pairs~$(\jj_\mm, \ii)$ with $1 \leqslant \mm \leqslant \rr$ and $\ii$ occurring in~$\TT_0$: the action of~$\gah\alpha\rr$ is to let $\jj_1 \wdots \jj_\mm$ cover the variables of~$\TT_0$. Converted into addresses, this gives~\eqref{E:NewCovering}. 
\end{proof}

Lemma~\ref{L:NewCovering} is important for the Polish algorithm as it bounds possible coverings. 

\begin{lemma}
\label{L:Included}
Assume that $\TT, \TT'$ are equal size trees satisfying $\Pol\TT \lLex \Pol{\TT'}$. Then the covering relation of~$\TT \act \Sol\TT{\TT'}$ is included in the transitive closure of~$\COV\TT$ and~$\COV{\TT'}$.
\end{lemma}

\begin{proof}
Assume $\Sol\TT{\TT'} = \gah\alpha\rr$ and let $\TT_1 = \TT \act \gah\alpha\rr$. We use the notation of Figure~\ref{F:Solution} once more, calling~$\jj_\mm$ the rightmost variable occurring in~$\TT_\mm$ for $0 \leqslant \mm \leqslant \rr$. Let~$\II$ denote the set of all~$\ii$'s occurring in the subtree~$\TT_0$. By Lemma~\ref{L:NewCovering}, the pairs that belong to~$\COV{\TT_1}$ and not to~$\COV\TT$ are the pairs $(\jj_1, \ii) \wdots (\jj_\rr, \ii)$ with $\ii$ in~$\II$. The hypothesis that $\gah\alpha\rr$ is $\Sol\TT{\TT'}$ implies that the number of~$\OPP$ following~$\et_{\jj-1}$ in~$\Pol{\lab{\TT'}}$ is larger than its counterpart in~$\TT$, so $\jj_1$ must cover strictly more positions in~$\TT'$ than in~$\TT$. So, necessarily, $\jj_1 \COV{\TT'} \jj_0$ holds. On the other hand, $\jj_0 \COVe\TT \ii$ holds for every~$\ii$ in~$\II$, and $\jj_\mm \COVe\TT \jj_1$ holds for $1 \leqslant \mm \leqslant \rr$. It follows that, for all~$\mm$ in~$\{1 \wdots \rr\}$ and~$\ii$ in~$\II$, the pair~$(\jj_\mm, \ii)$ belongs to the the transitive closure of~$\COV\TT$ and~$\COV{\TT'}$. 
\end{proof}

\begin{lemma}
\label{L:TransClos}
Assume that $\TT, \TT'$ are equal size trees. Then the Polish algorithm running on~$(\TT, \TT')$ terminates with a pair~$(\TTp, \TTp)$ such that $\COV{\TTp}$ is the transitive closure of~$\COV\TT$ and~$\COV{\TT'}$.
\end{lemma}

\begin{proof}
Let $(\TT_\mt, \TT'_\mt)$ denote the pair of trees obtained after $\mt$~steps of the Polish algorithm running on~$(\TT, \TT')$, and $\NN$ be the total number of steps. By Lemma~\ref{L:NewCovering}, the relations~$\COV{\TT_\mt}$ make a non-decreasing sequence with respect to inclusion, and so do the relations~$\COV{\TT'_\mt}$. So, in particular, the transitive closure of~$\COV\TT$ and~$\COV{\TT'}$ is included in the transitive closure of~$\COV{\TT_\NN}$ and~$\COV{\TT'_\NN}$. Now, by hypothesis, the latter is~$\COV{\TTp}$.

On the other hand, Lemma~\ref{L:Included} shows that, for every~$\mt$, the relation~$\COV{\TT_{\mt+1}}$ is included in the transitive closure of~$\COV{\TT_\mt}$ and~$\COV{\TT'_\mt}$, and so is~$\COV{\TT'_{\mt+1}}$. Hence the transitive closure of~$\COV{\TT_{\mt+1}}$ and~$\COV{\TT'_{\mt+1}}$ is the transitive closure of~$\COV{\TT_\mt}$ and~$\COV{\TT'_\mt}$. Hence $\COV{\TTp}$, which is the transitive closure of~$\COV{\TT_\NN}$ and~$\COV{\TT'_\NN}$, is the transitive closure of~$\COV{\TT_0}$ and~$\COV{\TT'_0}$.
\end{proof}

We are ready to put pieces together and state the main results of this section.

\begin{proposition}
\label{P:LubPol}
For $\TT, \TT', \TT''$ equal size trees, the following are equivalent:

\ITEM1 The tree~$\TT''$ is the least upper bound of~$\TT$ and~$\TT'$ in the Tamari lattice;

\ITEM2 The Polish algorithm running on~$(\TT, \TT')$ returns~$(\TT'', \TT'')$;

\ITEM3 The covering relation of~$\TT''$ is the transitive closure of those of~$\TT$ and~$\TT'$.
\end{proposition}

\begin{proof}
Let $\TTm$ be the least upper bound of~$\TT$ and~$\TT'$ in the Tamari lattice, and $\TTp$ be the tree such that the Polish algorithm running on~$(\TT, \TT')$ returns~$(\TTp, \TTp)$. By Lemma~\ref{L:NewCovering}, $\COV\TT$ and~$\COV{\TT'}$ are included in~$\COV{\TTm}$. Hence the transitive closure of~$\COV\TT$ and~$\COV{\TT'}$, which by Lemma~\ref{L:TransClos} is~$\COV{\TTp}$, is included in~$\COV{\TTm}$. 

On the other hand, by definition, we have $\TT \leT \TTp$ and $\TT' \leT \TTp$, whence $\TTm \leT \TTp$. This implies that $\COV{\TTm}$ is included in~$\COV{\TTp}$. Hence $\COV{\TTm}$ and~$\COV{\TTp}$ coincide, and, therefore, $\TTm = \TTp$ holds. So \ITEM1 and~\ITEM2 are equivalent. 

Next, as said above, \ITEM2 implies~\ITEM3 by Lemma~\ref{L:NewCovering}. Conversely, if $\TT''$ is such that $\COV{\TT''}$ is the transitive closure of~$\COV\TT$ and~$\COV{\TT'}$, then $\COV{\TT''}$ coincides with~$\COV{\TTp}$ and, therefore, we must have $\TT'' = \TTp$. So, \ITEM2 and~\ITEM3 are equivalent.
\end{proof}

\begin{corollary}
\label{C:Positive}
For $\TT, \TT'$ equal size trees, the following are equivalent:

\ITEM1 We have $\TT \leT \TT'$ in the Tamari order;

\ITEM2 There exists~$\ff$ in~$\Fs$ such that $\TT' = \TT \act \ff$ holds;

\ITEM3 $\SOL\TT{\TT'}$ is a positive $\AAA$-word, that is, the Polish algorithm running on~$(\TT, \TT')$ finishes with~$(\TT', \TT')$.

\ITEM4 The relation~$\COV\TT$ is included in~$\COV{\TT'}$.
\end{corollary}

\begin{proof}
The equivalence of~\ITEM1 and~\ITEM2 has been established in Lemma~\ref{L:Compat}.

Next, it is obvious that \ITEM3 implies~\ITEM2 as every element~$\gah\alpha\rr$ belongs to~$\Fs$. Conversely, if $\TT \leT \TT'$ holds, then the least upper bound of~$\TT$ and~$\TT'$ is~$\TT'$. Hence, by Proposition~\ref{P:LubPol}, the Polish algorithm running on~$(\TT, \TT')$ finishes with~$(\TT', \TT')$. This means that the word~$\SOL\TT{\TT'}$ contains positive letters~$\gah\alpha\rr$ only. So \ITEM2 implies~\ITEM3.

Finally, as observed above, \ITEM3 is equivalent to saying that the Polish algorithm running on~$(\TT, \TT')$ finishes with~$(\TT', \TT')$, whereas \ITEM4 is equivalent to saying that $\COV{\TT'}$ is the transitive closure of~$\COV\TT$ and~$\COV{\TT'}$. By Proposition~\ref{P:LubPol}, the latter properties are equivalent and, therefore, \ITEM3 and~\ITEM4 are equivalent.
\end{proof}

It should be noted that the equivalence of~\ITEM1 and~\ITEM4 in Corollary~\ref{C:Positive} already appears as~\cite[Theorem 2.1]{Pal1}. 


\subsection{The Polish normal form}
\label{SS:NormalForm}

One of the interests of Proposition~\ref{P:LubPol} and Corollary~\ref{C:Positive} is that they provide unique distinguished decompositions for every element of~$F$ and of~$\Fs$ in terms of the generators~$\gah\alpha\rr$. Indeed, we obtained for every pair of equal size trees~$(\TT, \TT')$ a certain signed $\AAAh$-word~$\SOL\TT{\TT'}$ such that~$\TT \act \SOL\TT{\TT'}$ is defined and equal to~$\TT'$. This word~$\SOL\TT{\TT'}$ does not depend on~$\TT$.
 
\begin{lemma}
Assume that $\ff$ belongs to~$F$ and $\TT \act \ff$ is defined. Then the signed $\AAAh$-word~$\SOL\TT{\TT \act \ff}$ is an expression of~$\ff$, and it does not depend on~$\TT$.
\end{lemma}

\begin{proof}
First, we have $\TT \act \ff = \TT \act \SOL\TT{\TT'}$, so, by Proposition~\ref{P:Geometry}, the word~$\SOL\TT{\TT'}$ is an expression of~$\ff$. Next, assume that $\sigma$ is a substitution, and let us compare the Polish algorithm running on a pair~$(\TT, \TT')$ and on the pair~$(\TT^\sigma, \TT'{}^\sigma)$. The word~$\Pol{\TT^\sigma}$ is obtained from the word~$\Pol\TT$ by replacing every variable~$\et_\ii$ with the corresponding word~$\Pol{\sigma(\ii)}$. As the variables occur in the same order in the words~$\Pol\TT$ and~$\Pol{\TT'}$, substituting~$\et_\ii$ with~$\Pol{\sigma(\ii)}$ introduces no new clash. Therefore, if $(\TT_\mt, \TT'_\mt)$ are the trees at the $\mt$th step of the algorithm running on~$(\TT, \TT')$, then $(\TT_\mt^\sigma, \TT'_\mt{}^\sigma)$ are the trees at the $\mt$th step of the algorithm running on~$(\TT^\sigma, \TT'{}^\sigma)$, implying $\SOL\TT{\TT'} = \SOL{\TT^\sigma}{\TT'{}^\sigma}$.

By definition, for every~$\ff$ in~$F$, there exists a unique pair of trees~$(\trm\ff, \trp\ff)$ such that every pair~$(\TT, \TT \act \ff)$ can be expressed as~$((\lab{\trm\ff})^\sigma, (\lab{\trp\ff})^\sigma)$. The above result then shows that $\SOL\TT{\TT \act \ff}$ coincides with~$\SOL{\trm\ff}{\trp\ff}$, which only depends on~$\ff$. 
\end{proof}

\index{Polish!normal form}
\begin{definition}
\label{D:PolishNF}
For~$\ff$ in~$F$, the \emph{Polish normal form} of~$\ff$ is the signed $\AAAh$-word~$\SOL{\trm\ff}{\trp\ff}$.
\end{definition}

\begin{example}
Let $\ff = \ga\ea \ga1 = \ga{11} \ga\ea$ (= $\xx_1\xx_2 = \xx_3\xx_1$). Then we have $\trm\ff = \et (\et (\et (\et\et)))$ and $\trp\ff = (\et \et) ((\et \et) \et)$. Running the Polish algorithm on these trees returns the (positive) $\AAAh$-word~$\ga\ea \ga1$: so the latter is the Polish normal form of~$\ff$. By contrast, $\ga{11} \ga\ea$, which is another $\AAAh$-expression of~$\ff$, is not normal. One verifies similarly that the word~$\ga\ea^2$ is normal, whereas the equivalent words~$\ga1 \gah\ea2$ and~$\ga1 \ga\ea \ga0$ are not.
\end{example}

Corollary~\ref{C:Positive} immediately implies:

\begin{proposition}
\label{P:Positive}
An element of~$F$ belongs to the submonoid~$\Fs$ if and only if its Polish normal form contains no letter~$\gah\alpha\rr\inv$.
\end{proposition}

As the family of generators~$\AAAh$ is infinite, it makes no sense to wonder whether Polish normal words form a rational language or whether the Polish normal form can be connected with an automatic structure. However, let us observe that being Polish normal is a local property that can be characterized in terms of adjacent letters.

\begin{proposition}
A positive $\AAAh$-word $\gah{\alpha_1}{\rr_1} \pdots \gah{\alpha_\ell}{\rr_\ell}$ is Polish normal if and only if
$$\alpha_\mt 0^{\rr_\mt} \lLR \alpha_{\mt+1} 1 0^{\rr_{\mt+1}-1} 1$$
holds for every~$\mt < \ell$, where $\lLR$ denotes the left--right (partial) ordering of addresses.
\end{proposition}

\begin{proof}
Let $\ww = \gah{\alpha_1}{\rr_1} \pdots \gah{\alpha_\ell}{\rr_\ell}$ and assume that $\TT \act \ww$ is defined. For~$0 \leqslant \mt \leqslant \ell$, put $\TT_\mt = \TT \act \gah{\alpha_1}{\rr_1} \pdots \gah{\alpha_\mt}{\rr_\mt}$. Then $\ww$ is normal if, for every~$1 \leqslant \mt \leqslant \ell$, we have $\gah{\alpha_\mt}{\rr_\mt} = \Sol{\TT_{\mt-1}}{\TT_\ell}$, that is, $\gah{\alpha_\mt}{\rr_\mt}$  appears at the $\mt$th step of the Polish algorithm running on~$(\TT, \TT \act \ww)$. Now, as shown in Figure~\ref{F:Solution}, the origin in~$\TT_\mt$ of the letter~$\OPP$ involved in the clash between~$\TT_{\mt-1}$ and~$\TT_{\mt}$ is~$\alpha_\mt0^{\rr_\mt}$, whereas the origin in~$\TT_\mt$ of the letter~$\et$ involved in the clash between~$\TT_\mt$ and~$\TT_{\mt+1}$ lies in~$\alpha_{\mt+1}10^{\rr_{\mt+1}-1}1\{0\}^*$. The normality condition is then that, in~$\Pol{\TT_\mt}$, the former letter lies on the left of the latter. By construction of the Polish encoding, this happens if and only if the first address precedes the second in the ``left--right--root'' linear ordering of addresses. Due to the form of the second address, this is equivalent to $\alpha_\mt0^{\rr_\mt} \lLR \alpha_{\mt+1}10^{\rr_{\mt+1}-1}1$.
\end{proof}

For instance, the word $\ga\ea \ga\ea$ is normal, as we have $\ea 0^1 = 0 \lLR \ea 1 0^{1-1} 1 = 11$, but $\ga1 \gah\ea2$ is not, as we do not have $1 0^1 = 10 \lLR \ea 10^{2-1}1 = 101$. 


\section{Distance in Tamari lattices}
\label{S:Distance}

We conclude this description of the connections between the Tamari lattice and the Thompson group~$F$ with a few observations about distances in~$\Tam\nn$. The general principle is that it is easy to obtain upper bounds, but difficult to prove lower bounds and many questions remain open in this area. Our main observation here is that the embedding of the monoid~$\Fs$ into the group~$F$ is not an isometry, and not even a quasi-isometry (Definition~\ref{D:Quasi}): for every positive constant~$\CC$, there exist elements of~$\Fs$ whose length in~$F$ is smaller than their length in~$\Fs$ by a factor at least~$\CC$. In terms of Tamari lattices, this implies that chains are not geodesic (Corollary~\ref{C:Chains}). 

The plan of the section is as follows. In Subsection~\ref{SS:Diameter}, we quickly survey the known results about the diameter of Tamari lattices. Then, we show in Subsection~\ref{SS:Syntactic} how to use the syntactic relations of~$\RRR$ to obtain (rather weak) distance lower bounds. Finally, in Subsection~\ref{SS:Quasigeod}, we use the covering relation to establish (stronger) lower bounds. 


\subsection{The diameter of~$\Tam\nn$}
\label{SS:Diameter}

Surprisingly, the diameter of the Tamari lattice~$\Tam\nn$ is not known for every~$\nn$. 

\index{distance (between trees)}
\begin{definition}
For~$\TT, \TT'$ in~$\Tam\nn$, the \emph{distance} between~$\TT$ and~$\TT'$, denoted by~$\dist(\TT, \TT')$ is the minimal number of rotations needed to transform~$\TT$ into~$\TT'$. The \emph{diameter} of~$\Tam\nn$ is the maximum of~$\dist(\TT, \TT')$ for~$\TT, \TT'$ in~$\Tam\nn$.
\end{definition}

\begin{theorem}[Sleator, Tarjan, Thurston \cite{STT}]
For~$\nn \geqslant 11$, the diameter of~$\Tam\nn$ is at most~$2\nn-6$; for~$\nn$ large enough, it is exactly~$2\nn-6$.
\end{theorem}

The argument uses the fact that the maximal distance between two size~$\nn$ trees is also the maximal number of flips needed to transform two triangulations of an $(\nn+2)$-gon one into the other. A lower bound for the latter is obtained by putting the considered triangulations on the two halves of a sphere and bounding the hyperbolic volume of the resulting tiled polyhedron. It is conjectured that the value $2\nn-6$ is correct for every~$\nn \geqslant 11$. However, due to its geometric nature, the argument of~\cite{STT} works only for~$\nn \geqslant \nn_0$, with no estimation of~$\nn_0$.

By contrast, combinatorial arguments involving the covering relation of Subsection~\ref{SS:Covering} lead to (weaker) results that are valid for every~$\nn$.

\begin{theorem}\cite{Dhw}
For $\nn = 2\pp^2$, the diameter of~$\Tam\nn$ is at least $2\nn - 2\sqrt{2\nn} + 1$ and, for every~$\nn$, it is at least $2\nn - \sqrt{70\nn}$.
\end{theorem}

Although no theoretical obstruction seems to exist, the covering arguments have not yet been developed enough to lead to an exact value of the diameter. However some candidates for realizing the maximal distance are known.

\begin{conjecture}\cite{Dhw}
\label{C:Distance}
For~$\alpha$ an address, let~$\Sp\alpha$ denote the tree recursively specified by the rules $\Sp{\ea} = \et$, $\Sp{0\alpha} = \Sp\alpha \OP \et$, and~$\Sp{1\alpha} = \et \OP \Sp\alpha$. Define
$$\ZZ_\nn = 
\begin{cases}
\Sp{111(01)^{\pp-2}}\\
\Sp{111(01)^{\pp-2}0}
\end{cases}
\ZZ'_\nn = 
\begin{cases}
\Sp{000(10)^{\pp-2}}
&\mbox{for $\nn = 2\pp+3$},\\
\Sp{000(10)^{\pp-2}1}
&\mbox{for $\nn = 2\pp+4$},
\end{cases}
$$
see Figure~\ref{F:Conjecture}. Then one has $\dist(\ZZ_\nn, \ZZ'_\nn) = 2\nn-6$ for $\nn \geqslant 9$.
\end{conjecture}

\begin{figure}[htb]
\begin{picture}(45,20)(0,2)
\put(2,0){\includegraphics{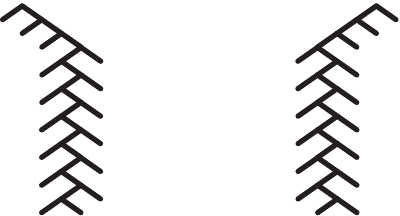}}
\put(0,10){$\ZZ_\nn$}
\put(41,10){$\ZZ'_\nn$}
\end{picture}
\caption{\sf\smaller The zigzag trees of Conjecture~\ref{C:Distance}, here for $\nn = 15$; the distance is~$24$, as predicted.}
\label{F:Conjecture}
\end{figure}

Conjecture~\ref{C:Distance} has been checked up to size~$19$ (sizes below~$9$ are special, because the trees are then too small for the generic scheme to start; by the way, the value $2\nn-6$ is valid for $\nn = 5, 6, 7$, but not for $\nn \le 4$ and for $\nn = 8$).


\subsection{Syntactic invariants}
\label{SS:Syntactic}

A natural way to investigate distances in the Tamari lattices is to use the action of~$F$ on trees and to study the length of the elements of~$F$ with respect to the generating family~$\AAA$. Indeed, Proposition~\ref{P:Geometry} directly implies 

\begin{lemma}
For all trees~$\TT, \TT'$, we have $\dist(\TT, \TT') = \LGA{\SOL\TT{\TT'}}$, where $\LGA\ff$ is the $\AAA$-length of~$\ff$, that is, the length of the shortest signed $\AAA$-word representing~$\ff$.
\end{lemma}

In order to establish (lower) bounds on~$\LGA\ff$, a natural approach is to use the syntatic properties of the relations of~$\RRR$.

\begin{lemma}
\label{L:Invariant}
For~$\ww$ a signed $\AAA$-word, denote by~$\Lgu\ww$ the number of letters~$\ga{1^\ii}^{\pm1}$ in~$\ww$.

\ITEM1 If $\uu, \uu'$ are $\RRR$-equivalent positive $\AAA$-words, then $\Lgu\uu = \Lgu{\uu'}$ holds.

\ITEM2 If $\ww, \ww'$ are signed $\AAA$-words, then $\ww \rev_{\RRR} \ww'$ implies $\Lgu\ww \geqslant\nobreak \Lgu{\ww'}$ holds.
\end{lemma}

\begin{proof}
In both cases, it suffices to inspect the relations of~$\RRR$. In the case of the pentagon relations, we have $\Lgu{\ga\alpha^2} = \Lgu{\ga{\alpha0} \ga\alpha \ga{\alpha1}}$, both being~$2$ for $\alpha$ in~$\{1\}^*$, and $0$ otherwise. Similarly, for~\ITEM2, we find $\Lgu{\ga\alpha\inv \ga{\alpha1}} = \Lgu{\ga\alpha \ga{\alpha0}\inv \ga\alpha\inv}$, both being~$2$ for $\alpha$ in~$\{1\}^*$, and $0$ otherwise. The inequality comes from $\Lgu{\ga\alpha\inv \ga\alpha} = 2 > 0 = \Lgu{\eps}$ for $\alpha$ in~$\{1\}^*$.
\end{proof}

Note that the counterpart of Lemma~\ref{L:Invariant}\ITEM2 involving left-reversing is false: $\ga\ea \ga0\inv$ is left-$\RRR$-reversible to~$\ga\ea\inv \ga1 \ga\ea$, and we have $\Lgu{\ga\ea \ga0\inv} = 1 < 3 = \Lgu{\ga\ea\inv \ga1 \ga\ea}$.

\begin{proposition}
\label{P:Invariant}
For every~$\ff$ of~$F$ and every $\AAA$-word~$\ww$ representing~$\ff$, we have
\begin{equation}
\LGA\ff \geqslant \Lgu{\DLR(\ww)} + \Lgu{\NLR(\ww)}.
\end{equation}
\end{proposition}

\begin{proof}
Put $\ell = \Lgu{\DLR(\ww)} + \Lgu{\NLR(\ww)}$. By definition, the word~$\ww$ is right-$\RRR$-reversible to the word~$\NR(\ww) \DR(\ww)\inv$, so, by Lemma~\ref{L:Invariant}\ITEM2, we have 
$$\Lgu\ww \geqslant \Lgu{\NR(\ww) \DR(\ww)\inv} = \Lgu{\NR(\ww)} + \Lgu{\DR(\ww)}.$$
Next, it follows from Proposition~\ref{P:MinFrac} that there exist a positive $\AAA$-word~$\uu$ satisfying $\NR(\ww) \equivp_{\RRR} \NLR(\ww) \, \uu$ and $\DR(\ww) \equivp_{\RRR} \DLR(\ww) \, \uu$. Then, by Lemma~\ref{L:Invariant}\ITEM1, we deduce $\Lgu{\NR(\ww)} + \Lgu{\DR(\ww)} \geqslant \ell$, whence $\Lg\ww \geqslant \Lgu\ww \geqslant \ell$.
 
Now assume $\ww' \equiv_{\RRR} \ww$. By the result above, we have $\Lg{\ww'} \geqslant \Lgu{\NLR(\ww')} + \Lgu{\DLR(\ww')}$. By Proposition~\ref{P:MinFrac}, we have $\NLR(\ww') \equivp_{\RRR} \NLR(\ww)$ and $\DLR(\ww') \equivp_{\RRR} \DLR(\ww)$, whence, by Lemma~\ref{L:Invariant}\ITEM1, $\Lgu{\NLR(\ww')} + \Lgu{\DLR(\ww')} = \ell$. Thus $\Lg{\ww'} \geqslant \ell$ holds for every word~$\ww'$ that represents~$\cl\ww$. By definition, this means that $\LGA\ff \geqslant \ell$ is true.
\end{proof}

Of course, a symmetric criterion involving~$\{0\}^*$ instead of~$\{1\}^*$ may be stated.

\begin{example}
Let $\ff = \ga1 \ga{11}\inv \ga\alpha \ga{11}\inv$. Left-reversing the word $\ga1 \ga{11}\inv \ga\alpha \ga{11}\inv$ yields the word $\ga{111}\inv \ga{1111}\inv \ga1 \ga\ea$, which in turn is right-reversible to $\ga1 \ga\ea \ga1\inv \ga{11}\inv$. We conclude that $\NLR(\ww)$ is $\ga1 \ga\ea$ and $\DLR(\ww)$ is $\ga{11} \ga1$. Then Proposition~\ref{P:Invariant} gives $\LGA\ff \geqslant 4$, that is, the word~$\ga1 \ga{11}\inv \ga\alpha \ga{11}\inv$ is geodesic.
\end{example}

By construction, the elements~$\Blue\TT$ involved in proof of Proposition~\ref{P:Connection} are represented by $\AAA$-words all letters of which are of the form~$\ga\alpha$ with $\alpha$ in~$\{1\}^*$, and, therefore, these words are geodesic. However, elements of this type are quite special, and the criterion of Proposition~\ref{P:Invariant} is rarely useful. In particular, it follows from the construction that every element of~$F$ can be represented by a word of the form~$\Blue\TT\inv \Blue{\TT'}$ but, even when the fraction is irreducible, that is, when the elements represented by~$\Blue\TT$ and~$\Blue{\TT'}$ admit no common left-divisor in~$\Fs$, it need not be geodesic, as shows the example of $\ga\ea^{-\pp} \ga{1^\pp} \ga{1^{\pp-1}} \pdots \ga1$, an irreducible fraction of length~$2\pp$ which is $\RRR$-equivalent to the positive--negative word $\ga\ea \ga0^{-\pp} \ga\ea\inv$ of length~$\pp+2$.


\subsection{The embedding of~$\Fs$ into~$F$}
\label{SS:Quasigeod}

More powerful results can be obtained using the covering relation of Subsection~\ref{SS:Covering}. As an example,  we shall now establish that the embedding of the monoid~$\Fs$ into the group~$F$ provided by Proposition~\ref{P:DualFraction} is not an isometry, that is, there exist elements of~$\Fs$ whose length as elements of~$F$ is smaller than their length as elements of~$\Fs$. This result is slightly surprising: clearly, fractions need not be geodesic in general, but we might expect that, when an element of~$F$ belongs to~$\Fs$, then its length inside~$\Fs$ equals its length inside~$F$. 

\index{quasi-isometry}
\begin{definition}
\label{D:Quasi}
If $(X, d), (X', d')$ are metric spaces, a map~$\ff : X \to X'$ is a \emph{quasi-isometry} if there exist $\CC \geqslant 1$ and $\CC' \geqslant 0$ such that $\frac1\CC d(\ff(x, y)) - \CC' \leqslant d'(\ff(x, y)) \leqslant \CC d(x, y) + \CC'$ holds for all~$x, y$ in~$X$.
\end{definition}

The result we shall prove is as follows.

\begin{proposition}
\label{P:NotQuasi}
For~$\ff$ in~$\Fs$, let $\LGAp\ff$ denotes the $\AAA$-length of~$\ff$ in~$\Fs$, that is, the length of a shortest positive $\AAA$-word representing~$\ff$. Then the embedding of~$\Fs$ into~$F$ is not a quasi-isometry of~$(\Fs, \LGAp{})$ into~$(F, \LGA{})$.
\end{proposition}

In order to establish Proposition~\ref{P:NotQuasi}, it is enough to exhibit a sequence of elements~$\ff_\pp$ of~$\Fs$ satisfying $\LGA{\ff_\pp} = o(\LGAp{\ff_\pp})$. This is what the next result provides.

\begin{lemma}
\label{L:Length}
For every $\pp \geqslant 1$, let $\uu_\pp$ be the $\AAAh$-word
$$\gah{(10)^\pp}1 \, \gah{(10)^{\pp-1}}2 \pdots \gah{10}\pp \, \gah\ea{\pp+1}.$$
Then, for every~$\pp$, we have $\LGA{\cl{\uu_\pp}} \leqslant 3\pp+1$ and $\LGAp{\cl{\uu_\pp}} = (\pp+1)(\pp+2)/2$.
\end{lemma}

Establishing Lemma~\ref{L:Length} requires to prove two inequalities, namely an upper bound on~$\LGA{\cl{\uu_\pp}}$ and a lower bound on~$\LGAp{\cl{\uu_\pp}}$. As always, the first task is easier than the second.

\begin{lemma}
\label{L:Upper}
For every~$\pp \geqslant 1$, we have $\LGA{\cl{\uu_\pp}} \leqslant 3\pp+1$.
\end{lemma}

\begin{proof}
For~$\pp \geqslant 1$, let $\ww_\pp = \ga{(10)^\pp} \, \ga{(10)^{\pp-1}1}\inv \, \ga{(10)^{\pp-1}} \, \ga{(10)^{\pp-2}1}\inv \, \ga{(10)^{\pp-2}} \pdots \ga{101}\inv \, \ga{10} \, \ga1\inv \ga\ea$. Then $\ww_\pp$ is a signed $\AAA$-word of length~$2\pp+1$. An easy induction using the formulas of~\eqref{E:Closure} shows that $\ww_\pp$ is right-reversible to the positive--negative word~$\uu_\pp \gah\ea\pp\inv$, see Figure~\ref{F:Upper}. As the latter word has $\AAA$-length~$3\pp+1$, the result follows.
\end{proof}

\begin{figure}[htb]
\begin{picture}(72,32)(3,0)
\put(-45,-47){\includegraphics{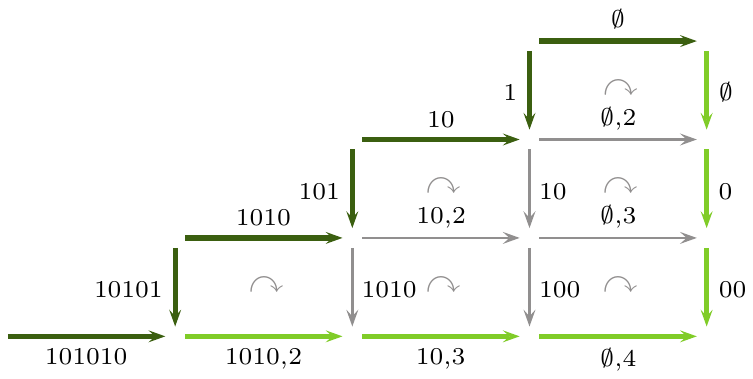}}
\end{picture}
\caption{\sf\smaller Right-reversing the signed $\AAA$-word~$\ww_\pp$, here with $\pp = 3$. The numerator is the word~$\uu_\pp$ of Lemma~\ref{L:Length}, whereas the denominator is the length~$\pp$ word~$\gah\ea\pp$ (as usual, the letters~``$a$'' have been skipped).}
\label{F:Upper}
\end{figure}

In order to complete the proof of Lemma~\ref{L:Length}, it remains to prove that $\LGAp{\cl{\uu_\pp}}$ is $\pp(\pp+3)/2$. As the length of~$\uu_\pp$ is~$\pp(\pp+3)/2$, the point is to prove:

\begin{lemma}
\label{L:Lower}
For every~$\pp \geqslant 1$, the word~$\uu_\pp$ is geodesic in~$\Fs$.
\end{lemma}

\begin{proof}
With the notation of Conjecture~\ref{C:Distance}, let $\TT_\pp$ be the size~$2\pp+2$ tree $\Sp{(10)^\pp1}$. Then $\TT_\pp \act \uu_\pp$ is defined and equal to~$\TT'_\pp = \Sp{0^{\pp+1}1^\pp}$, see Figure~\ref{F:Lower}. 

We look at the covering relations that are satisfied in~$\TT'_\pp$ but not in~$\TT_\pp$. First, $2\pp+1$ covers~$0$ in~$\TT'_\pp$ but not in~$\TT_\pp$. We deduce that every $\AAA$-word transforming~$\TT_\pp$ to~$\TT'_\pp$ contains at least one step with critical index~$2\pp+1$, the critical index of~$\ga\alpha$ being defined as the unique~$\jj$ such that applying~$\ga\alpha$ adds at least one relation~$\jj \COV{} \ii$, that is, the unique~$\jj$ such that $\add\TT\jj$ lies in~$\alpha10\{1\}^*$.

Now, here is the point: $2\pp$ covers~$0$ and~$1$ in~$\TT'_\pp$, but not in~$\TT_\pp$. We deduce that every $\AAA$-word transforming~$\TT_\pp$ to~$\TT'_\pp$ contains at least one step with critical index~$2\pp$. But we claim that every such word must actually contain at least \emph{two} such steps, that is, one step cannot be responsible for the two new covering relations. Indeed, $2\pp$ covers~$2$ in~$\TT_\pp$, but does not cover~$1$. The only situation when a step adding $2\pp \COV{} 1$ in a tree~$\TT$ can simultaneously add~$2\pp \COV{} 0$ is when $1$ covers~$0$ in~$\TT$. But this cannot happen here, because we are considering positive $\AAA$-words only, so any possible covering satisfied at an intermediate step must remain in the final tree~$\TT'_\pp$. As $1$ does not cover~$0$ in~$\TT'_\pp$, it is impossible that $1$ covers~$0$ at any intermediate step. Thus two steps are needed to ensure $2\pp \COV{} 1$ and $2\pp \COV{} 0$.

The sequel is similar. In~$\TT'_\pp$, the number~$2\pp-1$ covers~$0, 1$, and~$2$, whereas, in~$\TT_\pp$, it covers only~$3$. Then at least three steps are needed to ensure $2\pp-1 \COV{} 2$, $2\pp-1 \COV{} 1$, and $2\pp-1 \COV{} 0$. Indeed, a step can cause $2\pp-1$ to simultaneously cover~$1$ and~$2$ only if $2$ covers~$1$ at the involved step, which cannot happen as $2$ does not cover~$1$ in~$\TT'_\pp$; the argument is then the same for~$1$ and~$0$ once we know that $2$ does not cover~$1$.

Similarly, $4$~steps are needed to force $2\pp-2$ to cover~$3$ to~$0$, then $5$~steps to force $2\pp-3$ to cover~$4$ to~$0$, etc., and $\pp+1$~steps to force $\pp+1$ to cover~$\pp$ to~$0$. We conclude that $1 + 2 + \cdots + (\pp+1)$, that is, $(\pp+1)(\pp+2)/2$, steps at least are needed to go from~$\TT_\pp$ to~$\TT'_\pp$. Hence $\uu_\pp$ is geodesic among positive $\AAA$-words.
\end{proof}

\begin{figure}[htb]
\begin{picture}(75,28)(0,0)
\put(-43,-45){\includegraphics{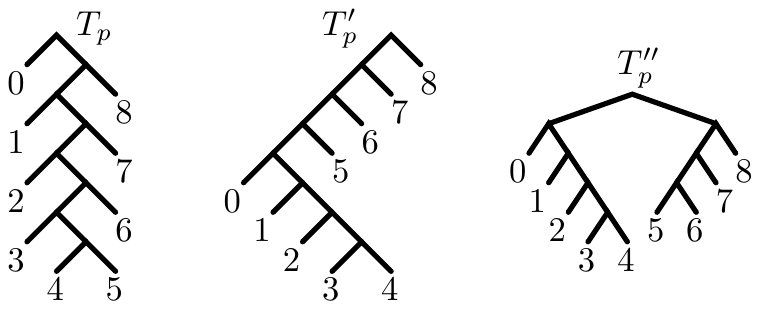}}
\end{picture}
\caption{\sf\smaller The trees of Lemma~\ref{L:Lower} with the leaf labelling involved in the covering relation, here in the case $\pp = 3$. The tree~$\TT''_\pp$ is the intermediate tree~$\TT_\pp \act \ww_\pp$ obtained after $2\pp+1$~steps: applying~$\gah\ea\pp$ to~$\TT''_\pp$ leads to~$\TT'_\pp$ in a total of $3\pp+1$~rotations.}
\label{F:Lower}
\end{figure}

Thus the proof of Proposition~\ref{P:NotQuasi} is complete. 

Finally, we translate the above arguments into the language of Tamari lattices.

\index{positive!distance}
\begin{definition}
For~$\TT, \TT'$ in~$\Tam\nn$ satisfying $\TT \leT \TT'$, the \emph{positive distance}~$\distp(\TT, \TT')$ from~$\TT$ to~$\TT'$ is the minimal number of left-rotations needed to transform~$\TT$ into~$\TT'$.
\end{definition}

\begin{proposition}
\label{P:Chains}
For every even~$\nn$, there exist $\TT, \TT'$ in~$\Tam\nn$ satisfying 
\begin{equation}
\label{E:Ratio}
\dist(\TT, \TT') \leqslant \frac{12}{\nn} \, \distp(\TT, \TT').
\end{equation}
\end{proposition}

\begin{proof}
Write $\nn = 2\pp+2$ and let $\TT_\pp, \TT'_\pp$ be the size~$\nn$ trees of the proof of Lemma~\ref{L:Lower}. Then we have $\dist(\TT_\pp, \TT'_\pp) \leqslant \nobreak 3\pp +\nobreak 1$ and $\distp(\TT_\pp, \TT'_\pp) = (\pp+1)(\pp+2)/2$, whence $\dist(\TT_\pp, \TT'_\pp) \leqslant \frac{12\nn-16}{\nn(\nn+2)}\, \distp(\TT_\pp, \TT'_\pp)$ in term of~$\nn$, and, a fortiori,~\eqref{E:Ratio}.
\end{proof}

\begin{corollary}
\label{C:Chains}
Chains are not geodesic in Tamari lattices; more precisely, for every~$\nn$, there exists a length~$\ell$ chain of~$\Tam\nn$ whose endpoints are at distance less than~$12\ell/\nn$.
\end{corollary}

We conclude with a few open questions.

\begin{question}
For~$\ww$ a positive $\AAA$-word, is the number~$\wit(\ww)$ of~\eqref{E:Wit} a least upper bound for the length of the words that are $\RRR$-equivalent to~$\ww$?
\end{question}

The answer is positive in the case of~$\ga\ea^\nn$. Indeed, we have $\lambda(\ga\ea^\nn) = \nn(\nn+1)/2$ and $\ga\ea^\nn \equivp_{\RRR} \gah{1^{\nn-1}}1 \, \gah{1^{\nn-2}}2 \pdots \gah1{\nn-1} \, \gah\ea\nn$. The general case is not known.

\begin{question}
Does the Polish normal form of Definition~\ref{D:PolishNF} satisfy some Fellow Traveler Property, that is, is the distance between the paths associated with the normal forms of elements~$\ff$ and~$\ff \ga\alpha^{\pm1}$ uniformly bounded?
\end{question}

A positive solution would provide a sort of infinitary automatic structure on~$F$. The question should be connected with a possible closure of Polish normal words under left- or right-$\RRR$-reversing

Finally, using mapping class groups and cell decompositions, D.\,Krammer constructed for every size~$\nn$ tree~$\TT$ an exotic lattice structure on~$\Tam\nn$ in which $\TT$ is the bottom element \cite{Kra}.

\begin{question}
Can the Krammer lattices be associated with submonoids of the Thompson group~$F$?
\end{question}

More generally, a natural combinatorial description of the Krammer lattices is still missing, but would be highly desirable. Connections with permutations and braids in the line of~\cite{BjW} can be expected. 




\begin{thebibliography}{10}

\bibitem{BaP}
{\relax J.L}.~Baril and {\relax J.M.}.~Pallo,  ``{Efficient lower and upper
  bounds of the diagonal-flip distance between triangulations}'', \emph{Inform.
  Proc. Letters} \textbf{100} (2006) 131--136.

\bibitem{BjW}
{\relax A}.~Bj\"orner and {\relax M}.~Wachs,  ``{Shellable nonpure complexes
  and posets. II}'', \emph{Trans. Amer. Math. Soc.} \textbf{349} (1997)
  3945--3975.

\bibitem{BrS}
{\relax M}.~Brin and {\relax C}.~Squier,  ``{Groups of piecewise linear
  homeomorphisms of the real line}'', \emph{Invent. Math.} \textbf{79} (1985)
  485--498.

\bibitem{CFP}
{\relax J.W}.~Cannon, {\relax W.J}.~Floyd, and {\relax W.R}.~Parry,
  ``{Introductory notes on Richard Thompson's groups}'', \emph{Enseign. Math.}
  \textbf{42} (1996) 215--257.

\bibitem{ClP}
A.~Clifford and G.~Preston, \emph{The algebraic theory of semigroups, volume
  1}, Amer. Math. Soc. Surveys, vol.~7, Amer. Math. Soc., 1961.

\bibitem{Dfg}
P.~Dehornoy,  ``The structure group for the associativity identity'',
  \emph{J.~Pure Appl. Algebra} \textbf{111} (1996) 59--82.

\bibitem{Dff}
\bysame,  ``Groups with a complemented presentation'', \emph{J.~Pure Appl.
  Algebra} \textbf{116} (1997) 115--137.

\bibitem{Dgd}
\bysame, \emph{Braids and Self-Distributivity}, Progress in Math., vol. 192,
  Birkh{\"a}user, 2000.

\bibitem{Dgj}
\bysame,  ``Study of an identity'', \emph{Algebra Universalis} \textbf{48}
  (2002) 223--248.

\bibitem{Dgp}
\bysame,  ``Complete positive group presentations'', \emph{J. of Algebra}
  \textbf{268} (2003) 156--197.

\bibitem{Dhb}
\bysame,  ``{Geometric presentations of Thompson's groups}'', \emph{J. Pure
  Appl. Algebra} \textbf{203} (2005) 1--44.

\bibitem{Dhw}
\bysame,  ``{On the rotation distance between binary trees}'', \emph{Advances
  in Math.} \textbf{223} (2010) 1316--1355.

\bibitem{Dia}
\bysame,  ``{The word reversing method}'', \emph{Intern. J. Alg. and Comput.}
  \textbf{21} (2011) 71--118.

\bibitem{Garside}
{\relax P}.~Dehornoy, {\relax with F}.~Digne, {\relax E}.~Godelle, {\relax
  D}.~Krammer, and {\relax J}.~Michel,  ``{Garside Theory}'', Book in progress,
  http://www.math.unicaen.fr/$\sim$garside/Garside.pdf.

\bibitem{Dei}
{\relax O}.~Deiser,  ``{Notes on the Polish Algorithm}'',
  http://page.mi.fu-berlin.de/deiser/wwwpublic/psfiles/polish.ps.

\bibitem{FrT}
{\relax H}.~Friedman and {\relax D}.~Tamari,  ``{Probl\`emes d'associativit\'e
  : Une structure de treillis finis induite par une loi demi-associative}'',
  \emph{J. Combinat. Th.} \textbf{2} (1967) 215--242.

\bibitem{Gub}
{\relax V.S}.~Guba,  ``{The Dehn function of Richard Thompson's group $F$ is
  quadratic}'', \emph{Invent. Math.} \textbf{163} (2006) 313--342.

\bibitem{HuT}
{\relax S}.~Huang and {\relax D}.~Tamari,  ``{Problems of associativity: A
  simple proof for the lattice property of systems ordered by a
  semi-associative law}'', \emph{J. Combinat. Th. Series A} \textbf{13} (1972)
  7--13.

\bibitem{Kra}
D.~Krammer,  ``{A class of Garside groupoid structures on the pure braid
  group}'', \emph{Trans. Amer. Math. Soc.} \textbf{360} (2008) 4029--4061.

\bibitem{Mac}
{\relax S}.~{Mac Lane}, \emph{{Natural associativity and commutativity}}, Rice
  University Studies, vol.~49, 1963.

\bibitem{McT}
{\relax R}.~McKenzie and {\relax R.J}.~Thompson,  ``{An elementary construction
  of unsolvable word problems in group theory}'', in \emph{Word Problems},
  Boone and {\it al}, eds., Studies in Logic, vol.~71, North Holland, 1973,
  457--478.

\bibitem{Pal0}
{\relax J.M}.~Pallo,  ``{Enumerating, ranking and unranking binary trees}'',
  \emph{The Computer Journal} \textbf{29} (1986) 171--175.

\bibitem{Pal1}
\bysame,  ``{An algorithm to compute the M\"obius function of the rotation
  lattice of binary trees}'', \emph{RAIRO Inform. Th\'eor. Applic.} \textbf{27}
  (1993) 341--348.

\bibitem{STT}
{\relax D}.~Sleator, {\relax R}.~Tarjan, and {\relax W}.~Thurston,  ``{Rotation
  distance, triangulations, and hyperbolic geometry}'', \emph{J. Amer. Math.
  Soc.} \textbf{1} (1988) 647--681.

\bibitem{Stn}
{\relax R}.~Stanley, \emph{{Enumerative Combinatorics, vol. 2}}, Cambridge
  Studies in Advances Math., no.~62, Cambridge Univ. Press, 2001.

\bibitem{Sta}
{\relax J.D}.~Stasheff,  ``{Homotopy associativity of $H$-spaces}'',
  \emph{Trans. Amer. Math. Soc.} \textbf{108} (1963) 275--292.

\bibitem{Sun}
{\relax Z}.~{\v Suni\'c},  ``{Tamari lattices, forests and Thompson monoids}'',
  \emph{Europ. J. Combinat.} \textbf{28} (2007) 1216--1238.

\bibitem{Tam}
{\relax D}.~Tamari,  ``{The algebra of bracketings and their enumeration}'',
  \emph{Nieuw Archief voor Wiskunde} \textbf{10} (1962) 131--146.

\bibitem{Tho}
{\relax R.J}.~Thompson,  ``{Embeddings into finitely generated simple groups
  which preserve the word problem}'', in \emph{Word Problems II}, Adian, Boone,
  and Higman, eds., Studies in Logic, North Holland, 1980, 401--441.
\end{thebibliography}
\end{document}